\documentclass[11pt]{article}
\usepackage{amsmath,amssymb,mathtools}
\usepackage[margin=1in]{geometry}
\usepackage{amsthm}
\usepackage{tikz}
\usetikzlibrary{arrows.meta}
\usetikzlibrary{positioning}

\theoremstyle{definition}

\theoremstyle{remark}
\usepackage{tikz}
\usepackage{amssymb, graphicx, caption, verbatim,  amsfonts}
\usepackage{longtable}
\usepackage{hyperref}
\usepackage{cleveref}
\usepackage{wasysym}
\newcommand{\De}{\Delta}
\newcommand{\g}{\gamma}

\newtheorem{theorem}{Theorem}
\newtheorem{lemma}[theorem]{Lemma}

\newcommand{\E}{\ensuremath{\mathbb E}}
\newcommand{\R}{\ensuremath{\mathbb R}}

\newcommand{\lab}{\label}  \newcommand{\ra}{\ensuremath{\rightarrow}}   \def\de{{\mathbf{\delta}}} \def\De{{{\Delta}}}  
 \def\var{{\mathrm{var}}} \def\beq{\begin{eqnarray}} \def\eeq{\end{eqnarray}} \def\ben{\begin{enumerate}}
\def\een{\end{enumerate}} \def\bit{\begin{itemize}}
\def\bel{\begin{lemma}}
\def\eel{\end{lemma}}
\def\eit{\end{itemize}} \def\beqs{\begin{eqnarray*}} \def\eeqs{\end{eqnarray*}} \def\bel{\begin{lemma}} \def\eel{\end{lemma}}
 \newcommand{\Z}{\mathbb{Z}}   
\newcommand{\T}{\mathbb{T}}       

 \newcommand{\I}{I}   \newcommand{\p}{\mathbb{P}}
    \newcommand{\one}{\mathrm{1}}
  \newcommand{\la}{\lambda} \newcommand{\Tr}{\text{Trace}} 
  \def\eps{{\epsilon}}   \def\g{G}

\renewcommand{\g}{\mathcal{G}}

\newcommand{\ent}{\mathrm{ent}}
\newcommand{\hess}{{\mathop{\triangledown^2}}}
\renewcommand{\hess}{{\nabla^2}}

\newtheorem{lemma}{Lemma}
\newtheorem{proposition}{Proposition}

\newtheorem{definition}{Definition}
\newtheorem{notation}{Notation}
\newtheorem{remark}{Remark}
\newtheorem{claim}{Claim}

\numberwithin{equation}{section}
\numberwithin{figure}{section}

\renewcommand{\g}{\gamma}
\newcommand{\diag}{\mathrm{diag}}

\newcommand{\spec}{\mathrm{spec}}
\newcommand{\Spec}{\mathrm{Spec}}
\newcommand{\Leb}{\mathrm{Leb}}

\newcommand{\tilh}{\tilde{h}}

\renewcommand{\Tr}{\mathbf{tr}}

\newcommand{\rel}{\to}

\newcommand{\HIVE}{\mathtt{HIVE}}
\newcommand{\AHIVE}{\mathtt{AUGHIVE}}

\newcommand{\oct}{\mathbf{oct}}
\newcommand{\GT}{\mathtt{GT}}

\newcommand{\Var}{\operatorname{var}}

\renewcommand{\ent}{\operatorname{ent}}
\newcommand{\diamondplus}{|\mathtt{tet}|}

\renewcommand{\T}{\mathbb T}
\newcommand{\volBH}{\hat{\operatorname{vol}}_{BH}}
\numberwithin{equation}{section}
\newcommand{\Herm}{\mathrm{Herm}}

\begin{document}
\title{Hives from deformed GUE minor processes}

\author{Hariharan Narayanan\\
School of Technology and Computer Science, TIFR Mumbai\\
\texttt{hariharan.narayanan@tifr.res.in}}
\maketitle 

\begin{abstract}
We construct random hives from deformed GUE minor processes.  Starting from
two independent diagonally deformed GUE matrices
\[
X=\sqrt{n}(wG+uD),\qquad Y= \sqrt{n}(w'G'+u'D'),
\]
where \(D,D'\) are diagonal and have GUE spectra, we use their minor processes to form a
double hive and then apply the octahedron recurrence.  Under the matching
condition
\[
\frac{u}{w^2}=\frac{u'}{(w')^2},
\]
we prove that the resulting hive law is close, in relative entropy, to a GUE
hive law.  More precisely, if
\[
a^2=w^2+u^2,\qquad b^2=(w')^2+(u')^2,
\]
then the produced hive density \(q_n\) satisfies
\[
D_{\mathrm{KL}}\!\left(
q_n\,
\middle\|\,
\operatorname{Density}\bigl(\mathcal H_n(a\sqrt n,b\sqrt n,c_{**}\sqrt n)\bigr)
\right)
=
O(n\log n).
\]
Writing \(\delta=u+u'\), the third scale is
\[
c_{**}^2=w^2+(w')^2+\delta^2=a^2+b^2+2uu'.
\]
It is also the limit of a tetrahedral optimization problem and is the unique
solution with \(a^2+b^2\leq c_{**}^2<(a+b)^2\) of
\[
\delta^2
=
\frac{
2c_{**}^4(c_{**}^2-a^2-b^2)
}{
(c_{**}^2-a^2+b^2)(c_{**}^2+a^2-b^2)
}.
\]
Thus the construction realizes GUE hive laws, up to subleading relative
entropy, throughout the right-angled and obtuse regime.  The introduction records
two explicit surface-tension approximations and numerical comparisons which
motivated the construction.
\end{abstract}
\newpage
\tableofcontents
\section{Introduction}

Hives were introduced by Knutson and Tao in their proof of the saturation
conjecture and have since become a central object in the study of the Horn
problem, Littlewood--Richardson coefficients, and spectra of sums of
Hermitian matrices \cite{KT1,KT2,Horn,Kly,FultonYoungTableaux}.  A hive is a
discrete concave function on a triangular
lattice satisfying three families of rhombus inequalities.  When the three
boundary increments are fixed, the set of all such hives forms a convex
polytope.  Its volume is closely related to the density of the eigenvalue
distribution of a sum of two Hermitian matrices with prescribed spectra.
More precisely, a formula of Knutson and Tao \cite[Theorem~4]{KT2}
expresses the density of $\operatorname{spec}(X+Y)$, for independent Haar
conjugates $X=U\operatorname{diag}(\lambda)U^*$ and
$Y=V\operatorname{diag}(\mu)V^*$, in terms of Vandermonde determinants and
the volume of the hive polytope with boundary data $(\lambda,\mu;\nu)$.
Coquereaux and Zuber stated the result in a more explicit form
\cite{CZ,Zuberhorn}.

In the large $n$ limit, these hive volumes are governed by a variational
principle.  The local contribution to the entropy is encoded by a convex
surface tension function
\[
    \sigma : \mathbb R_+^3 \to \mathbb R,
\]
defined as the negative logarithm of the limiting normalized volume of a
periodic polytope of rhombus-concave functions with prescribed average
Hessian.  If $h$ is a surface tension minimizing continuum hive with boundary
data determined by limiting spectral profiles $\lambda,\mu,\nu$, then the
leading-order entropy of the corresponding hive polytope is described by the
integral
\[
    2\int_T \sigma\bigl((-1)(\hess h)_{\mathrm{ac}}\bigr)\,dx,
\]
where \((\hess h)_{\mathrm{ac}}\) denotes the absolutely continuous part of
the Hessian measure.
Together with the continuum Vandermonde terms, this gives the rate function
appearing in the large deviation principle for the spectrum of a sum of two
random Hermitian matrices \cite{NarSheff}.

The existence of the surface tension and its role in the large deviation
principle were established in \cite{NarSheff}.  However, the
function $\sigma$ is not known explicitly.  This lack of a closed formula is
one of the main obstacles to obtaining a more concrete description of the
limit shape of random hives and of the associated randomized Horn problem.
A special class of boundary conditions is provided by GUE spectral data
\cite{AndersonGuionnetZeitouni,TaoBook,Gustavsson}.  In
that case, previous work \cite{GangNar} gives a sharp
identity for the total entropy in terms of the Euclidean area of a triangle
with side lengths determined by the $L^2$ norms of the three boundary
profiles.  This identity gives the exact value of the surface-tension integral
along a surface tension minimizing GUE continuum hive, but it does not by
itself determine the pointwise value of $\sigma$.  The same GUE boundary
regime was studied in \cite{NarSheffTao}, where correlation decay estimates
for GUE eigengaps were used to prove concentration of the associated random
hives; in particular, after scaling, the variance tends to zero and compactness
gives subsequential convergence to deterministic continuum hives.  The
subsequent work \cite{Nar} proves that this subsequential limit  can be
replaced with a sequential limit in the GUE boundary setting: the normalized random hives converge in
probability to a single continuum hive, and the value of this limiting hive at
a point is described by a variational problem over asymptotic height functions
for lozenge tilings.  These results identify and characterize the GUE limit
shape, while the present paper focuses on extracting information about the
surface tension by constructing GUE hive laws in additional geometric regimes.

The main purpose of this paper is to give a probabilistic construction of
GUE hive laws in the (right-angled and) obtuse regimes, up to subleading
relative entropy.  The construction is developed in
Section~\ref{sec:deformed-gue-hives}.  Let
\[
    X = \sqrt n\,(wG + uD),
    \qquad
    Y = \sqrt n\,(w'G' + u'D'),
\]
where $G,G',\widetilde G,\widetilde G'$ are independent GUE matrices and
\[
D=\diag(\spec(\widetilde G)),\qquad
D'=\diag(\spec(\widetilde G')).
\]
The minor processes of
$X$ and $Y$ give two Gelfand-Tsetlin patterns
\cite{FultonYoungTableaux}.  Using the large-gap GT-to-hive realization
recalled in Proposition~\ref{gt-rem}(iv), these may be placed as the upper
panels of a double hive, after which the octahedron recurrence produces a hive
\cite{NarSheffTao,KTW,Speyer}.  We prove in \cref{thm:main} that, under
the matching condition
\[
    \frac{u}{w^2}=\frac{u'}{(w')^2},
\]
the density $q_n$ of the resulting hive satisfies
\[
D_{\mathrm{KL}}\!\left(
q_n\,
\middle\|\,
\operatorname{Density}\bigl(\mathcal H_n(a\sqrt n,b\sqrt n,c_{**}\sqrt n)\bigr)
\right)
=
O(n\log n),
\]
where
\[
    a^2=w^2+u^2,
    \qquad
    b^2=(w')^2+(u')^2.
\]
Since the hive has order $n^2$ degrees of freedom, this is a subleading
relative-entropy error.  Under the matching condition, the remaining
parameter has the explicit form
\[
    c_{**}^2=w^2+(w')^2+(u+u')^2=a^2+b^2+2uu'.
\]
It is also determined by a limiting tetrahedral optimization problem.
Equivalently, if $\delta=u+u'$, then $c_{**}$ is the unique solution with
$a^2+b^2\leq c_{**}^2<(a+b)^2$ of
\[
    \delta^2
    =
    \frac{
        2c_{**}^4(c_{**}^2-a^2-b^2)
    }{
        (c_{**}^2-a^2+b^2)(c_{**}^2+a^2-b^2)
    }.
\]
Thus, for a target right-angled or obtuse triple $(a,b,c)$ with
$c^2\geq a^2+b^2$, the required diagonal displacement is obtained by setting
$c_{**}=c$ in this formula.  This gives a geometric interpretation of the
deformed minor-process construction in terms of tetrahedral volume.

The paper is organized as follows.  The remainder of this introduction
records two explicit surface-tension approximations and numerical experiments
that motivated the construction.  \Cref{sec:prelim} recalls the definitions of
hives, augmented hives, Gelfand-Tsetlin patterns, the Knutson--Tao
density formula, and the GUE hive maximum entropy measures.
\Cref{sec:deformed-gue-hives} describes the construction of hives from
deformed GUE minor processes.

\subsection{Motivation for constructing  GUE hives with more general boundary parameters}
\label{ssec:motivation}

The construction was motivated by questions about the hive surface
tension.  Here we record two explicit approximations,
$\widetilde\sigma$ and $\sigma_{\mathrm{BH}}$, together with numerical
comparisons, efficient algorithms for which, rely on the results of this paper.  

\subsubsection{Barvinok-Hartigan  approximation $\sigma_{BH}$ of the hive surface tension function}\lab{sec:BH}

Let \(\T_n=(\mathbb Z/n\mathbb Z)^2\), and let \(V(\T_n)\) be its vertex
set.  For a function \(g:V(\T_n)\to\mathbb R\), let \(D_0g,D_1g,D_2g\)
denote the three periodic rhombus Hessians corresponding to the three
families of rhombi.  Following \cite{NarayananEquilateral}, for
\(s=(s_0,s_1,s_2)\in\mathbb R_+^3\), define
\[
P_n(s)
:=
\left\{
g:V(\T_n)\to\mathbb R:
\sum_{v\in V(\T_n)}g(v)=0,\quad
D_i g(v)\leq s_i\ \text{for all }v\in V(\T_n),\ i=0,1,2
\right\}.
\]
The surface tension is the large-\(n\) normalized logarithmic volume of this
periodic hive polytope.  


The basis of the Barvinok-Hartigan approximation to polytope volume is the following fact from Fourier analysis, that appears in Section 6.1 of \cite{BH}.

Let $x_1,\ldots,x_n$ be independent exponential random variables with
\[
\mathbb{E}x_j=\zeta_j,\qquad j=1,\ldots,n.
\]
Let $a_1,\ldots,a_n\in\mathbb{R}^d$ span $\mathbb{R}^d$, and define
\[
Y=\sum_{j=1}^n x_j a_j.
\]
At any continuity point $b\in\mathbb{R}^d$ of its density, the density
of $Y$ is given by the inverse Fourier transform
\[
p_Y(b)
=
\frac{1}{(2\pi)^d}
\int_{\mathbb{R}^d}
e^{-i\langle b,t\rangle}
\prod_{j=1}^n
\frac{1}{1-i\zeta_j\langle a_j,t\rangle}
\,dt.
\]
If this integral is not absolutely convergent, it is interpreted by
inserting $e^{-\varepsilon |t|^2}$ and then taking
$\varepsilon\downarrow0$.

From this, defining
\[
P=\{x\in \mathbb{R}^n_{\ge 0}: Ax=b\}, \] \[
f(x) = n+
\sum_{j=1}^n
\ln x_j , \] and taking   $z = (\zeta_1,\dots , \zeta_n)$ to be the point maximizing $f$ on $P$,
Barvinok and Hartigan show in \cite{BH} that,   \[
\operatorname{vol} P
=
e^{f(z)}\left(\det AA^{T}\right)^{1/2}
\frac{1}{(2\pi)^d}
\int_{\mathbb{R}^d}
e^{-i\langle b,t\rangle}
\prod_{j=1}^{n}
\frac{1}{1-i\zeta_j\langle a_j,t\rangle}
\,dt .
\]

The results of \cite{BH} give sufficient conditions on $A$ and $b$, under which  \[
\operatorname{vol} P\] is multiplicatively approximated to within \(1 \pm \epsilon\)
by \[  \volBH P  :=
\frac{1}{(2\pi)^{d/2}}
\left(
\frac{\det(AA^{T})}{\det(BB^{T})}
\right)^{1/2}
e^{f(z)} .
\]
Here \(A\) is the \(d\times n\) matrix whose \(j\)-th column is \(a_j\), and
\[
        B=A\operatorname{diag}(\zeta_1,\ldots,\zeta_n),
        \qquad\text{equivalently}\qquad
        BB^T=\sum_{j=1}^n \zeta_j^2 a_ja_j^T .
\]
Thus \(BB^T\) is the covariance matrix of \(Y=\sum_j x_ja_j\), since an
exponential random variable with mean \(\zeta_j\) has variance \(\zeta_j^2\).

Unfortunately, these conditions are not satisfied for the polytope $P_n(s)$ whose normalized log volume we wish to asymptotically evaluate. Nonetheless, there is the possibility that it offers a useful approximation, and in fact Barvinok and Rudelson \cite{BR}, more recently provided   upper and lower bounds relating $\operatorname{vol} P
$ and  $\volBH P$.

As we shall see, $\volBH$ does suggest in an  explicit closed-form, though it turns out to not be the correct surface tension.  


In what follows, we outline the heuristics that lead to the expression we call $\sigma_{BH}$.
Let \(V_n=(\mathbb Z/n\mathbb Z)^2\), and let \(G_n\) be the \((n^2-1)\)-dimensional space of mean-zero functions \(g:V_n\to\mathbb R\). Write the three discrete rhombus Hessians as

\(D_i g(v)= \nabla_i^2 g(v),\qquad i=0,1,2.\)
For \(s=(s_0,s_1,s_2)\), introduce slack variables

\[
x_i(v)=s_i-D_i g(v)\ge 0.
\]
Thus \(P_n(s)\) is affinely identified with

\[
\{x\in\mathbb R_{\ge 0}^{3n^2}: x=s-Dg,\ g\in G_n\}.
\]
Equivalently it is an affine section of the orthant.

Let

\[
D:G_n\to \mathbb R^{3n^2},\qquad Dg=(D_0g,D_1g,D_2g).
\]
The volume in slack coordinates differs from the volume in \(g\)-coordinates by

\[
d\operatorname{vol}_{x}
=
\sqrt{\det{}'(D^*D)}\,d\operatorname{vol}_{g}.
\]
By translation symmetry and strict concavity, the maximizer in the Barvinok-Hartigan formula is

\[
z_i(v)=s_i.
\]
Consequently, in the present application the Barvinok--Hartigan matrix \(B\)
is the matrix obtained from the affine-constraint matrix \(A\) in slack
coordinates by multiplying every column corresponding to a type \(i\) rhombus
slack by \(s_i\). Equivalently,
\[
        B=A\,\operatorname{diag}\bigl(s_0 I_{n^2},s_1 I_{n^2},s_2 I_{n^2}\bigr),
        \qquad
        BB^T=A\,\operatorname{diag}\bigl(s_0^2 I_{n^2},s_1^2 I_{n^2},s_2^2 I_{n^2}\bigr)A^T .
\]

\begin{claim}\lab{cl:BH}
Let \(A:\mathbb R^N\to\mathbb R^d\) have full row rank, let \(R\) be a
positive diagonal \(N\times N\) matrix, and let
\[
D:G\to\mathbb R^N
\]
be a linear map whose image is \(\ker A\).  Interpreting \(\det'\) as the
product of the non-zero eigenvalues,
\beq \lab{eq:BH}
\det(R)^2\det(AA^T)\det'(D^TR^{-2}D)
=
\det(AR^2A^T)\det'(D^TD).
\eeq
\end{claim}
\begin{proof}
Choose an orthogonal decomposition
\[
\mathbb R^N=(\ker A)^\perp\oplus \ker A.
\]
Let \(Q_1\) and \(Q_2\) be matrices whose columns are orthonormal bases of
\((\ker A)^\perp\) and \(\ker A\), respectively. Since
\(\ker A=\operatorname{im}D\), after replacing the domain of \(D\) by
\((\ker D)^\perp\) we may write
\[
D=Q_2C
\]
for some invertible matrix \(C\).

Because \(AQ_2=0\), we have
\[
AA^T=(AQ_1)(AQ_1)^T,
\]
and
\[
AR^2A^T=(AQ_1)(Q_1^TR^2Q_1)(AQ_1)^T.
\]
Hence
\[
\frac{\det(AR^2A^T)}{\det(AA^T)}
=
\det(Q_1^TR^2Q_1).
\]
By Jacobi's complementary minor identity applied to \(R^2\),
\[
\det(Q_1^TR^2Q_1)
=
\det(R)^2\det(Q_2^TR^{-2}Q_2).
\]
On the other hand,
\[
D^TD=C^TC,
\qquad
D^TR^{-2}D=C^TQ_2^TR^{-2}Q_2C,
\]
so
\[
\frac{\det'(D^TR^{-2}D)}{\det'(D^TD)}
=
\det(Q_2^TR^{-2}Q_2).
\]
Combining the last three displays gives
\[
\det(R)^2\det(AA^T)\det'(D^TR^{-2}D)
=
\det(AR^2A^T)\det'(D^TD),
\]
as desired.
\end{proof}

We now apply the determinant identity to the slack-coordinate description of
\(P_n(s)\).  Here \(N=3n^2\), \(d=2n^2+1\), and
\(\dim G_n=n^2-1\).  The slack-constraint matrix displayed above may have
redundant rows.  For the application of Claim~\ref{cl:BH}, replace it by any
full-row-rank matrix \(\widetilde A:\mathbb R^N\to\mathbb R^d\) with the
same kernel, obtained by choosing a basis of its row space.  This replacement
does not change the affine hull in slack space, and
\(\ker \widetilde A=\operatorname{im}D\).  The slack-to-height Jacobian contributes
\(\sqrt{\det{}'(D^*D)}\), while the Barvinok--Hartigan determinant ratio
contributes the reciprocal square root of the corresponding weighted
operator.  Using Claim~\ref{cl:BH} with \(A=\widetilde A\) and
\[
R=\operatorname{diag}\bigl(s_0 I_{n^2},s_1 I_{n^2},s_2 I_{n^2}\bigr)
\]
therefore motivates the approximation
\[
\volBH P_n(s)
=
\frac{e^{3n^2}}{(2\pi)^{(2n^2+1)/2}}
\frac{1}{
\sqrt{\det{}'\!\left(D^*
\operatorname{diag}(s_0^{-2},s_1^{-2},s_2^{-2})D\right)}
}.
\]
Note that 
\[
L_s:=D^*\operatorname{diag}(s_0^{-2},s_1^{-2},s_2^{-2})D.
\]
Fourier modes on \(V_n\) are

\[
\phi_{\theta}(a,b)=e^{i(a\theta_1+b\theta_2)},
\qquad
\theta_1=\frac{2\pi k}{n},\quad
\theta_2=\frac{2\pi \ell}{n}.
\]
Let

\[
\theta_0=-\theta_1-\theta_2.
\]
As operators, the three rhombus Hessians respectively have eigenvalues (up to sign),

\[
 (1-e^{i\theta_1})(1-e^{i\theta_2}),\qquad
 (1-e^{i\theta_2})(1-e^{i\theta_0}),\qquad
 (1-e^{i\theta_0})(1-e^{i\theta_1}).
\]
Therefore their squared moduli are respectively

\(
16\sin^2(\theta_1/2)\sin^2(\theta_2/2),
\)
\(
16\sin^2(\theta_2/2)\sin^2(\theta_0/2),
\) and 
\(
16\sin^2(\theta_0/2)\sin^2(\theta_1/2).
\)
Hence the nonzero Fourier eigenvalues of \(L_s\) are

\[
\lambda_s(\theta_1,\theta_2)
=
16\left(
\frac{\sin^2(\theta_1/2)\sin^2(\theta_2/2)}{s_0^2}
+
\frac{\sin^2(\theta_2/2)\sin^2(\theta_0/2)}{s_1^2}
+
\frac{\sin^2(\theta_0/2)\sin^2(\theta_1/2)}{s_2^2}
\right),
\]
excluding only the zero mode \((k,\ell)=(0,0)\), corresponding to constants.
Therefore, this motivates the large-\(n\) approximation 

\[ - \sigma_{BH} :=
\lim_{n\to\infty}
\frac{1}{n^2-1}\log \volBH P_n(s)
=
3-\log(2\pi)
-\frac12
\int_{[0,2\pi]^2}
\log \lambda_s(\theta_1,\theta_2)
\frac{d\theta_1d\theta_2}{(2\pi)^2}
.
\]

\begin{proposition}
For all sufficiently small   $\de > 0$, there exists $\eps > 0$ such that $$\exp(-\sigma_{\mathrm{BH}}(\eps, 1, 1)) >  \left(\frac{e^2}{2\pi} - \de\right)\exp(-\sigma(\eps, 1, 1)) > 1.176 \exp(-\sigma(\eps, 1, 1)).$$
\end{proposition}
\begin{proof} 
For \(s=(\epsilon,1,1)\), the first term inside \(\lambda_s\) dominates away from a set of small measure:
\[
\lambda_{\epsilon,1,1}(\theta_1,\theta_2)
\sim
\frac{16}{\epsilon^2}
\sin^2\left(\frac{\theta_1}{2}\right)
\sin^2\left(\frac{\theta_2}{2}\right).
\]
Therefore
\[
\int_{[0,2\pi]^2}
\log \lambda_{\epsilon,1,1}(\theta_1,\theta_2)
\frac{d\theta_1d\theta_2}{(2\pi)^2}
=
-2\log\epsilon+o_\eps(1),
\]
because
\[
\int_0^{2\pi}
\log\sin^2\left(\frac{\theta}{2}\right)
\frac{d\theta}{2\pi}
=
-2\log 2,
\]
and the factor \(16\) cancels the two \(-2\log 2\) contributions.

Thus
\[
-\sigma_{\mathrm{BH}}(\epsilon,1,1)
=
3-\log(2\pi)+\log\epsilon+o_\eps(1),
\]
so
\[
\exp\bigl(-\sigma_{\mathrm{BH}}(\epsilon,1,1)\bigr)
=
\frac{e^3}{2\pi}\,\epsilon\,(1+o_\eps(1)).
\]
Equivalently,
\[
\sigma_{\mathrm{BH}}(\epsilon,1,1)
=
-\log\epsilon-3+\log(2\pi)+o_\eps(1).
\]

Comparing with the $\tilde{\sigma}$ asymptotic,
\[
\exp\bigl(-\tilde{\sigma}(\epsilon,1,1)\bigr)
\sim e\epsilon,
\]
whereas
\[
\exp\bigl(-\sigma_{\mathrm{BH}}(\epsilon,1,1)\bigr)
\sim
\frac{e^3}{2\pi}\,\epsilon.
\]
Their ratio is
\[
\frac{
\exp\bigl(-\sigma_{\mathrm{BH}}(\epsilon,1,1)\bigr)
}{
\exp\bigl(-\tilde{\sigma}(\epsilon,1,1)\bigr)
}
\longrightarrow
\frac{e^2}{2\pi}
\approx 1.1760048.
\]
For \(R\geq 1\), the constraints defining \(P_n(\epsilon,1,R)\) are weaker
than those defining \(P_n(\epsilon,1,1)\), and hence
\[
P_n(\epsilon,1,R)\supseteq P_n(\epsilon,1,1).
\]
Letting \(R\to\infty\), and using the two-coordinate degeneration, gives
\[
\exp(-\tilde\sigma(\epsilon,1,1))
=
\exp(-\sigma(\epsilon,1,\infty))
\geq
\exp(-\sigma(\epsilon,1,1)).
\]
This proves the proposition.

\end{proof}

\subsubsection{A new approximation $\tilde{\sigma}$ of the hive surface tension function}\lab{sec:tilde-sigma}

We now introduce a second explicit approximation to the hive surface tension.
Unlike the Barvinok--Hartigan expression of Subsubsection~\ref{sec:BH}, this
formula is not obtained from a general Gaussian approximation theorem.
Instead it is designed to take into account two pieces of information that
are available for the true surface tension: the exact two-dimensional
degeneration and the integrated GUE entropy identity.

For \(s=(s_0,s_1,s_2)\in \mathbb R_{>0}^3\), set
\[
T_{ij}(s)
  =
  \frac{e}{\pi}(s_i+s_j)
      \sin\left(\frac{\pi s_i}{s_i+s_j}\right),
      \qquad 0\leq i<j\leq 2,
\]
and
\[
F(s)
  =
  \frac{4e}{\pi}
      \frac{(s_0+s_1+s_2)s_0s_1s_2}
           {(s_0+s_1)(s_1+s_2)(s_2+s_0)} .
\]
We define
\[
\tilde{\sigma}(s_0,s_1,s_2)
   =
   \max\{-\log T_{01}(s),-\log T_{02}(s),-\log T_{12}(s),-\log F(s)\}.
\]
Equivalently,
\[
\exp(-\tilde\sigma(s))
=
\min\{T_{01}(s),T_{02}(s),T_{12}(s),F(s)\}.
\]
Thus \(\tilde\sigma\) is symmetric in the three slack variables and is
logarithmically homogeneous:
\[
\tilde\sigma(ts)=\tilde\sigma(s)-\log t,
\qquad t>0.
\]
This is the same homogeneity possessed by the true surface tension, since
scaling all three rhombus slacks scales the local polytope volume by the
corresponding factor per vertex.

The terms \(T_{ij}\) encode the known two-coordinate degeneration
\cite{SchTao,JohnstonProchno}.  When the
third slack is sent to infinity, the local hive constraints reduce to a
one-dimensional interlacing problem, and the surface tension satisfies
\[
\exp(-\sigma(s_i,s_j,\infty))
=\frac{e}{\pi}(s_i+s_j)
\sin\left(\frac{\pi s_i}{s_i+s_j}\right).
\]
Thus each \(T_{ij}\) is forced by a boundary degeneration of the theory.  The
minimum over the three pairwise terms enforces these three degenerations
simultaneously.

The remaining term \(F\) is the symmetric three-variable correction suggested
by the GUE boundary case.  It is the simplest homogeneous symmetric expression
with the same rational dependence on the three slacks as the GUE integrated
entropy formula suggests after normalizing by the pairwise sums.  In the
balanced region of the simplex it is often the active term, while near the
edges one of the pairwise degeneration terms takes over.  The active region
for \(F\) on the normalized simplex is shown in Figure~\ref{fig:F}.

\begin{figure}[h!]
  \centering
  \includegraphics[width=0.6\textwidth]{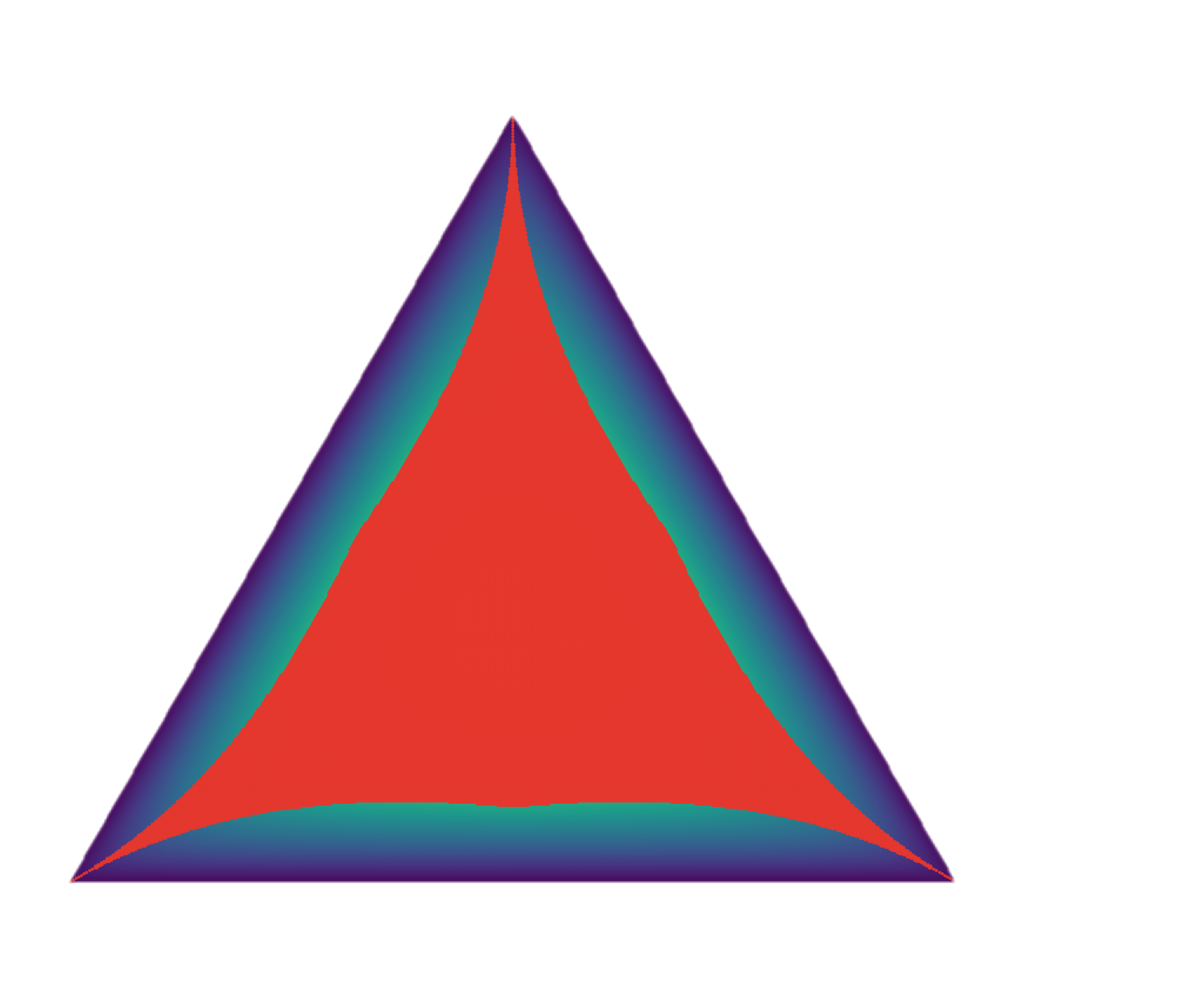} 
  \caption{The active region in red for $F$ on the triangle $\{(s_0, s_1, s_2) \in \mathbb{R}_+^3: s_0 + s_1 + s_2 = 1\}$.}
  \label{fig:F}
\end{figure}

This approximation should be regarded as a constrained ansatz rather than a
formula derived from first principles.  Its appeal is that it is explicit,
symmetric, has the correct logarithmic homogeneity, agrees with the known
pairwise degenerations, and incorporates the symmetric GUE entropy heuristic
through the term \(F\).  The numerical evidence in Subsubsection~\ref{sec:Num}
suggests that these constraints capture a substantial part of the true
surface tension in the tested regimes.

\subsubsection{Numerical experiments}\lab{sec:Num}

We compare empirical averages of the surface-tension approximations on
sampled GUE hives with the limiting minimum surface-tension integral
given by \cite{GangNar}. In the first family, for each \(b\), we set
\[
        a=1, \qquad c=\sqrt{1+b^{2}},
\]
generate \(M=100\) hives of size \(N=100\), average the sampled hives to
produce an approximation of the center of mass, and evaluate the pointwise
approximations \(\sigma_{\mathrm{BH}}\) and \(\tilde\sigma\) on the hive
slacks of this approximate center of mass.

Combining the surface-tension variational principle of \cite{NarSheff}
with the GUE boundary entropy identity \cite[Theorem~9]{GangNar}, the
theoretical limiting minimum average surface tension is
\[
        E_{\mathrm{GUE}}(a,b,c):=-\frac54
        - \log\left(\frac{4\Delta_{abc}^2}{abc}\right),
\]
where \(\Delta_{abc}\) denotes the area of the Euclidean triangle with side
lengths \(a,b,c\).  In the right-triangle family with side lengths
\((1,b,\sqrt{1+b^2})\), this becomes
\[
        -\frac54
        - \log\left(\frac{b}{\sqrt{1+b^2}}\right).
\]
We also perform analogous experiments in the deformed regime with \(a=b\)
and \(c\geq \sqrt{a^2+b^2}\).  Since the boundary data come from random GUE
spectra, the empirical standard deviations of the measured boundary sides
have visible finite-size fluctuations; the tables below therefore use the
measured boundary side lengths throughout. The value
\(E_{\mathrm{GUE}}(a,b,c)\), evaluated at this measured triple, is a
limiting benchmark; it is not an exact evaluation of the unknown
surface-tension integral of the finite averaged hive. The discrepancy
columns therefore include finite-size and sampling effects as well as
approximation error. In these tables, \(\int f\) denotes the normalized
integral \(2\int_T f\,dx\) with ordinary planar Lebesgue measure.

The scripts and averaged hive data used to reproduce the numerical tables in
this subsubsection are available at
\url{https://github.com/harius80/general-gue-hive-v2-tables}, with the version
used for this draft tagged as \texttt{v1-paper-artifact}.

\begin{table}[ht]
\centering
\begin{tabular}{r r r r r r}
\hline
$a$ & $b$ & $c$ & $E_{\mathrm{GUE}}$ & $\int \sigma_{\mathrm{BH}}-E_{\mathrm{GUE}}$ & $\int \tilde{\sigma}-E_{\mathrm{GUE}}$ \\
\hline
1.005215961 & 1.003797387 & 1.428392913 & -0.902321189 & 0.058309405 & 0.007799368 \\
1.005323323 & 2.006380661 & 2.252487872 & -1.139520066 & 0.042063570 & 0.002504651 \\
1.004951168 & 3.018099799 & 3.187762839 & -1.200196813 & 0.024917878 & -0.000240172 \\
1.004228197 & 4.019762283 & 4.150095911 & -1.222261919 & 0.012638397 & -0.000322546 \\
1.003941708 & 5.021774471 & 5.127528634 & -1.233051405 & 0.004108598 & 0.000400882 \\
1.004334630 & 6.027406821 & 6.116465724 & -1.239621513 & -0.002046527 & 0.001054968 \\
1.004358622 & 7.038781415 & 7.113872383 & -1.243722890 & -0.006795411 & 0.001650896 \\
1.004541034 & 8.040359123 & 8.108783142 & -1.246021456 & -0.010654372 & 0.001990415 \\
1.004396958 & 9.036942007 & 9.096209106 & -1.247837245 & -0.013255740 & 0.002691587 \\
1.005509340 & 10.039370774 & 10.093210570 & -1.250132650 & -0.015836561 & 0.002749842 \\
1.004604737 & 11.046523764 & 11.100398808 & -1.249660236 & -0.017679698 & 0.003136968 \\
1.004765504 & 12.058124508 & 12.102640026 & -1.251061825 & -0.019301336 & 0.003447622 \\
1.004827178 & 13.069043992 & 13.111393263 & -1.251566156 & -0.020902437 & 0.003392579 \\
1.005471791 & 14.064885007 & 14.102854533 & -1.252756631 & -0.021981079 & 0.003719346 \\
1.005685105 & 15.058945161 & 15.096485956 & -1.253163317 & -0.023101185 & 0.003781554 \\
1.005802918 & 16.081590007 & 16.115802693 & -1.253653235 & -0.024218808 & 0.003714284 \\
1.004617872 & 17.100303274 & 17.135687259 & -1.252505562 & -0.024930960 & 0.003960104 \\
1.004888003 & 18.092896045 & 18.123631626 & -1.253170696 & -0.025614417 & 0.004069573 \\
1.004650923 & 19.063136451 & 19.094649761 & -1.252962974 & -0.026437356 & 0.004056875 \\
1.005115100 & 20.091564769 & 20.120663171 & -1.253639155 & -0.027167716 & 0.003900652 \\
\hline
\end{tabular}
\caption{Table 1 uses measured boundary side lengths throughout. The columns $a$, $b$, and $c$ are measured from the boundary of the averaged hive $h$; $E_{\mathrm{GUE}}$ is the limiting benchmark evaluated at this measured triple, and both discrepancy columns subtract that benchmark. The fifth column recomputes $\sigma_{\mathrm{BH}}$; both empirical discrepancy columns use the same interior hive sites.}
\label{tab:sample-out-100-right-triangle-measured-abc-direct-bh}
\end{table}

\begin{table}[ht]
\centering
\small
\begin{tabular}{r r r r r r r}
\hline
$w=w'$ & $a$ & $b$ & $c$ & $E_{\mathrm{GUE}}$ & $\int \sigma_{\mathrm{BH}}-E_{\mathrm{GUE}}$ & $\int \tilde{\sigma}-E_{\mathrm{GUE}}$ \\
\hline
0.2 & 1.024163337 & 1.024057507 & 2.028499487 & 1.997406985 & -0.114639236 & 0.018600656 \\
0.4 & 1.082625417 & 1.083241122 & 2.090582657 & 0.696252187 & -0.081107852 & 0.007200553 \\
0.6 & 1.172129704 & 1.170225987 & 2.183572357 & 0.001744161 & -0.047986656 & 0.000358700 \\
0.8 & 1.288184320 & 1.287005094 & 2.312655159 & -0.445927766 & -0.021503564 & -0.002923652 \\
1 & 1.420512193 & 1.421098363 & 2.464195978 & -0.757514359 & -0.002107933 & -0.003906648 \\
2 & 2.248710857 & 2.248429915 & 3.494543001 & -1.575314860 & 0.038714111 & 0.000784702 \\
3 & 3.175466064 & 3.177554497 & 4.729654244 & -1.995876285 & 0.049061791 & 0.003898951 \\
4 & 4.140141436 & 4.141463309 & 6.046879383 & -2.287836586 & 0.053195010 & 0.005681186 \\
5 & 5.128391729 & 5.118838448 & 7.414775633 & -2.512018475 & 0.054839514 & 0.006285524 \\
6 & 6.113612401 & 6.117119471 & 8.809497982 & -2.694364384 & 0.055740764 & 0.006667216 \\
7 & 7.098142988 & 7.092169066 & 10.179556037 & -2.847589479 & 0.056331831 & 0.006874918 \\
8 & 8.096376438 & 8.086282625 & 11.581413788 & -2.981590522 & 0.056741710 & 0.007101435 \\
9 & 9.099281548 & 9.093649260 & 13.004965108 & -3.099957230 & 0.057375301 & 0.007325534 \\
10 & 10.095622676 & 10.101947128 & 14.413914587 & -3.206290775 & 0.057596483 & 0.007733413 \\
\hline
\end{tabular}
\caption{Table 2 uses measured boundary side lengths throughout. Here $u = u' = 1$. The columns $a$, $b$, and $c$ are measured from the  boundary of the averaged hive $h$; $E_{\mathrm{GUE}}$ is the limiting benchmark evaluated at this measured triple, and both discrepancy columns subtract that benchmark. The sixth column evaluates $\sigma_{\mathrm{BH}}$ directly, by numerical quadrature, at the slack ratios measured in the simulation. This avoids the small boundary error introduced by the simplex lookup table, which clamps near-boundary slack ratios to the edge of its grid; both empirical discrepancy columns use the same interior hive sites.}
\label{tab:sample-out-100-deformed-v2-redo-measured-a-b-c-direct-bh}
\end{table}

The small discrepancies in the last column support the accuracy of $\tilde{\sigma}$ on these averaged hives, subject to the finite-size and sampling effects described above.

\begin{figure}[h!]
  \centering
  \includegraphics[width=\textwidth]{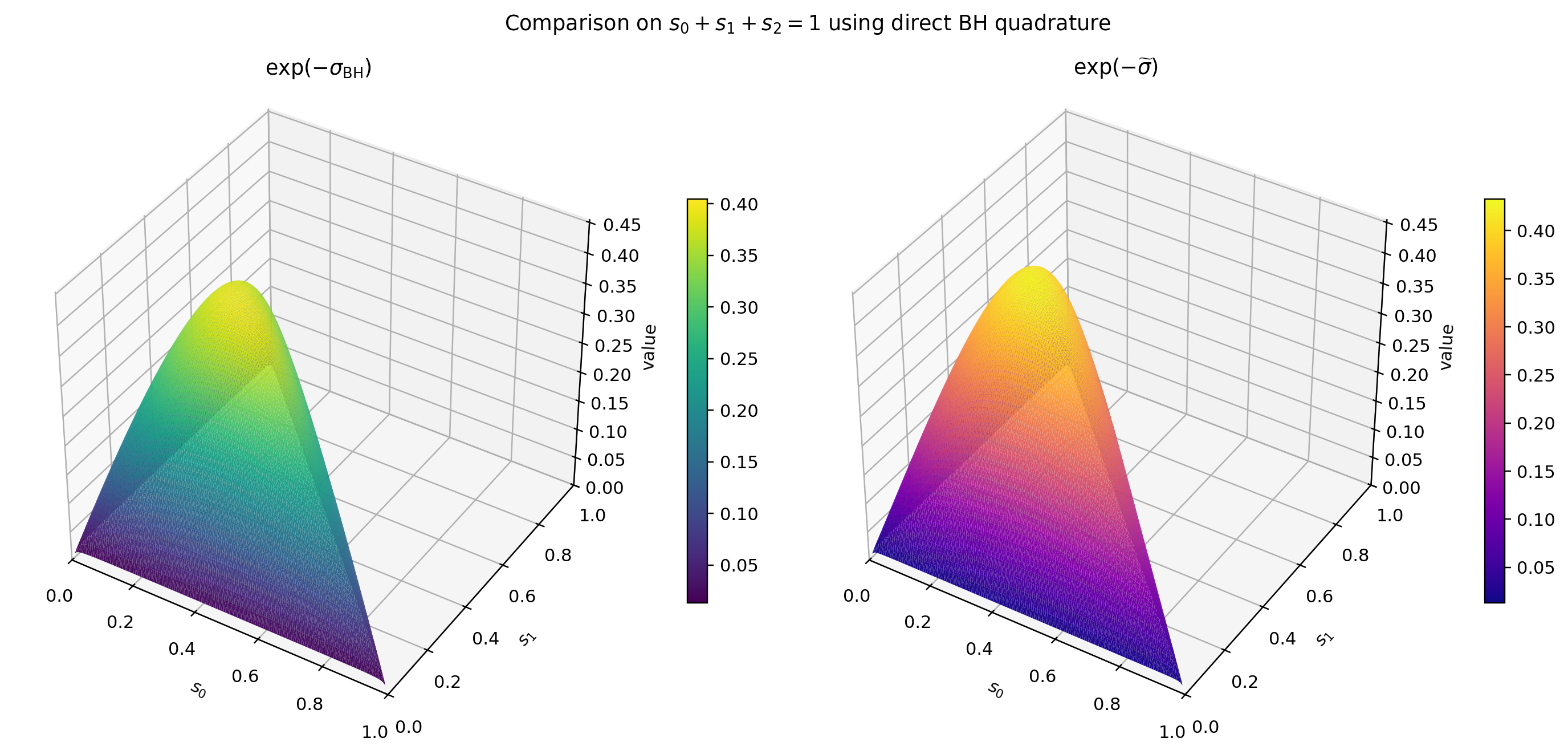}
  \caption{On the left is the plot of $\exp(-\sigma_{\mathrm{BH}})$. On the right is the plot of $\exp(-\tilde{\sigma})$, both on the simplex $s_0+s_1+s_2=1$, using direct quadrature for $\sigma_{\mathrm{BH}}$.}
  \label{fig:bh-tilde-direct}
\end{figure}

\clearpage
\section{Preliminaries}
\lab{sec:prelim}




For a $n \times n$ Hermitian matrix $W$, let $\spec(W)$ denote the vector in $\R^n$ whose coordinates are the eigenvalues of $W$ listed in non-increasing order. We denote by $\Spec_n$ the set of all vectors in $\R^n$ whose coordinates are non-increasing and sum to $0$.  Thus \(\Spec_n+\mathbb R\one\) is the full ordered Weyl chamber, and ordinary GUE spectra lie in \(\Spec_n+\mathbb R\one\), not necessarily in \(\Spec_n\).

We use the following normalization of the GUE.  Let \(\Herm_n\) be the real
vector space of \(n\times n\) complex Hermitian matrices, equipped with
Lebesgue measure
\[
dH=\prod_{i=1}^n dH_{ii}
\prod_{1\leq i<j\leq n}
d(\operatorname{Re}H_{ij})\,d(\operatorname{Im}H_{ij}).
\]
A random Hermitian matrix \(G\) has law \(GUE_n\) if its density with respect
to \(dH\) is
\[
\frac{1}{(2\pi)^{n/2}\pi^{n(n-1)/2}}
\exp\left(-\frac12\operatorname{Tr}(H^2)\right).
\]
Equivalently, the diagonal entries are independent \(N(0,1)\) random
variables, and the real and imaginary parts of the off-diagonal entries are
independent \(N(0,1/2)\) random variables, subject to Hermitian symmetry.  By
the Weyl integration formula, the ordered spectrum \(\lambda=\spec(G)\) has
density on the full Weyl chamber
\[
\mathbb R^n_{\geq}:=\{\lambda\in\mathbb R^n:\lambda_1\geq\cdots\geq\lambda_n\}
\]
given by
\[
\frac{1}{(2\pi)^{n/2}V_n(\tau_n)}
V_n(\lambda)^2
\exp\left(-\frac12|\lambda|^2\right)\,d\lambda.
\]
Here, for \(x\in\mathbb R^n\), we write
\[
x^\downarrow_1\geq x^\downarrow_2\geq\cdots\geq x^\downarrow_n
\]
for the decreasing rearrangement of the entries of \(x\), and set
\begin{equation}\label{eq:sorted-vandermonde}
V_n(x)=\prod_{1\leq i<j\leq n}(x^\downarrow_i-x^\downarrow_j).
\end{equation}
Thus \(V_n\) is the nonnegative Vandermonde obtained after sorting the entries
in decreasing order.  Also
\[
\tau_n=\left(\frac{n-1}{2},\frac{n-3}{2},\ldots,-\frac{n-3}{2},-\frac{n-1}{2}\right),
\]
so that \(V_n(\tau_n)=\prod_{k=1}^{n-1}k!\).
More generally, if \(s>0\), then \(\spec(sG)\) has density on
\(\mathbb R^n_{\geq}\) given by
\[
\frac{1}{(2\pi)^{n/2}s^{n^2}V_n(\tau_n)}
V_n(\lambda)^2
\exp\left(-\frac1{2s^2}|\lambda|^2\right)\,d\lambda.
\]

\begin{definition}[Entropy, relative entropy, and mutual information]
All entropies in this paper are differential entropies with respect to the
Lebesgue measure on the affine span of the relevant polytope or cone.  If
\(p\) is a probability density, we write
\[
\ent(p):=-\int p(x)\log p(x)\,dx.
\]
If \(X\) has density \(p_X\), we also write \(\ent(X)=\ent(p_X)\).  For two
random variables \(X,Y\) with joint density \(p_{X,Y}\) and conditional
density \(p_{X\mid Y=y}\), the conditional entropy is
\[
\ent(X\mid Y)
:=
\mathbb E\bigl[\ent(p_{X\mid Y})\bigr]
=
-\int p_{X,Y}(x,y)\log p_{X\mid Y=y}(x)\,dx\,dy,
\]
whenever the conditional densities exist and the integral is well-defined.  For two
probability densities \(q,p\) on the same space, the relative entropy is
\[
D_{\mathrm{KL}}(q\|p)
:=
\int q(x)\log\frac{q(x)}{p(x)}\,dx,
\]
with the usual convention that it is \(+\infty\) unless \(q\) is absolutely
continuous with respect to \(p\).  For random variables \(X,Y\), their mutual
information is
\[
I(X;Y):=
D_{\mathrm{KL}}\bigl(\operatorname{Law}(X,Y)\|
\operatorname{Law}(X)\otimes\operatorname{Law}(Y)\bigr).
\]
When the relevant conditional densities exist, this is equivalently
\[
I(X;Y)=\ent(X)-\ent(X\mid Y).
\]
\end{definition}

\begin{lemma}[Data-processing inequality; see {\cite[Theorem~2.8.1]{CoverThomas}}]\lab{lem:data-processing}
Let \(P\) and \(Q\) be probability measures on a measurable space, and let
\(\Phi\) be a measurable map.  Then
\[
D_{\mathrm{KL}}(\Phi_\#P\|\Phi_\#Q)
\leq
D_{\mathrm{KL}}(P\|Q),
\]
with the convention that the right-hand side may be \(+\infty\).  In
particular, if \(X,Y\) are random variables and \(\Phi\) is measurable, then
\[
I(X;\Phi(Y))\leq I(X;Y).
\]
\end{lemma}

\label{subsec:Hives}

\subsection{Gelfand--Tsetlin patterns,  hives and the octahedron recurrence}

\begin{figure}[!h]
\begin{center}
\begin{minipage}{0.45\linewidth}
\centering
\includegraphics[width=\linewidth]{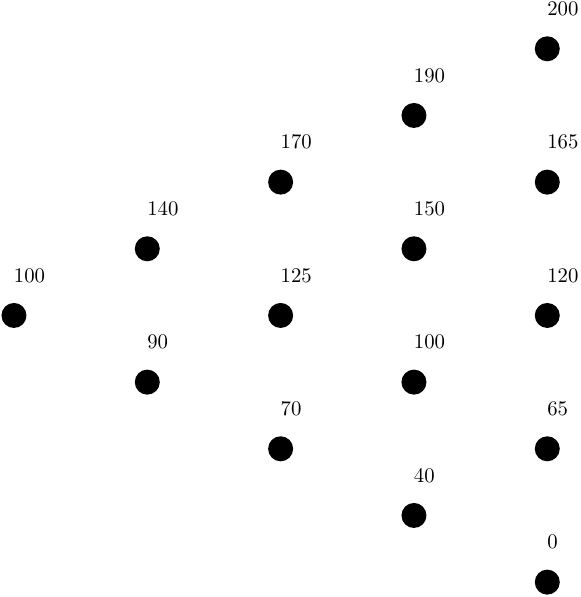}
\caption{A hive for \(n=5\).}
\label{fig:hive-example}
\end{minipage}
\hfill
\begin{minipage}{0.45\linewidth}
\centering
\includegraphics[width=\linewidth]{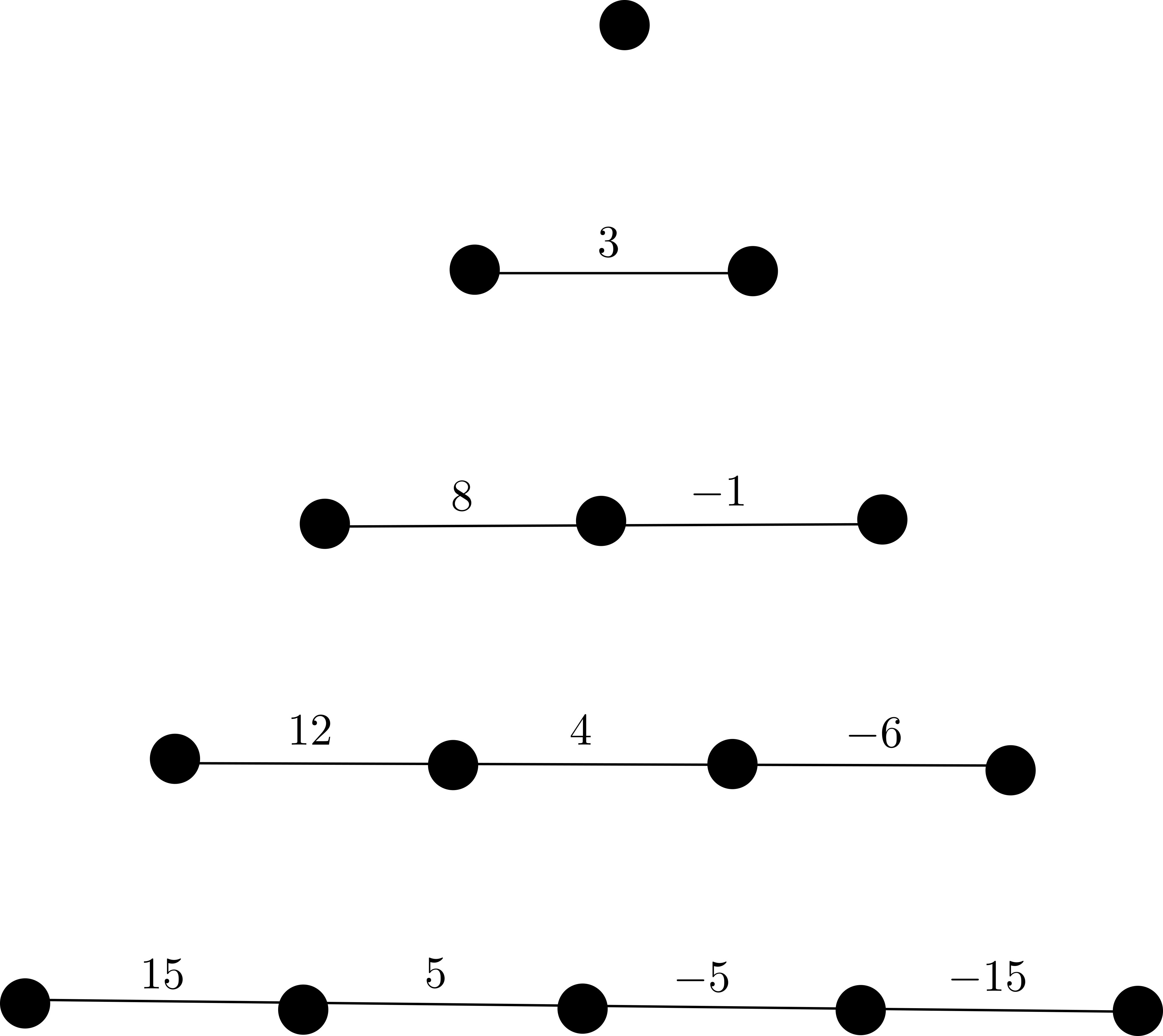}
\caption{A Gelfand--Tsetlin pattern for \(n=5\).}
\label{fig:gt-example}
\end{minipage}
\end{center}
\end{figure}

Let $T$ be the triangle $\{(x, y) \in [0, 1]^2| x \leq y\}.$ Let $T_n$ denote the set $nT \cap \Z^2.$
\begin{definition}[Discrete Hessian and the $\De_i$ on $T_n$]
Let $f: T_n \ra \R$ be a function.
\bit
\item Let $E_0(T_n)$ be the set of all parallelograms $e_0\subseteq T_n$ whose vertices are $\{(v_1, v_2),  (v_1 + 1, v_2),  (v_1 + 1, v_2 + 1),  (v_1 +2, v_2 + 1)\}.$

\item Let $E_1(T_n)$ be the set of all parallelograms $e_1\subseteq T_n$ whose vertices are $\{(v_1, v_2),  (v_1 + 1, v_2),  (v_1, v_2 + 1),  (v_1 +1, v_2 + 1)\}.$ 
\item Let $E_2(T_n)$ be the set of all parallelograms $e_2\subseteq T_n$ whose vertices are $\{(v_1, v_2),  (v_1 + 1, v_2+1),  (v_1, v_2 + 1),  (v_1 +1, v_2 + 2)\}.$ 
\eit

We define the discrete Hessian $\hess(f):E(T_n) \ra \R$ to be a  real-valued function on the set $E(T_n) = E_0(T_n) \cup E_1(T_n) \cup E_2(T_n)$ and the $\De_i$ from $\R^{T_n}$ to $\R^{E_i(T_n)}$ by 
\beq\lab{eq:A}
\hess f(e_0) := \De_0 f(e_0) :=  f(v_1, v_2) - f(v_1 + 1, v_2) - f(v_1 + 1, v_2 + 1) + f(v_1 +2, v_2 + 1).\nonumber\\
\hess f(e_1) := \De_1 f(e_1) :=  - f(v_1, v_2) + f(v_1 + 1, v_2) + f(v_1, v_2 + 1) - f(v_1 + 1,  v_2 + 1).\nonumber\\
\hess f(e_2) := \De_2 f(e_2) :=  f(v_1, v_2) - f(v_1+1, v_2+1) - f(v_1, v_2 + 1) + f(v_1 + 1, v_2 + 2).\nonumber\\
\eeq
\end{definition}

\begin{definition}[Rhombus concavity]\lab{def:rhomb}
Given a function $h:T \ra \R$, and a positive integer $n$,  let $h_n$ denote the function from $T_n$ to $\R$ such that for $(nx, ny) \in T_n$, $h_n(nx, ny) = n^2 h(x, y).$
A function $h:T \ra \R$ is called {\bf rhombus concave} if for any positive integer $n$, 
and any $i$, $\De_i h_n$ is nonpositive on $E_i(T_n),$ and $h$ is continuous on $T$.
The corresponding function $h_n$ is called {\bf discrete (rhombus) concave}.  
Note that a necessary and sufficient condition for a function $h_n$ from $T_n$ to $\R$ to be discrete concave,  is that the piecewise linear extension (which we denote $\tilh_n$) of $h_n$ to $nT$ is  concave.
Here each piece is an isosceles right triangle with a $\sqrt{2}-$length  hypotenuse parallel to the vector $(1, 1).$
\end{definition}

\begin{definition}[Hive]Let $ H_n(\lambda_n, \mu_n;  \nu_n)$ denote the set of all discrete concave functions $h_n:T_n \ra \R,$ (which,  following Knutson and Tao \cite{KT1},  we call hives)
such that \ben \item $\forall i \in [n]\cup\{0\},\quad h_n(0,  i) = \sum_{j = 1}^i \la_n(j).$
\item $\forall i \in [n]\cup\{0\},\quad h_n(i,  n) = \sum_{j = 1}^n \la_n(j) + \sum_{j = 1}^i \mu_n(j).$
\item $\forall i \in [n]\cup\{0\},\quad h_n(i,  i) = \sum_{j = 1}^i \nu_n(j).$
\een
Let $|H_n(\lambda_n, \mu_n;  \nu_n)|$ denote the ${n-1 \choose 2}-$dimensional Lebesgue measure of this hive polytope. 
\end{definition}

   Denoting probability densities with respect to the $n-1$ dimensional Lebesgue measure $$\Leb_{n-1, 0}(d\nu) = (d\nu(1)) \dots (d\nu(n-1))$$ on the hyperplane  in $\R^n$ consisting of vectors whose coordinates sum to $0$  by $\rho_n$, the following theorem is due to Knutson and Tao \cite[Theorem~4]{KT2}.  Coquereaux and Zuber stated the result in a more explicit form; see Proposition~4 in \cite{CZ} and Equation~(4) in \cite{Zuberhorn}.
   \begin{thm}[Knutson--Tao]\lab{thm:1}
    Let $X_n = U_n \mathrm{diag}(\la_n)U_n^*$ and $Y_n = V_n \mathrm{diag}(\mu_n)V_n^*$ where $U_n$ and $V_n$ are independent random unitary matrices sampled from the Haar measure on the unitary group $\mathbb{U}_n.$ Then, 
   \begin{eqnarray}\lab{eq:2.4new} \rho_n\left[\spec(X_n + Y_n) = \nu_n\right] =  \frac{V_n(\nu_n)V_n(\tau_n)}{V_n(\la_n)V_n(\mu_n)} |H_n(\la_n, \mu_n; \nu_n)|.\end{eqnarray}
   \end{thm}

\begin{proposition}[Gelfand-Tsetlin facts, {\cite[Proposition 2]{NarSheffTao}}]\label{gt-rem}
Let \(\lambda=(\lambda_1\geq \cdots \geq \lambda_n)\in\Spec_n+\mathbb R\one\). For
\(a\in\mathbb R^n\), let \(\GT_{\diag(\lambda)\rel a}\) denote the real
Gelfand-Tsetlin polytope of triangular arrays
\[
        \Gamma=\{\lambda^{(k)}_i:1\leq i\leq k\leq n\}
\]
satisfying
\[
        \lambda^{(k+1)}_i
        \geq
        \lambda^{(k)}_i
        \geq
        \lambda^{(k+1)}_{i+1},
        \qquad
        1\leq i\leq k<n,
\]
with top row \(\lambda^{(n)}=\lambda\), and diagonal boundary condition
\[
        \sum_{i=1}^k \lambda^{(k)}_i
        =
        \sum_{i=1}^k a_i,
        \qquad 1\leq k\leq n.
\]
Let \(\GT_{\diag(\lambda)\rel *}:=\bigcup_a\GT_{\diag(\lambda)\rel a}\).
Then:
\begin{enumerate}
\item[(i)] If \(a\in\mathbb R^n\), then the Schur--Horn relation
\(\diag(\lambda)\rel a\) holds if and only if
\(\GT_{\diag(\lambda)\rel a}\) is nonempty.
\item[(ii)] The \(\binom n2\)-dimensional volume of
\(\GT_{\diag(\lambda)\rel *}\) is \(V_n(\lambda)/V_n(\tau_n)\).
\item[(iii)] If \(A\) is Haar-uniform on the unitary orbit with spectrum
\(\lambda\), then the eigenvalues of the northwest principal minors of
\(A\) form a uniformly distributed point of
\(\GT_{\diag(\lambda)\rel *}\). Its boundary vector \(a\) is the diagonal of
\(A\).
\item[(iv)] If \(\Lambda\in\Spec_n+\mathbb R\one\) has large gaps,
\[
        \min_{1\leq i<n}(\Lambda_i-\Lambda_{i+1})
        >
        \lambda_1-\lambda_n,
\]
then for every \(a\in\mathbb R^n\) there is a volume-preserving linear
bijection
\[
        \GT_{\diag(\lambda)\rel a}
        \longleftrightarrow
        H_n(\Lambda,\lambda;\Lambda+a).
\]
Under this bijection, a Gelfand-Tsetlin pattern
\((\lambda^{(k)}_i)_{1\leq i\leq k\leq n}\) is sent to the hive \(h:T_n\to
\mathbb R\) given by
\[
        h(i,j)
        =
        \Lambda_1+\cdots+\Lambda_j
        +
        \lambda^{(j)}_1+\cdots+\lambda^{(j)}_i .
\]
\end{enumerate}
\end{proposition}

\begin{figure}[h]
\begin{center}
\begin{minipage}{0.45\linewidth}
\centering
\includegraphics[width=\linewidth]{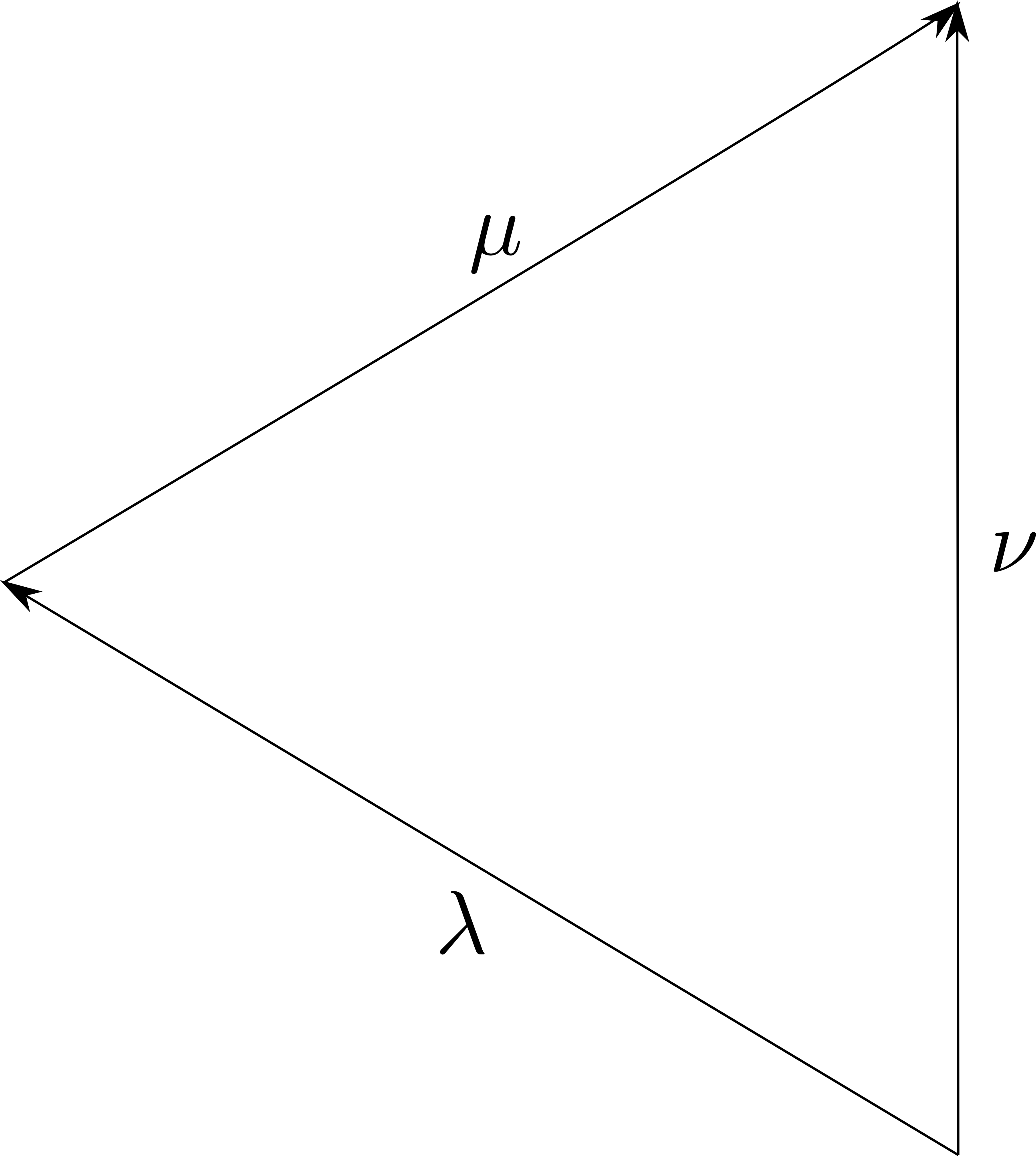}
\end{minipage}
\hfill
\begin{minipage}{0.45\linewidth}
\centering
\includegraphics[width=\linewidth]{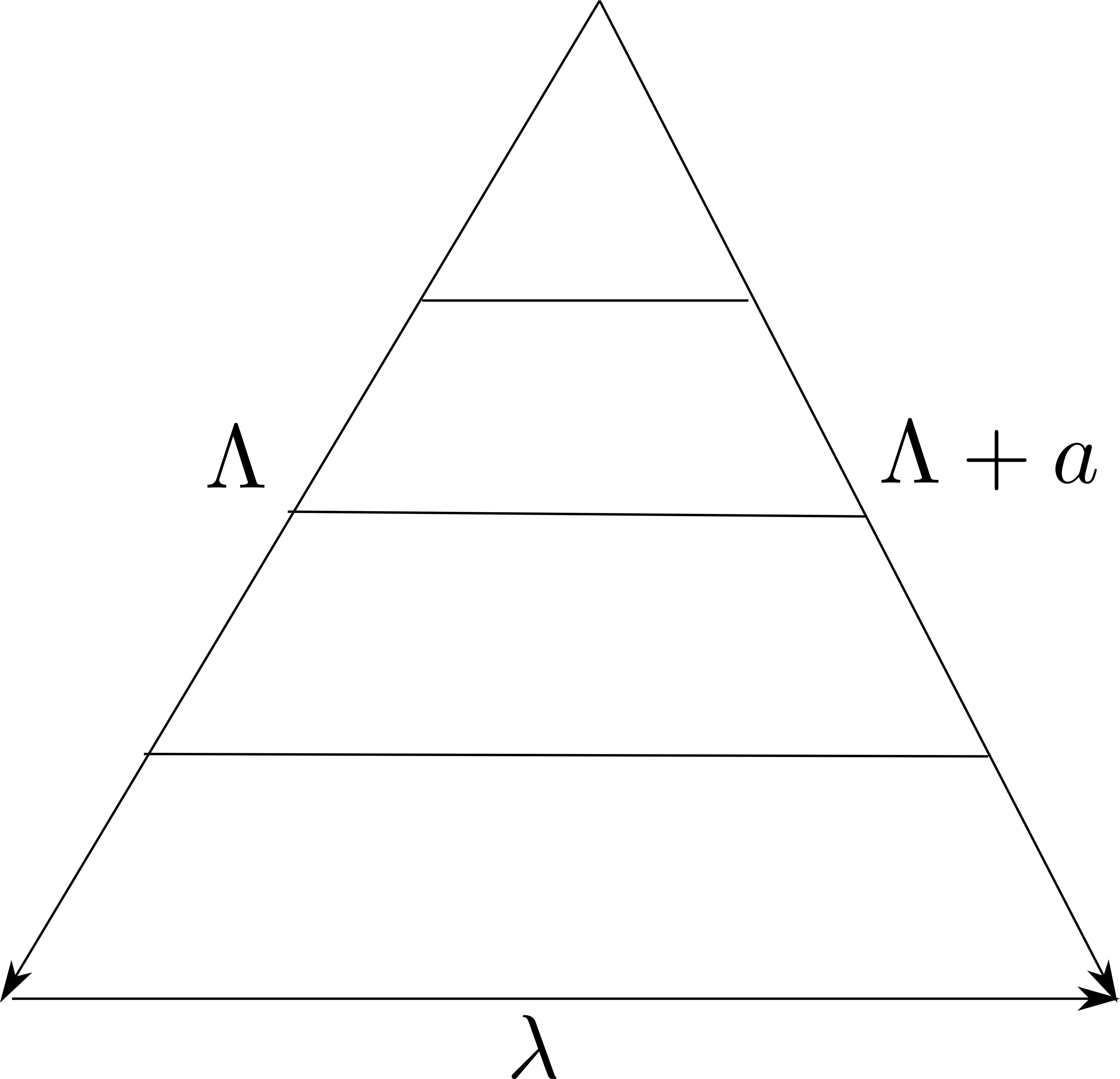}
\end{minipage}
\caption{A hive schematic and the corresponding Gelfand--Tsetlin-to-hive
schematic used in the large-gap realization of
Proposition~\ref{gt-rem}(iv).}
\label{fig:hive-gt-schematics}
\end{center}
\end{figure}

\begin{definition}[Octahedron recurrence]\lab{def:oct}
We use the max-plus form of the octahedron recurrence, in the sense of
\cite{KTW,Speyer,NarSheffTao}.  One starts with real values assigned to a
stepped surface in the parity sublattice of \(\mathbb Z^3\).  A local move
replaces one face of an elementary octahedron by the opposite face.  If the
six vertices of this octahedron are
\[
(i,j,k\pm 1),\qquad (i\pm 1,j,k),\qquad (i,j\pm 1,k),
\]
and all values except \(F(i,j,k+1)\) are already known, then the missing value
is defined by
\[
F(i,j,k+1)
=
\max\{F(i+1,j,k)+F(i-1,j,k),\,
F(i,j+1,k)+F(i,j-1,k)\}
-F(i,j,k-1).
\]
Equivalently,
\[
F(i,j,k+1)+F(i,j,k-1)
=
\max\{F(i+1,j,k)+F(i-1,j,k),\,
F(i,j+1,k)+F(i,j-1,k)\}.
\]
Iterating these local moves transports the data from one stepped surface to
another.  In the hive setting this transport is the
Knutson--Tao--Woodward/Speyer piecewise-linear, volume-preserving bijection
between the two hive decompositions corresponding to associativity
\cite{KTW,Speyer,NarSheffTao}; we denote it by \(\oct\).

For the applications below, Proposition~\ref{gt-rem}(iv) first lets us regard
Gelfand--Tsetlin patterns as hives with one large-gap side.  With this
identification, \(\oct\) sends a pair of Gelfand--Tsetlin patterns, equivalently
a double hive, to an augmented hive, or to a pair of hives glued along their
common boundary. With the input ordering used below, the first output hive
has boundary \((\mu,\lambda;\nu)\). We then apply the standard
volume-preserving hive commutor \cite{HenriquesKamnitzer}, which interchanges
the first two boundary components, to obtain boundary
\((\lambda,\mu;\nu)\). Throughout the paper \(h_n\) denotes this
boundary-ordered hive component when the input is the deformed GUE double
hive. The construction and its measure-preserving properties are detailed
in the proof of Theorem~\ref{thm:main}.
\end{definition}

\begin{figure}[h]
\begin{center}
\begin{minipage}{0.45\linewidth}
\centering
\includegraphics[width=\linewidth]{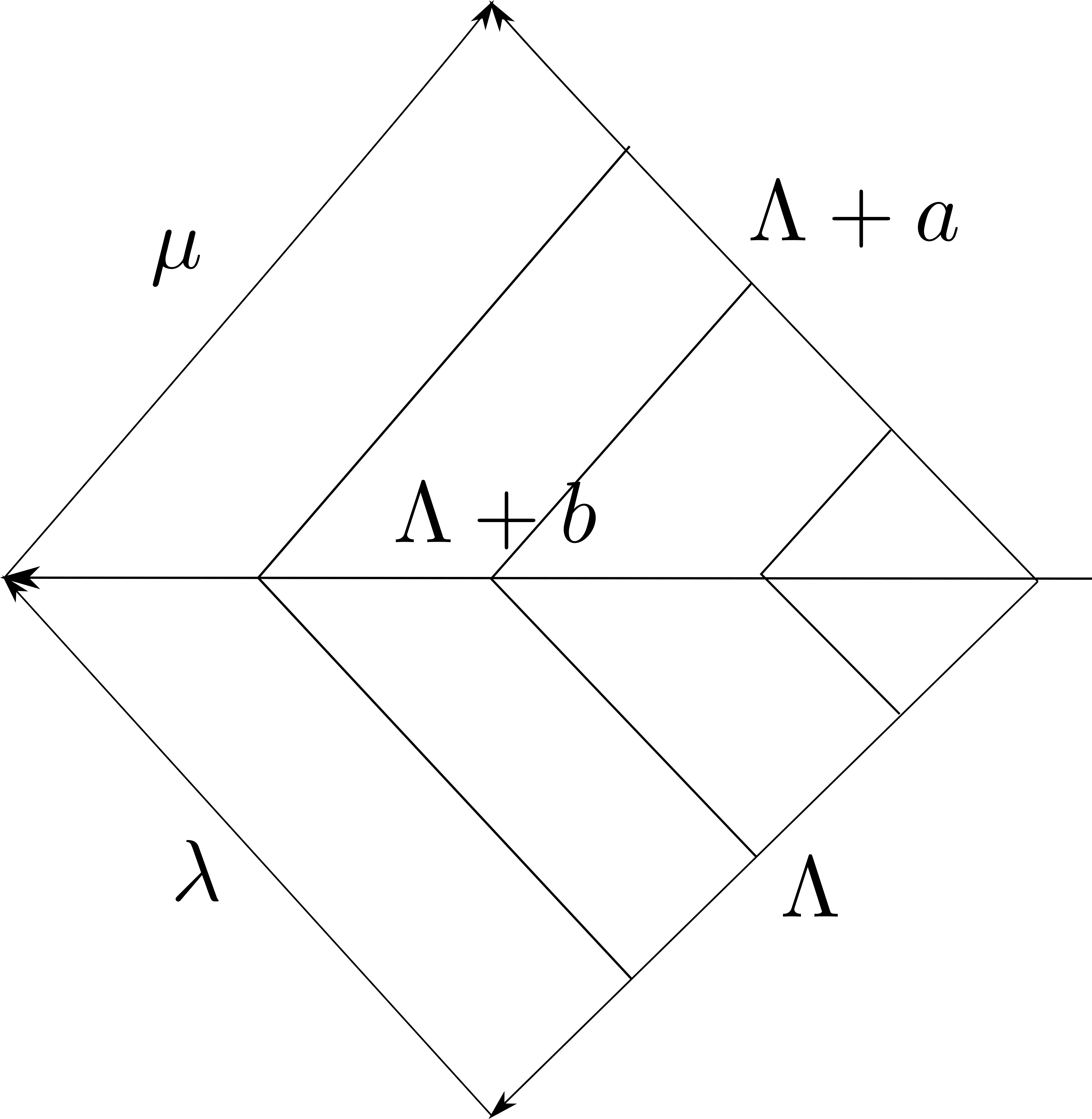}
\caption{The large-gap realization turns the two Gelfand--Tsetlin inputs into
a pair of hives with a common edge before applying the octahedron recurrence.}
\label{fig:excavated}
\end{minipage}
\hfill
\begin{minipage}{0.45\linewidth}
\centering
\includegraphics[width=\linewidth]{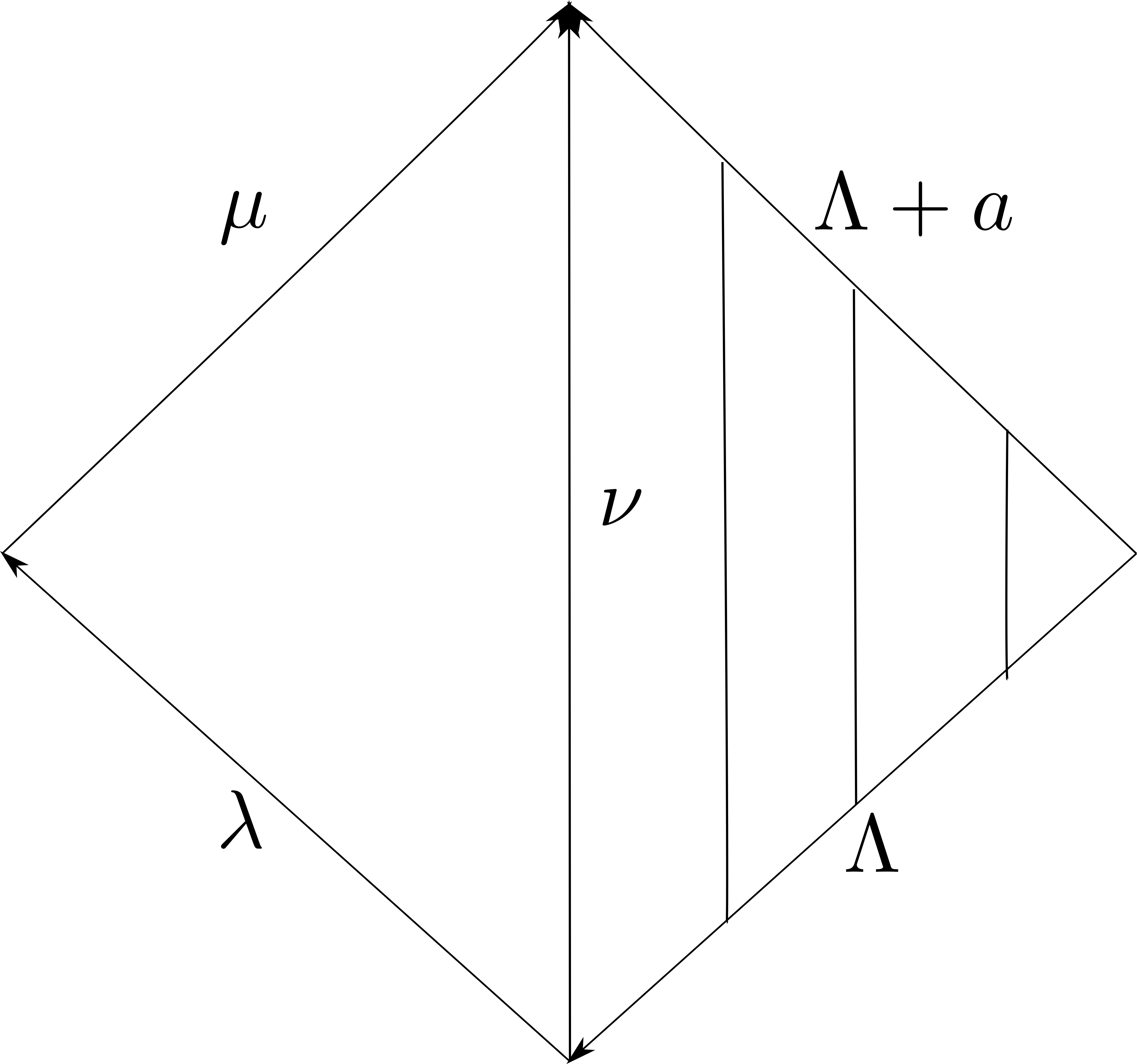}
\caption{An augmented hive, viewed as a hive together with a
Gelfand--Tsetlin pattern.}
\label{fig:augment}
\end{minipage}
\end{center}
\end{figure}

\subsection{GUE hive densities and maximum entropy}
\[
Z(\bar a,\bar b,\bar c)
=
(2\pi)^n
\left(
\frac{\bar a^2\bar b^2\bar c^2}
     {\bar a^2+\bar b^2+\bar c^2}
\right)^{n^2/2}
V_n(\tau_n)^{-2}.
\]

The following appears as Theorem 7 of \cite{GangNar}.
\begin{theorem}[Maximum entropy triply augmented hive]\lab{thm:gaussian1}
Let \(\bar a,\bar b,\bar c\) be parameters in one of the following two regimes:
\[
\bar a,\bar b,\bar c>0,
\]
or
\[
\bar a,\bar b>0,
\qquad
-\bar c^2>\bar a^2+\bar b^2.
\]
Let \(\mathbb A^3\) denote the cone of triply augmented hives, where
\[
\mathbb A^3(\lambda,\mu;\nu)
:=
GT(\lambda)\times GT(\mu)\times GT(\nu)\times H_n(\lambda,\mu;\nu),
\]
for
\[
\lambda,\mu,\nu\in \Spec_n+\mathbb R\one,
\qquad
\sum_i\lambda_i+\sum_i\mu_i=\sum_i\nu_i.
\]
Define a probability density \(p\) on \(\mathbb A^3\), with respect to
Lebesgue measure, by
\[
p(g_\lambda,g_\mu,g_\nu,h)
=
Z^{-1}(\bar a,\bar b,\bar c)
\exp\left[
-\frac12
\left(
\frac{|\lambda|^2}{\bar a^2}
+
\frac{|\mu|^2}{\bar b^2}
+
\frac{|\nu|^2}{\bar c^2}
\right)
\right],
\]
for
\[
(g_\lambda,g_\mu,g_\nu,h)
\in
\mathbb A^3(\lambda,\mu;\nu).
\]
Then \(p\) is the unique probability density on \(\mathbb A^3\) that
maximizes differential entropy among all densities \(q\) satisfying
\[
\mathbb E_q |\lambda|^2
=
\frac{\bar a^2(\bar b^2+\bar c^2)n^2}
     {\bar a^2+\bar b^2+\bar c^2},
\]
\[
\mathbb E_q |\mu|^2
=
\frac{\bar b^2(\bar c^2+\bar a^2)n^2}
     {\bar a^2+\bar b^2+\bar c^2},
\]
and
\[
\mathbb E_q |\nu|^2
=
\frac{\bar c^2(\bar a^2+\bar b^2)n^2}
     {\bar a^2+\bar b^2+\bar c^2}.
\]
\end{theorem}

\begin{remark}[Use of the maximum-entropy statement]\lab{rem:maxent-kl}
In the imaginary-parameter case, we use the maximum-entropy assertion in
Theorem~\ref{thm:gaussian1} only in its exponential-family, or equivalently
KL-duality, form.  Namely, once the displayed exponential density is
normalizable and the three quadratic boundary moments agree, the identity
\[
-\mathbb E_q\log p=-\mathbb E_p\log p=\ent(p)
\]
implies
\[
D_{\mathrm{KL}}(q\Vert p)=\ent(p)-\ent(q)\geq 0.
\]
This argument does not require log-concavity of the density; in particular,
it remains valid when the coefficient of \(|\nu|^2\) has the sign
corresponding to \(\bar c^2<0\).
\end{remark}

When \(s<0\) and \(c^2=s\), let \(GUE_-(s)\) denote the non-probability
Gaussian measure on the real vector space of \(n\times n\) Hermitian matrices
with density
\[
2^{-n/2}\pi^{-n^2/2}|c|^{-n^2}
\exp\left[-\frac{\operatorname{Tr}(X^2)}{2s}\right]\,dX
=
2^{-n/2}\pi^{-n^2/2}|c|^{-n^2}
\exp\left[\frac{\operatorname{Tr}(X^2)}{2|c|^2}\right]\,dX .
\]
Following  \cite[Lemma~1]{GangNar}, its pushforward under
\(X\mapsto\lambda=\spec(X)\) has spectral density
\[
GUE_-(\lambda;c)
:=
(2\pi)^{-n/2}|c|^{-n^2}
\frac{V_n(\lambda)^2}{V_n(\tau_n)}
\exp\left[\frac{|\lambda|^2}{2|c|^2}\right],
\qquad
\lambda\in \Spec_n+\mathbb R\one,
\]
and is zero otherwise.  These objects are used only inside
completion-of-squares identities, as the formal continuation of the usual GUE
convolution formula to a negative variance parameter.

 Following Theorem 6 of \cite{GangNar}  we identify two cases, which however are handled in a unified fashion: \begin{enumerate} \item  $\bar{a}, \bar{b}, \bar{c} > 0$, \item $\bar{a}, \bar{b} > 0$ and $-\bar{c}^2 > \bar{a}^2 + \bar{b}^2$.\end{enumerate}
  Thus, in the second case, $\bar{c}$ is imaginary.
  
  Let
\[
\mathcal A_n
:=
\left\{
(\lambda,\mu,\nu)\in \mathbb R^n\times \mathbb R^n\times \mathbb R^n
:
\sum_{i=1}^n \lambda_i+\sum_{i=1}^n \mu_i
=
\sum_{i=1}^n \nu_i
\right\}.
\]
We equip \(\mathcal A_n\) with the Lebesgue measure \(d\mathfrak m_n\) obtained from the
coordinates
\[
(\lambda_1,\ldots,\lambda_n,\mu_1,\ldots,\mu_n,\nu_1,\ldots,\nu_{n-1}),
\]
with
\[
\nu_n
=
\sum_{i=1}^n \lambda_i+\sum_{i=1}^n \mu_i
-
\sum_{i=1}^{n-1}\nu_i .
\]
That is,
\[
d\mathfrak m_n(\lambda,\mu,\nu)
=
d\lambda_1\cdots d\lambda_n\,
d\mu_1\cdots d\mu_n\,
d\nu_1\cdots d\nu_{n-1}.
\]

\begin{theorem}[Theorem 6, \cite{GangNar}]\lab{thm:gaussian} 
For $\la, \mu, \nu \in \Spec_n + \R\one$, (where $\one$ denotes the vector of all ones), let $$F(\la, \mu, \nu) := \frac{V_n(\lambda)V_n(\mu)V_n(\nu)}{V_n(\tau_n)} |H_n(\lambda, \mu; \nu)| \exp\left(-\frac{1}{2}\left(\frac{|\lambda|^2}{\bar{a}^2} + \frac{|\mu|^2}{\bar{b}^2} + \frac{|\nu|^2}{\bar{c}^2}\right)\right),$$ and otherwise let $F(\la, \mu, \nu) = 0.$ Then,
\[
\int_{\mathcal A_n} F(\lambda,\mu,\nu)\,
d\mathfrak m_n(\lambda,\mu,\nu)
=
(2\pi)^n
\left(
\frac{\bar a^2\bar b^2\bar c^2}{\bar a^2+\bar b^2+\bar c^2}
\right)^{n^2/2}.
\]

\end{theorem}

\begin{definition}
Let the corresponding probability density on hives (sampled uniformly from the normalized Lebesgue measure once the boundary is fixed in the above manner) be denoted $\mathcal{H}_n(a, b, c)$, where $(a, b, c)$ is related to $(\bar a, \bar b, \bar c)$ in the following way as noted on page 17 of \cite{GangNar}:  $a, b, c,$ satisfy  $$a^2  :=  \frac{\bar{a}^2(\bar{b}^2 + \bar{c}^2)}{(\bar{a}^2 + \bar{b}^2 + \bar{c}^2)}, \quad b^2 := \frac{\bar{b}^2(\bar{c}^2 + \bar{a}^2)}{(\bar{a}^2 + \bar{b}^2 + \bar{c}^2)},$$  and
$$c^2 := \frac{\bar{c}^2(\bar{a}^2 + \bar{b}^2)}{(\bar{a}^2 + \bar{b}^2 + \bar{c}^2)}.$$  
If $c^2  > a^2  + b^2,$ then we are in the obtuse case, and $a, b > 0$ and $-\bar{c}^2 > \bar{a}^2 + \bar{b}^2.$ However the same formulae hold in this case as well.
For $c^2=a^2+b^2$, define $\mathcal H_n(a,b,c)$ by the limit
$\bar c^2\to+\infty$ with $\bar a=a$ and $\bar b=b$. Equivalently,
the triply augmented density has quadratic weight
\[
\exp\left[-\frac12\left(\frac{|\lambda|^2}{a^2}
+\frac{|\mu|^2}{b^2}\right)\right]
\]
and normalizing constant $(2\pi)^n(a^2b^2)^{n^2/2}V_n(\tau_n)^{-2}$.
Its hive marginal is the usual independent-GUE sum hive law. This
convention includes the right-angle case in Theorem~\ref{thm:main}.
\end{definition}

\section{Double hives}\lab{sec:double-hives}

\begin{definition}[Double hive]\lab{def:double-hive}
For \(\nu\in\Spec_n\), write
\[
\nu^\vee=(-\nu_n,\ldots,-\nu_1).
\]
A double hive with exterior boundary
\((\alpha,\beta,\eta,\phi)\) and middle side \(\nu\) is a pair of hives
\[
(h_1,h_2)\in
H_n(\alpha,\beta;\nu)\times H_n(\eta,\phi^\vee;\nu^\vee),
\qquad
\phi^\vee=(-\phi_n,\ldots,-\phi_1),
\]
viewed as glued along the common middle side, with opposite orientations on
the two copies of that side.  The corresponding double hive polytope is
\[
\mathcal D(\alpha,\beta,\eta,\phi^\vee)
:=
\bigcup_{\nu\in\Spec_n + \R \one}
H_n(\alpha,\beta;\nu)
\times
H_n(\eta,\phi^\vee;\nu^\vee).
\]
A quadruply augmented double hive is obtained by additionally choosing
Gelfand--Tsetlin patterns on the four exterior boundary components with matching spectra.
\end{definition}

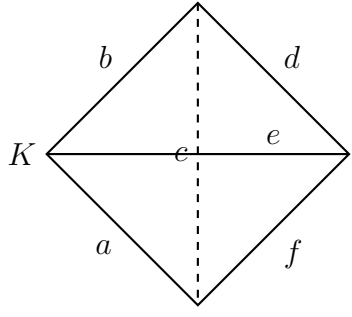
\begin{figure}[h!]
  \centering
  \begin{tikzpicture}[scale=1.0, every node/.style={font=\large}]
    \coordinate (L) at (0,0);
    \coordinate (T) at (2,2);
    \coordinate (R) at (4,0);
    \coordinate (B) at (2,-2);
    \draw[thick] (L) -- (T) -- (R) -- (B) -- cycle;
    \draw[thick] (L) -- (R);
    \draw[thick,dashed] (T) -- (B);
    \node[left] at (L) {\(K\)};
    \node[below left] at (1,-1) {\(a\)};
    \node[above left] at (1,1) {\(b\)};
    \node[above right] at (3,1) {\(d\)};
    \node[below right] at (3,-1) {\(f\)};
    \node[left] at (2,0) {\(c\)};
    \node[above] at (3,0) {\(e\)};
  \end{tikzpicture}
  \caption{The labelling of the tetrahedron. The vertical dashed line is labeled $c$ and  the horizontal solid line is labeled $e$.}
  \label{fig:example}
\end{figure}

Let $\perp_{K, c}$ denote the length of the perpendicular from $K$ to the side $c$.
 Let $\diamondplus_{abfd}$ denote the maximum volume that a tetrahedron can have, with the sides $a, b, f, d$ as prescribed in Figure~\ref{fig:example}.
\begin{equation}\label{eq:tetrahedral-identity}
\max_c
\frac{\Delta_{abc}^2}{abc}
\frac{\Delta_{cdf}^2}{cdf}
=
\max_c
\frac{\Delta_{abc}^2\Delta_{cdf}^2}{abfd\,c^2}
=
\max_c
\frac{|\perp_{K,c}|^2\Delta_{cdf}^2}{4abfd}
=
\frac{9}{4abfd}\diamondplus_{abfd}^2.
\end{equation}

The construction in the following lemma is not general enough to give rise to all \(a,b,d,f\) that can be the side lengths of a tetrahedron as in Figure~\ref{fig:example}, but provides some intuition on where the tetrahedral identity \eqref{eq:tetrahedral-identity} arises from in the context of Theorem~\ref{thm:main}.
\begin{lemma}\lab{lem:doublehive}
Let \(a_{\mathrm{raw}},b_{\mathrm{raw}},d_{\mathrm{raw}},f_{\mathrm{raw}}>0\), and let
\[
\alpha\sim N(0,a_{\mathrm{raw}}^2),\qquad
\beta\sim N(0,b_{\mathrm{raw}}^2),\qquad
\eta \sim N(0,d_{\mathrm{raw}}^2),\qquad
\phi\sim N(0,f_{\mathrm{raw}}^2)
\]
be independent, and condition on
\[
\alpha+\beta+\eta-\phi=0.
\]
Set
\[
t=a_{\mathrm{raw}}^2+b_{\mathrm{raw}}^2+d_{\mathrm{raw}}^2+f_{\mathrm{raw}}^2.
\]
Define \(a,b,d,f\) by
\[
a^2=\operatorname{Var}(\alpha\mid \alpha+\beta+\eta-\phi=0),
\]
\[
b^2=\operatorname{Var}(\beta\mid \alpha+\beta+\eta-\phi=0),
\]
\[
d^2=\operatorname{Var}(\eta\mid \alpha+\beta+\eta-\phi=0),
\]
and
\[
f^2=\operatorname{Var}(\phi\mid \alpha+\beta+\eta-\phi=0).
\]
Also set
\[
c_*^2=\operatorname{Var}(\alpha+\beta\mid \alpha+\beta+\eta-\phi=0).
\]
Then
\[
\frac{\Delta_{abc_*}^2}{abc_*}
\frac{\Delta_{c_*df}^2}{c_*df}
=
\frac{
9|\mathtt{tet}|_{abfd}^2
}{
4abfd
}.
\]
\end{lemma}

\par\smallskip
\noindent\hyperref[proof:lem:doublehive]{Proof of Lemma~\ref*{lem:doublehive}.}
\par\smallskip

\begin{lemma}[Maximum entropy quadruply augmented double hive]\lab{lem:double-maxent}
Let \(\bar a,\bar b,\bar d>0\).  Assume either
\[
\bar f > 0,
\]
or
\[
\bar f^2<0,
\qquad
T:=\bar a^2+\bar b^2+\bar d^2+\bar f^2<0.
\]
Let \(\mathbb D^4\) be the cone of double hives
augmented on the four exterior boundary components:
\[
\mathbb D^4
=
\bigcup_{\alpha,\beta,\eta,\phi,\nu}
GT(\alpha)\times GT(\beta)\times GT(\eta)\times GT(\phi)
\times H_n(\alpha,\beta;\nu)
\times H_n(\eta,\phi^\vee;\nu^\vee),
\]
where
\[
\alpha,\beta,\eta,\phi\in \Spec_n+\mathbb R\one,
\qquad
\sum_i\alpha_i+\sum_i\beta_i+\sum_i\eta_i-\sum_i\phi_i=0,
\]
and
\[
\nu^\vee=(-\nu_n,\ldots,-\nu_1),\qquad
\phi^\vee=(-\phi_n,\ldots,-\phi_1).
\]
Thus the fiber volume over
\((\alpha,\beta,\eta,\phi,\nu)\) contains the exterior factor
\[
\frac{V_n(\alpha)V_n(\beta)V_n(\eta)V_n(\phi)}
     {V_n(\tau_n)^4}
|H_n(\alpha,\beta;\nu)|\,|H_n(\eta,\phi^\vee;\nu^\vee)|
\]

Equip \(\mathbb D^4\) with coordinate Lebesgue measure: use
\[
d\alpha_1\cdots d\alpha_n\,
d\beta_1\cdots d\beta_n\,
d\eta_1\cdots d\eta_n\,
d\phi_1\cdots d\phi_{n-1}\,
d\nu_1\cdots d\nu_{n-1}
\]
on its boundary variables, where trace compatibility determines
\(\nu_n\) and \(\phi_n\), together with the product of the usual
Gelfand--Tsetlin row coordinates and the interior hive coordinates on
each fiber.  This is the reference measure for all double-hive densities
and entropies below.  Define a density \(p_{\mathrm{dbl}}\) with respect to
this measure by
\[
p_{\mathrm{dbl}}(x)
=
Z_{\mathrm{dbl}}^{-1}
\exp\left[
-\frac12
\left(
\frac{|\alpha(x)|^2}{\bar a^2}
+
\frac{|\beta(x)|^2}{\bar b^2}
+
\frac{|\eta(x)|^2}{\bar d^2}
+
\frac{|\phi(x)|^2}{\bar f^2}
\right)
\right].
\]
Then \(p_{\mathrm{dbl}}\) is the unique probability density on
\(\mathbb D^4\) maximizing differential entropy among all densities \(q\)
such that
\[
\mathbb E_q|\alpha|^2=\mathbb E_{p_{\mathrm{dbl}}}|\alpha|^2,\qquad
\mathbb E_q|\beta|^2=\mathbb E_{p_{\mathrm{dbl}}}|\beta|^2,
\]
\[
\mathbb E_q|\eta|^2=\mathbb E_{p_{\mathrm{dbl}}}|\eta|^2,\qquad
\mathbb E_q|\phi|^2=\mathbb E_{p_{\mathrm{dbl}}}|\phi|^2.
\]
\end{lemma}

\par\smallskip
\noindent\hyperref[proof:lem:double-maxent]{Proof of Lemma~\ref*{lem:double-maxent}.}
\par\smallskip

Recall that \(\diamondplus_{abdf}\) denotes the maximum volume of a
Euclidean tetrahedron with four prescribed cyclic edge lengths
\(a,b,d,f\), as in Figure~\ref{fig:example}.  Equivalently,
\[
\max_x
\frac{\Delta_{abx}^2}{abx}
\frac{\Delta_{xdf}^2}{xdf}
=
\frac{9\diamondplus_{abdf}^2}{4abdf}.
\]

\begin{lemma}[Tetrahedral analogue of Lemma 5 in \cite{GangNar}]\lab{lem:tet-analogue of GangNar}
Let \(\bar a,\bar b,\bar d>0\), and assume either \(\bar f>0\), or \(\bar f^2<0\) and
\[
T:=\bar a^2+\bar b^2+\bar d^2+\bar f^2<0.
\]
Define
\[
a^2=\frac{\bar a^2(\bar b^2+\bar d^2+\bar f^2)}{T},\qquad
b^2=\frac{\bar b^2(\bar a^2+\bar d^2+\bar f^2)}{T},
\]
\[
d^2=\frac{\bar d^2(\bar a^2+\bar b^2+\bar f^2)}{T},\qquad
f^2=\frac{\bar f^2(\bar a^2+\bar b^2+\bar d^2)}{T}.
\]
Then
\[
\frac{\bar a^2\bar b^2\bar d^2\bar f^2}{\bar a^2+\bar b^2+\bar d^2+\bar f^2}
=
36\diamondplus_{abdf}^2.
\]

\end{lemma}

\par\smallskip
\noindent\hyperref[proof:lem:tet-analogue of GangNar]{Proof of Lemma~\ref*{lem:tet-analogue of GangNar}.}
\par\smallskip

\begin{lemma}[Quadruply augmented double-hive partition function]\lab{lem:double-partition}
Let \(\bar a^2,\bar b^2,\bar d^2>0\), put
\[
\bar T:=\bar a^2+\bar b^2+\bar d^2+\bar f^2,
\]
and assume either \(\bar f^2>0\), or
\[
\bar f^2<0,
\qquad \bar T<0.
\]
Set
\[
A=n\bar a^2,\qquad B=n\bar b^2,\qquad
D=n\bar d^2,\qquad F=n\bar f^2,
\]
and let \(p_{\mathrm{dbl}}\) be the maximum-entropy density on
\(\mathbb D^4\) from Lemma~\ref{lem:double-maxent}, with raw squared
parameters \(A,B,D,F\).  Let \(a,b,d,f\) be the geometric side lengths
associated to the unscaled squared parameters
\(\bar a^2,\bar b^2,\bar d^2,\bar f^2\), as in
Lemma~\ref{lem:tet-analogue of GangNar}.  Then
the normalizing constant 
\[Z_{\mathrm{dbl}}
=
\frac{(2\pi)^{3n/2}}{V_n(\tau_n)^3}
\left(
\frac{ABDF}{A+B+D+F}
\right)^{n^2/2}
=\frac{(2\pi)^{3n/2}}{V_n(\tau_n)^3}
\left(36n^3|\mathtt{tet}|_{abfd}^2\right)^{n^2/2}.
\]

Consequently,
\[
\ent(p_{\mathrm{dbl}})
=
\frac{n^2}{2}
\log\left(
\frac{\bar a^2\bar b^2\bar d^2\bar f^2}{\bar T}
\right)
+\frac{15}{4}n^2
+O(n\log n).
\]
\end{lemma}

\par\smallskip
\noindent\hyperref[proof:lem:double-partition]{Proof of Lemma~\ref*{lem:double-partition}.}
\par\smallskip

\begin{lemma}[One-hive marginal of the  quadruply augmented double hive law]\lab{lem:double-one-hive}
Assume \(\bar a,\bar b,\bar d>0\), and assume either \(\bar f>0\), or
\[
\bar f^2<0,
\qquad
\bar a^2+\bar b^2+\bar d^2+\bar f^2<0.
\]
Let \(p_{\mathrm{dbl}}\) be the density on \(\mathbb D^4\) from
Lemma~\ref{lem:double-maxent}.  Put
\[
T=\bar a^2+\bar b^2+\bar d^2+\bar f^2.
\]
Define the geometric side lengths
\[
a^2=\frac{\bar a^2(T-\bar a^2)}{T},\qquad
b^2=\frac{\bar b^2(T-\bar b^2)}{T},
\]
\[
d^2=\frac{\bar d^2(T-\bar d^2)}{T},\qquad
f^2=\frac{\bar f^2(T-\bar f^2)}{T},
\]
and
\[
c_*^2=
\frac{(\bar a^2+\bar b^2)(\bar d^2+\bar f^2)}{T}.
\]
In the second case these quantities are still positive, since
\(\bar d^2+\bar f^2<-(\bar a^2+\bar b^2)\).
Let \(h_1\) denote the first hive, so that
\[
h_1\in H_n(\alpha,\beta;\nu).
\]
Then the marginal law of \(h_1\) is
\[
\mathcal H_n(a,b,c_*).
\]
Similarly, the marginal law of the second hive is
\(\mathcal H_n(d,f,c_*)\); in the case \(\bar f^2<0\), this is understood
using the symmetry of the three boundary components to place the imaginary
raw parameter in the third slot of Theorem~\ref{thm:gaussian1}.  Moreover,
\(c_*\) is the value of the glued side selected by Lemma~\ref{lem:doublehive}:
it maximizes
\[
x\mapsto
\frac{\Delta_{abx}^2}{abx}\frac{\Delta_{xdf}^2}{xdf}.
\]
\end{lemma}

\par\smallskip
\noindent\hyperref[proof:lem:double-one-hive]{Proof of Lemma~\ref*{lem:double-one-hive}.}
\par\smallskip

\begin{lemma}[Weyl integration formula for Hermitian matrices; see {\cite[Sec.~3.1]{Mehta}}]\lab{lem:Weyl}
Let $ \Herm_n$ denote the real vector space of $n\times n$
complex Hermitian matrices, equipped with Lebesgue measure
\[
  dH
  = \prod_{i=1}^n dH_{ii}
    \prod_{1\leq i<j\leq n}
    d(\operatorname{Re}H_{ij})\,d(\operatorname{Im}H_{ij}).
\]
Let $dU$ be normalized Haar probability measure on $\mathrm U(n)$.
We equip \(\Spec_n+\mathbb R\one\) with the induced \(n\)-dimensional
Lebesgue measure \(d\lambda\).
Then, for every nonnegative measurable function
$F\colon\mathcal H_n\to[0,\infty]$ (and hence for every integrable
function $F$),
\[
	  \int_{\Herm_n} F(H)\,dH
	  =
	  \frac{\pi^{n(n-1)/2}}{V_n(\tau_n)}
	  \int_{\Spec_n+\mathbb R\one}\int_{\mathrm U(n)}
	  F\!\left(U\operatorname{diag}(\lambda_1,\ldots,\lambda_n)U^*\right)
	  V_n(\lambda)^2\,dU\,d\lambda,
	\]
	where \(V_n\) is the sorted Vandermonde defined above.  Since
	\(\lambda\in\Spec_n+\mathbb R\one\), this is simply
	\(\prod_{1\leq i<j\leq n}(\lambda_i-\lambda_j)\).

In particular, if $F$ is invariant under unitary conjugation, then
\[
	  \int_{\Herm_n} F(H)\,dH
	  =
	  \frac{\pi^{n(n-1)/2}}{V_n(\tau_n)}
	  \int_{\Spec_n+\mathbb R\one}
	  F\!\left(\operatorname{diag}(\lambda_1,\ldots,\lambda_n)\right)
	  V_n(\lambda)^2\,d\lambda.
	\]
\end{lemma}
\begin{proposition}[Middle side of a double hive]\lab{prop:middle}
Let
\[
        \alpha=\spec(A),\qquad
        \beta=\spec(B),\qquad
        \eta=\spec(E),\qquad
        \phi=\spec(F)
\]
belong to \(\Spec_n^\circ\), and assume the trace compatibility condition
\[
        \sum_i\alpha_i+\sum_i\beta_i+\sum_i\eta_i - \sum_i\phi_i=0 .
\]
Let \(U_1,U_2,U_3,U_4\) be independent Haar-distributed unitary matrices and set
\[
        S=U_1AU_1^*+U_2BU_2^*,
        \qquad
        T=U_3EU_3^*-U_4FU_4^* .
\]
For \(\nu\in\Spec_n\), write
\[
        \nu^\vee=(-\nu_n,\ldots,-\nu_1),
\]
so that \(\nu^\vee=\spec(-X)\) whenever \(\nu=\spec(X)\).
Suppose that the double hive polytope
\[
        \mathcal D(\alpha,\beta,\eta,\phi^\vee)
        :=
        \bigcup_{\nu\in\Spec_n}
        H_n(\alpha,\beta;\nu)
        \times
        H_n(\eta,\phi^\vee;\nu^\vee)
\]
is full dimensional, where the two hives are glued along the common side with values
\(\nu\) on the first hive and \(\nu^\vee\) on the second. 
 If a double hive is
sampled from normalized Lebesgue measure on
\(\mathcal D(\alpha,\beta,\eta,\phi^\vee)\), then its middle side has the same
distribution as the density-defined regular conditional distribution
(with the version at zero specified in the proof) of
\[
        \spec(S)
        =
        \spec(U_1AU_1^*+U_2BU_2^*)
\]
given the matrix equation
\[
        S+T=0 .
\]
Equivalently,
\[
        \spec(\text{middle side})
        \stackrel{d}{=}
        \spec(U_1AU_1^*+U_2BU_2^*)
        \,\big|\,
        U_1AU_1^*+U_2BU_2^*
        +U_3EU_3^*-U_4FU_4^*=0 .
\]
\end{proposition}

\par\smallskip
\noindent\hyperref[proof:prop:middle]{Proof of Proposition~\ref*{prop:middle}.}
\par\smallskip

\section{Hives from deformed GUE minor processes}\lab{sec:deformed-gue-hives}
\begin{notation} We say that a sequence of events $(E_n)_{n \geq 1}$ occurs with overwhelming probability as $n \ra \infty$ if for every positive constant $C$,  the probability of $E_n$ is  $ 1 - O(n^{-C})$.
\end{notation}
The following is the eigenvalue-rigidity estimate of Lemma 6 in
\cite{NarSheffTao}, with the classical locations written in the
normalization used here.
\begin{lemma}[Eigenvalue rigidity]\lab{lem:rigidity}
Let \(A\) be a matrix such that \(A/\sqrt n\) has the distribution of GUE.
Let
\[
\lambda_1 \leq \lambda_2 \leq \cdots \leq \lambda_n
\]
denote its eigenvalues, and let \(\gamma_i\) be the corresponding classical
location for the semicircle law, defined by
\[
\int_{-\infty}^{\gamma_i}
\frac{1}{2\pi}\sqrt{(4-x^2)_+}\,dx
=
\frac{i}{n}.
\]
Then for every \(1\leq i\leq n\),
\[
\mathbb P\left(
n^{-1/3}\min(i,n-i+1)^{1/3}
\left|\lambda_i-n\gamma_i\right|
\geq T
\right)
\ll
n^{O(1)}\exp(-cT^c)
\]
for all \(T>0\), where \(c>0\) is an absolute constant.

In particular, with overwhelming probability,
\[
\lambda_i
=
n\gamma_i
+
O\!\left(
n^{1/3}\min(i,n-i+1)^{-1/3}\log^{O(1)} n
\right),
\]
and the same estimate holds with \(\lambda_i\) replaced by
\(\mathbb E\lambda_i\).
\end{lemma}

\subsection{Deformed GUE inputs}\lab{ssec:4.1}
At the base scale, consider matrices
\[
        X=wG+ u D,\qquad Y=w'G'+ u'D',
\]
where \(G,G'\) are independent GUE matrices and \(D,D'\) are diagonal matrices with non-increasing eigenvalues $\la_{11} \geq \dots \geq \la_{nn}$, with spectrum that respectively equal to the spectrum of two GUE matrices $(\tilde{G}, \tilde{G}')$  such that $(G, G', \tilde{G}, \tilde{G}')$ are all independent. 
In the finite-\(n\) hive construction below, the actual inputs are the
scaled matrices \(\sqrt n\,X\) and \(\sqrt n\,Y\).

We initialize the upper double hive using the two
deformed minor processes.  We then apply the octahedron recurrence (see \cite{NarSheffTao}) to obtain an augmented hive $(h_n, g_n)$, where $h_n$ is a hive and $g_n$ is a Gelfand-Tsetlin pattern.   

\begin{lemma}[Fixed spectrum and fixed diagonal version]\lab{lem:fixed-spectrum-fixed-diagonal}
Let \(w>0\) and \(u\geq 0\), and let
\[
        X=wG+uD,
\]
where \(G\) is a GUE matrix and \(D=\operatorname{diag}(d_1,\ldots,d_n)\) is
fixed.  Let
\[
        \lambda=\spec(X),
        \qquad
        a=\diag(X).
\]
For \(1\leq k\leq n\), let \(\lambda^{(k)}\) be the ordered eigenvalue vector
of the \(k\times k\) northwest principal minor of \(X\), and write
\[
        \Gamma=(\lambda^{(1)},\ldots,\lambda^{(n)}).
\]
Then, conditional on both
\[
        \spec(X)=\lambda
        \qquad\text{and}\qquad
        \diag(X)=a,
\]
the minor process \(\Gamma\) is uniformly distributed, with respect to
Lebesgue measure, on the Gelfand-Tsetlin fiber
\[
        \GT_{\diag(\lambda)\rel a}.
\]
Equivalently, its conditional density is
\[
        \frac{1}{|\GT_{\diag(\lambda)\rel a}|}
        \mathbf 1_{\{\Gamma\in \GT_{\diag(\lambda)\rel a}\}} .
\]
\end{lemma}

\par\smallskip
\noindent\hyperref[proof:lem:fixed-spectrum-fixed-diagonal]{Proof of Lemma~\ref*{lem:fixed-spectrum-fixed-diagonal}.}
\par\smallskip

\begin{lemma}\lab{prop:3}
Let
\[
X_n=\sqrt n\,(wG_n+uD_n),
\qquad
D_n=\operatorname{diag}(d_1,\ldots,d_n),
\]
where \(G_n\) is GUE, \(D_n\) is independent of \(G_n\), and
\((d_1,\ldots,d_n)\) is the ordered spectrum of an independent GUE
matrix. Let \(\Gamma_n\) be the complete eigenvalue minor process of
\(X_n\). If \(w>0\) and \(u\geq 0\) are fixed, then
\[
I(D_n;\Gamma_n)=O(n \log \log n).
\]
\end{lemma}

\par\smallskip
\noindent\hyperref[proof:prop:3]{Proof of Lemma~\ref*{prop:3}.}
\par\smallskip

\begin{lemma}[Expected logarithmic Vandermonde of the deformed GUE]\lab{lem:12}
Let \(G_n\) and \(\widetilde G_n\) be independent \(n\times n\)
GUE matrices, normalized so that
\[
\mathbb E\operatorname{Tr}(G_n^2)=n^2.
\]
Let
\[
D_n=\operatorname{diag}
\bigl(\operatorname{spec}(\widetilde G_n)\bigr),
\qquad
Z_n=wG_n+uD_n,
\]
where \(w\in\mathbb R\) and \(u\geq 0\) are fixed, and
\[
a=\sqrt{w^2+u^2}>0.
\]
If
\[
V_n(\lambda)
=
\prod_{1\leq i<j\leq n}|\lambda_i-\lambda_j|,
\]
then
\[
\mathbb E\log V_n(\operatorname{spec}Z_n)
=
\frac{n^2}{4}\log n
+\frac{n^2}{2}\log a
-\frac{n^2}{8}
+O(n\log n).
\]

The implicit constant may depend on \(w\) and \(u\), but not on \(n\).
\end{lemma}

\par\smallskip
\noindent\hyperref[proof:lem:12]{Proof of Lemma~\ref*{lem:12}.}
\par\smallskip

\subsection{Entropy of a deformed GUE minor process.}\lab{ssec:4.2}

\begin{proposition}\lab{prop:4}
Let \(w>0\) and \(u\geq 0\), and let
\[
X_n=\sqrt{n}\,(wG_n+uD_n),
\qquad
a=\sqrt{w^2+u^2},
\]
where \(G_n\) is a GUE matrix and \(D_n\) is diagonal with spectrum
distributed as that of an independent GUE matrix, sorted in non-increasing order. Let \(\Gamma_n\)
denote the complete eigenvalue minor process of \(X_n\), regarded as a
random point in its
\[
m_n=\frac{n(n+1)}{2}
\]
Gelfand-Tsetlin coordinates. Then the asymptotic differential
entropy is
\[
\mathrm{ent}(\Gamma_n)
=
\frac{n^2}{2}
\left(
\frac54+
\log\left(\frac{w^2}{\sqrt{w^2+u^2}}\right)
\right)
+O(n\log n).
\]
\end{proposition}

\par\smallskip
\noindent\hyperref[proof:prop:4]{Proof of Proposition~\ref*{prop:4}.}
\par\smallskip

\begin{lemma}[One-sided Vandermonde comparison]\lab{lem:one-sided-vandermonde}
For each $i < j$, let
\[
\widetilde\lambda_n(i) - \widetilde\lambda_n(j)
\geq
\left(1-\frac1n\right)(\lambda_n(i) - \lambda_n(j)),
\]
where both \(\lambda_n^{\mathrm{cl}}\) and \(\lambda_n\) are ordered
decreasingly. Then
\[
\log V_n(\widetilde\lambda_n)
\geq
\log V_n(\lambda_n) - O(n).
\]
Consequently,
\[
\mathbb E\log V_n(\widetilde\lambda_n)
\geq
\mathbb E\log V_n(\lambda_n) - O(n).
\]
\end{lemma}

\par\smallskip
\noindent\hyperref[proof:lem:one-sided-vandermonde]{Proof of Lemma~\ref*{lem:one-sided-vandermonde}.}
\par\smallskip

Let \(G_n,\widetilde G_n, G'_n,\widetilde G'_n,  \check{G}_n\) be independent GUE matrices and let
\[
D_n=\operatorname{diag}(\operatorname{spec}(\widetilde G_n)),
\qquad
D'_n=\operatorname{diag}(\operatorname{spec}(\widetilde G'_n)).
\]
Fix \(w,w'>0\) and \(u,u'\geq 0\), set \(a=\sqrt{w^2+u^2}\), and set
\[
X_n=\sqrt n\,(wG_n+uD_n), \qquad
Y_n=\sqrt n\,(w'G'_n+u'D'_n),
\qquad
\check{\lambda}_n=\operatorname{spec}(\sqrt n\,\check{G}_n).
\]
Set \(\lambda_n = \operatorname{spec}(X_n)\) and
\(\lambda'_n = \operatorname{spec}(Y_n)\).

Let \(\lambda_n^{\mathrm{cl}}\) be the deterministic classical spectrum
for \(\sqrt{n} G_n\), ordered decreasingly, and define
\beq\lab{eq:large-enough-gaps}
\widetilde\lambda_n
=
\frac{1}{n}\lambda_n^{\mathrm{cl}}
+
\left(1-\frac1n\right)\check{\lambda}_n.
\eeq
The two terms in \eqref{eq:large-enough-gaps} have separate roles.
The deterministic classical spectrum \(\lambda_n^{\mathrm{cl}}\) gives a
deterministic lower bound on the gaps of \(\widetilde\lambda_n\), while the
independent GUE spectrum \(\check\lambda_n\) gives a clean entropy and
Vandermonde lower bound for the regularized side.  The exceptional-event
correction below makes the side large enough for both Gelfand--Tsetlin
patterns to embed into the relevant hive polytopes on every outcome.
Indeed, since both summands in \(\widetilde\lambda_n\) are ordered
decreasingly, for every \(i<j\),
\[
\widetilde\lambda_n(i)-\widetilde\lambda_n(j)
\geq
\left(1-\frac1n\right)
\bigl(\check\lambda_n(i)-\check\lambda_n(j)\bigr),
\]
and hence
\[
V_n(\widetilde\lambda_n)
\geq
\left(1-\frac1n\right)^{\binom n2}
V_n(\check\lambda_n).
\]
Thus the Vandermonde contribution of the artificial large side \(L_n\) can
be bounded below using the standard GUE Vandermonde asymptotics.

Set
\[
L_n^{(0)}=n^{10} \widetilde\lambda_n,
\qquad
g_n=n^9
\min_{1\leq i<n}
\bigl(\lambda_n^{\mathrm{cl}}(i)-\lambda_n^{\mathrm{cl}}(i+1)\bigr),
\]
and put
\[
\mathcal R_n
=
\max\bigl\{
\lambda_n(1)-\lambda_n(n),
\lambda'_n(1)-\lambda'_n(n)
\bigr\},
\qquad
A_n=(\mathcal R_n-g_n+1)_+.
\]
Recall that \(\tau_n\) has unit adjacent gaps, and define the corrected
large-gap side by
\begin{equation}\lab{eq:corrected-large-gap}
L_n=L_n^{(0)}+A_n\tau_n.
\end{equation}
Since every summand in \(\widetilde\lambda_n\) is decreasing,
\[
\min_{1\leq i<n}\bigl(L_n^{(0)}(i)-L_n^{(0)}(i+1)\bigr)
\geq g_n.
\]
If \(A_n=0\), then \(\mathcal R_n\leq g_n-1\), while if \(A_n>0\), then
\(g_n+A_n=\mathcal R_n+1\).  Thus, on every outcome,
\begin{equation}\lab{eq:large-gap-every-outcome}
\min_{1\leq i<n}\bigl(L_n(i)-L_n(i+1)\bigr)
\geq \mathcal R_n+1.
\end{equation}
In particular, Proposition~\ref{gt-rem}(iv) applies simultaneously to both
minor processes on the whole probability space.  Moreover, \(g_n\asymp n^9\),
whereas the two spectral ranges have Gaussian tails.  Indeed,
\(\operatorname{range}(H)\leq2\|H\|_{\mathrm{op}}\leq2\|H\|_{\mathrm F}\),
and therefore \(\mathcal R_n\) is bounded by \(C\sqrt n\) times the sum of
the Frobenius norms of the four independent GUE matrices defining \(X_n\)
and \(Y_n\).  These Frobenius norms are Euclidean norms of Gaussian vectors
of dimension \(n^2\).  Since
\[
A_n\leq(1+\mathcal R_n)
\mathbf 1_{\{\mathcal R_n>g_n-1\}},
\]
their standard Gaussian tail bound and integration of the tail give, for every fixed
\(k\geq1\), \(m\geq0\), and every \(K>0\),
\begin{equation}\lab{eq:exceptional-correction-tail}
\E\left[A_n^k(1+\mathcal R_n)^m\right]=O(n^{-K})
.
\end{equation}
Thus the correction is zero with overwhelming probability and contributes
less than any inverse power of \(n\) to all fixed polynomial moments used
below.

Define the two remaining hive boundaries by
\[
M_n=L_n,
\qquad
N_n=L_n+\operatorname{diag}(X_n).
\]
Then
\[
\sum_i\lambda_n(i)+\sum_i M_n(i)=\sum_iN_n(i).
\]

\begin{lemma}[Entropy transfer through the random large-gap side]
\lab{lem:random-side-entropy}
For \(n\geq2\), let \(\Gamma_n^X\) and \(\Gamma_n^Y\) be the complete eigenvalue minor
processes of \(X_n\) and \(Y_n\), respectively, and let
\(R_n=\check\lambda_n\).  With the coordinate Lebesgue measures in which
Proposition~\ref{gt-rem}(iv) is volume-preserving, the following assertions
hold.
\begin{enumerate}
\item The auxiliary spectrum satisfies
\[
\bigl|\ent(R_n)\bigr|=O(n\log n).
\]
\item For the one-hive construction,
\[
\bigl|\ent(L_n\mid\Gamma_n^X)\bigr|=O(n\log n).
\]
For the double-hive construction one has the exact identity
\[
\ent(L_n\mid\Gamma_n^X,\Gamma_n^Y)
=
\ent(R_n)
+n\log\left(n^{10}\left(1-\frac1n\right)\right).
\]
In particular, this conditional entropy is also \(O(n\log n)\) in
absolute value.
\item The large-gap bijections of Proposition~\ref{gt-rem}(iv), which apply
on every outcome by \eqref{eq:large-gap-every-outcome}, have extensions which
retain the random side \(L_n\) and have block-triangular Jacobian of absolute
determinant one.
Consequently,
\[
\ent(h_n)=\ent(\Gamma_n^X,L_n)
=\ent(\Gamma_n^X)+\ent(L_n\mid\Gamma_n^X)
\]
and
\begin{align*}
&\ent\!\left(h_n^{(1)},h_n^{(2)}
       \hbox{ pasted along }L_n\right)\\
&\qquad=
\ent(\Gamma_n^X,\Gamma_n^Y,L_n)\\
&\qquad=
\ent(\Gamma_n^X)+\ent(\Gamma_n^Y)
+\ent(L_n\mid\Gamma_n^X,\Gamma_n^Y).
\end{align*}
\end{enumerate}
\end{lemma}

\par\smallskip
\noindent\hyperref[proof:lem:random-side-entropy]{Proof of
Lemma~\ref*{lem:random-side-entropy}.}
\par\smallskip

\begin{proposition}[Moments of the corrected large side]
\lab{prop:large-side-moments}
Let
\[
Q_n=\frac{\E|L_n|^2}{n^2}.
\]
Then
\begin{equation}\lab{eq:large-side-Q}
Q_n=n^{21}+O(n^{20}).
\end{equation}
Moreover,
\begin{align}
\frac{\E|\lambda_n|^2}{n^2}&=a^2n,&
\frac{\E|\lambda'_n|^2}{n^2}&=b^2n,\lab{eq:large-side-spectral-moments}\\
\frac{\E\langle L_n,\diag(X_n)\rangle}{n^2}
&=un^{11}+O(n^{10}),&
\frac{\E\langle L_n,\diag(Y_n)\rangle}{n^2}
&=u'n^{11}+O(n^{10}),\lab{eq:large-side-mixed-moments}\\
\frac{\E|\diag(X_n)|^2}{n^2}&=w^2+u^2n,&
\frac{\E|\diag(Y_n)|^2}{n^2}&=(w')^2+(u')^2n.\lab{eq:large-side-diagonal-moments}
\end{align}
Consequently, for
\[
N_n=L_n+\diag(X_n),\qquad P_n=L_n-\diag(Y_n),
\]
one has
\begin{equation}\lab{eq:large-side-exterior-moments}
\frac{\E|N_n|^2}{n^2}
=Q_n+2un^{11}+O(n^{10}),\qquad
\frac{\E|P_n|^2}{n^2}
=Q_n-2u'n^{11}+O(n^{10}).
\end{equation}
\end{proposition}

\par\smallskip
\noindent\hyperref[proof:prop:large-side-moments]{Proof of
Proposition~\ref*{prop:large-side-moments}.}
\par\smallskip

\begin{lemma}[KL bound using regularized large gaps]\lab{lem:KLhive}
Let \(h_n\in H_n(\lambda_n,M_n;N_n)\) be the hive obtained, via
Proposition~\ref{gt-rem}(iv), from the large-gap construction applied to the
deformed GUE minor process for \(X_n\), followed by the hive commutor
of Definition~\ref{def:oct} to change the boundary order from
\((M_n,\lambda_n;N_n)\) to \((\lambda_n,M_n;N_n)\).
 Conditionally on \(h_n\), attach three
independent uniform Gelfand--Tsetlin patterns with top rows
\[
\lambda_n,\qquad M_n,\qquad N_n.
\]
Let \(q_n\) be the resulting density on the triply augmented hive space
\(\mathbb A^3\).

Let \(p_n\) be the maximum-entropy triply augmented hive density from
Theorem~\ref{thm:gaussian1}, with the same expected quadratic boundary
moments as \(q_n\):
\[
\mathbb E_{p_n}|\lambda|^2=\mathbb E_{q_n}|\lambda|^2,\qquad
\mathbb E_{p_n}|\mu|^2=\mathbb E_{q_n}|\mu|^2,\qquad
\mathbb E_{p_n}|\nu|^2=\mathbb E_{q_n}|\nu|^2.
\]
Then
\[
D_{\mathrm{KL}}(q_n\Vert p_n)=O(n\log n).
\]
\end{lemma}

\par\smallskip
\noindent\hyperref[proof:lem:KLhive]{Proof of Lemma~\ref*{lem:KLhive}.}
\par\smallskip

\begin{lemma}[Finite-dimensional inverse for the negative raw branch]
\lab{lem:double-moment-inverse}
Let \(X,Y,Z,W>0\) satisfy
\begin{equation}\lab{eq:inverse-moment-region}
 X+Y+Z<W<(\sqrt X+\sqrt Y+\sqrt Z)^2.
\end{equation}
There are unique \(S,A,B,D>0\) such that
\begin{align}
 X&=A+\frac{A^2}{S},&
 Y&=B+\frac{B^2}{S},&
 Z&=D+\frac{D^2}{S},\lab{eq:inverse-three-moments}\\
 W&=A+B+D+\frac{(A+B+D)^2}{S}.\lab{eq:inverse-fourth-moment}
\end{align}
Consequently, \(T=-S\) and \(F=-(S+A+B+D)\) give the unique solution
of these moment equations in the raw-parameter branch \(T<0\), \(F<0\).
The scalar equation determining \(S\) has a nonvanishing derivative at its
root.

Moreover, let \(w,w'>0\), let \(u,u'>0\), and suppose
\[
 a^2=w^2+u^2,\qquad b^2=(w')^2+(u')^2,\qquad
 \theta=\frac{u}{w^2}=\frac{u'}{(w')^2}.
\]
If
\begin{align*}
 X_n&=a^2n,&Y_n&=b^2n,\\
 Q_n&=n^{21}+O(n^{20}),\\
 Z_n&=Q_n-2u'n^{11}+O(n^{10}),&
 W_n&=Q_n+2un^{11}+O(n^{10}),
\end{align*}
then \eqref{eq:inverse-moment-region} holds for all sufficiently large \(n\),
and the corresponding solution satisfies
\begin{equation}\lab{eq:inverse-raw-asymptotics}
 S_n=\theta^{-2}n+O(1),\qquad
 A_n=w^2n+O(1),\qquad
 B_n=(w')^2n+O(1).
\end{equation}
\end{lemma}

\par\smallskip
\noindent\hyperref[proof:lem:double-moment-inverse]{Proof of
Lemma~\ref*{lem:double-moment-inverse}.}
\par\smallskip

\begin{lemma}[Double-hive KL bound using a shared large-gap side]\lab{lem:doubleKLhive}
Let \(w,w'>0\) and \(u,u'\geq 0\). Let \(G_n,\widetilde G_n,G'_n,\widetilde G'_n\) be independent GUE
matrices, and set
\[
D_n=\operatorname{diag}(\operatorname{spec}(\widetilde G_n)),
\qquad
D'_n=\operatorname{diag}(\operatorname{spec}(\widetilde G'_n)).
\]
Let
\[
X_n=\sqrt n\,(wG_n+uD_n),
\qquad
Y_n=\sqrt n\,(w'G'_n+u'D'_n),
\]
and put
\[
a^2=w^2+u^2,\qquad b^2=(w')^2+(u')^2.
\]
Assume
\[
\frac{u}{w^2}=\frac{u'}{(w')^2}.
\]
Let
\[
\lambda_n=\operatorname{spec}(X_n),
\qquad
\mu_n=\operatorname{spec}(Y_n).
\]
Let \(L_n\) be the corrected regularized large-gap side from
\eqref{eq:corrected-large-gap}, and set
\[
N_n=L_n+\operatorname{diag}(X_n),
\qquad
P_n=L_n - \operatorname{diag}(Y_n).
\]
Use the compatible orientations
\[
h_n^{(1)}\in H_n(L_n,\lambda_n;N_n),
\qquad
h_n^{(2)}\in H_n(P_n,\mu_n;L_n).
\]
The first hive is the image of the minor process of \(X_n\) under
Proposition~\ref{gt-rem}(iv).  To construct the second, negate and reverse
each row of the minor process of \(Y_n\), obtaining the minor process of
\(-Y_n\), with top row \(\mu_n^\vee\).  Applying the same proposition with
large side \(L_n\) gives a hive
\[
\check h_n\in H_n(L_n,\mu_n^\vee;P_n).
\]
Reflect this hive by setting
\[
h_n^{(2)}(i,j)=\check h_n(j-i,j).
\]
This reflection preserves all rhombus inequalities and gives precisely
\(h_n^{(2)}\in H_n(P_n,\mu_n;L_n)\).  More explicitly, if
\(\mu^{(j)}_1\geq\cdots\geq\mu^{(j)}_j\) is the spectrum of the
\(j\times j\) minor of \(Y_n\), then
\begin{equation}\label{eq:oriented-second-input}
h_n^{(2)}(i,j)
=\sum_{r=1}^j P_n(r)+\sum_{r=1}^i\mu_r^{(j)}.
\end{equation}
Thus this is the usual GT-to-hive formula with base \(P_n\); its validity
here follows from reflection of the large-\(L_n\) construction and does not
require a separate large-gap hypothesis on \(P_n\).  Negation and row
reversal, and the hive reflection, are volume-preserving transformations.

Paste \(h_n^{(2)}\) and \(h_n^{(1)}\) along their common side \(L_n\).
Their trace relations are
\[
\sum_iP_n(i)+\sum_i\mu_n(i)=\sum_iL_n(i),\qquad
\sum_iL_n(i)+\sum_i\lambda_n(i)=\sum_iN_n(i).
\]
The exterior sides, in the clockwise cyclic order of the diagram, are
\[
\lambda_n,\qquad \mu_n,\qquad P_n,\qquad N_n.
\]
\begin{center}
\begin{tikzpicture}[scale=0.9, every node/.style={font=\small}]
  \coordinate (L) at (0,0);
  \coordinate (T) at (2,2);
  \coordinate (R) at (4,0);
  \coordinate (B) at (2,-2);
  \draw[thick] (L) -- (T) -- (R) -- (B) -- cycle;
  \draw[thick] (L) -- (R);
  \node[below left] at (1,-1) {\(\lambda_n\)};
  \node[above left] at (1,1) {\(\mu_n\)};
  \node[above right] at (3,1) {\(P_n\)};
  \node[below right] at (3,-1) {\(N_n\)};
  \node[above] at (2,0) {\(L_n\)};
\end{tikzpicture}
\end{center}
Decorate these four exterior sides with independent uniform
Gelfand--Tsetlin patterns.  Let \(q_n^{\mathrm{dbl}}\) be the resulting
density.  We regard \(\mathbb D^4\) in this triangulation of the same
quadrilateral; the volume-preserving identification with its earlier
triangulation, including the boundary-order conventions, is specified in
the proof of Theorem~\ref{thm:main}.

Let \(p_n^{\mathrm{dbl}}\) be the maximum-entropy density on \(\mathbb D^4\)
whose four exterior sides, in clockwise cyclic order, have the same expected
quadratic moments as
\[
\lambda_n,\qquad \mu_n,\qquad P_n,\qquad N_n .
\]
Then
\[
D_{\mathrm{KL}}\!\left(q_n^{\mathrm{dbl}}\Vert p_n^{\mathrm{dbl}}\right)
=O(n\log n).
\]
\end{lemma}

\par\smallskip
\noindent\hyperref[proof:lem:doubleKLhive]{Proof of Lemma~\ref*{lem:doubleKLhive}.}
\par\smallskip

\begin{lemma}\lab{lem:tet}
Let $w,w'>0$ and $u,u'\geq0$, and put
\[
a^2=w^2+u^2,\qquad b^2=(w')^2+(u')^2,
\qquad \delta=u+u'.
\]
Let $d,f>0$ satisfy $f-d=\delta$ and $d+f>|a-b|$.
There is a unique solution $x$ in the admissible interval
\[
\max\{|a-b|,\delta\}<x<\min\{a+b,d+f\}
\]
to the equation
\begin{equation}\label{eq:tet-cstar-equation}
\frac{\Delta_{abx}^2}{abx}\frac{\Delta_{xdf}^2}{xdf}
=\frac{9\diamondplus_{abfd}^2}{4abfd}.
\end{equation}
Denote this solution by $c_*$. Then the limit
\[
c_{**}:=\lim_{\substack{f\to\infty\\ f-d=\delta}}c_*
\]
exists. It is the unique solution with
$a^2+b^2\leq c_{**}^2<(a+b)^2$ of
\[
\delta^2
=\frac{2c_{**}^4(c_{**}^2-a^2-b^2)}
{(c_{**}^2-a^2+b^2)(c_{**}^2+a^2-b^2)}.
\]
In particular, this equation determines the displacement $\delta$ for a
nondegenerate right-angled or obtuse target triangle $(a,b,c_{**})$.
If additionally $u/w^2=u'/(w')^2$, then
\begin{equation}\label{eq:cstar-explicit-matching}
c_{**}^2=w^2+(w')^2+(u+u')^2=a^2+b^2+2uu'.
\end{equation}

Moreover, suppose
\[
a_n=a\sqrt n,\qquad b_n=b\sqrt n,
\]
and $d_n,f_n>0$ satisfy
\[
Q_n=n^{21}+O(n^{20}),\qquad
d_n^2=Q_n-2u'n^{11}+O(n^{10}),
\qquad
f_n^2=Q_n+2u n^{11}+O(n^{10}).
\]
For all sufficiently large $n$, let $c_n$ be the maximizer, over the
nonempty interval
\[
\max\{|a_n-b_n|,|d_n-f_n|\}<x<\min\{a_n+b_n,d_n+f_n\},
\]
of
\[
x\longmapsto
\frac{\Delta_{a_nb_nx}^2}{a_nb_nx}
\frac{\Delta_{xd_nf_n}^2}{xd_nf_n}.
\]
Then
\[
c_n=c_{**}\sqrt n+O(n^{-1/2}).
\]
\end{lemma}

\par\smallskip
\noindent\hyperref[proof:lem:tet]{Proof of Lemma~\ref*{lem:tet}.}
\par\smallskip

\begin{theorem}\lab{thm:main}
Let \(w,w'>0\) and \(u,u'\geq 0\), and assume
\[
\frac{u}{w^2}=\frac{u'}{(w')^2}.
\]
Set
\[
a^2=w^2+u^2,\qquad b^2=(w')^2+(u')^2,
\qquad \delta=u+u'.
\]
Define
\[
c_{**}=\sqrt{w^2+(w')^2+\delta^2}
       =\sqrt{a^2+b^2+2uu'}.
\]
By Lemma~\ref{lem:tet}, this is also the unique solution with
$a^2+b^2\leq c_{**}^2<(a+b)^2$ of
\[
\delta^2
=\frac{2c_{**}^4(c_{**}^2-a^2-b^2)}
{(c_{**}^2-a^2+b^2)(c_{**}^2+a^2-b^2)}.
\]

Let \(h_n\) be the hive obtained from the two deformed GUE minor processes by
applying the octahedron recurrence and the boundary-ordering hive commutor
specified in the proof, and write
\[
q_n:=\operatorname{Density}(h_n).
\]
Then
\[
D_{\mathrm{KL}}\!\left(
q_n\,
\middle\|\,
\operatorname{Density}\bigl(\mathcal H_n(a\sqrt n,b\sqrt n,c_{**}\sqrt n)\bigr)
\right)
=
O(n\log n).
\]
\end{theorem}

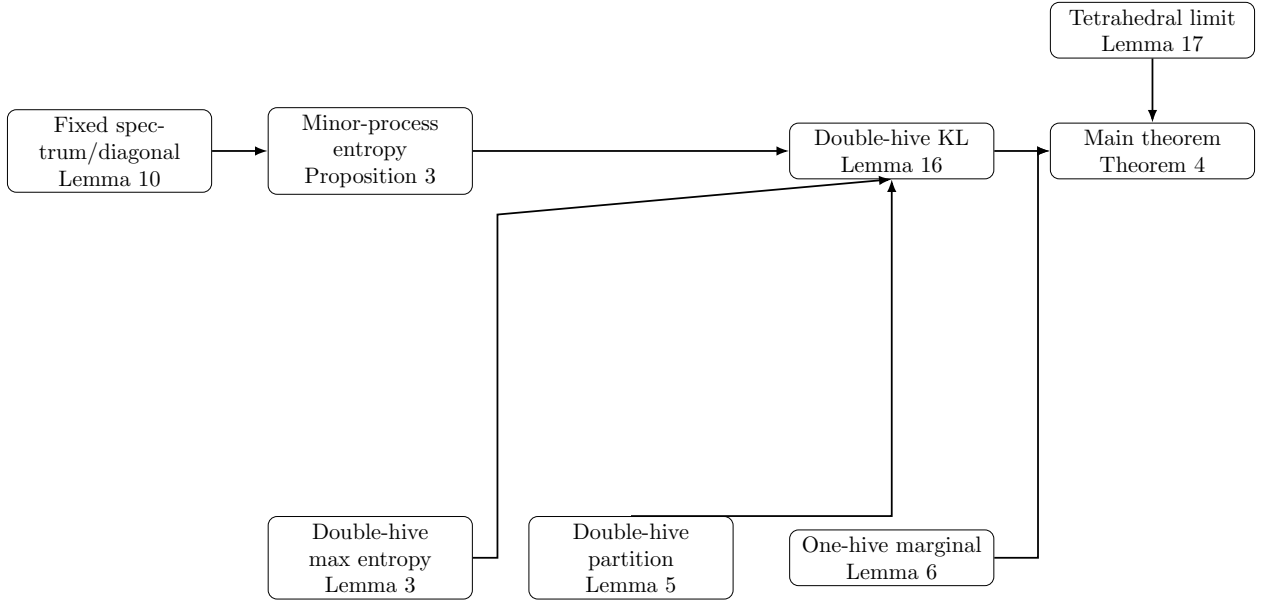
\begin{figure}[h!]
\centering
\resizebox{\textwidth}{!}{%
\begin{tikzpicture}[
  dep/.style={-{Latex[length=2mm]}, thick},
  result/.style={draw, rounded corners, align=center, inner sep=3pt,
    text width=3.0cm},
  input/.style={draw, rounded corners, align=center, inner sep=3pt,
    text width=3.0cm, fill=gray!8},
  every node/.style={font=\small}
]
\node[result] (p3) at (0,4.0) {Fixed spectrum/diagonal\\ Lemma~\ref{prop:3}};
\node[result] (p4) at (4.1,4.0) {Minor-process entropy\\ Proposition~\ref{prop:4}};
\node[result] (dkl) at (12.3,4.0) {Double-hive KL\\ Lemma~\ref{lem:doubleKLhive}};
\node[result] (main) at (16.4,4.0) {Main theorem\\ Theorem~\ref{thm:main}};

\node[result] (tet) at (16.4,5.9) {Tetrahedral limit\\ Lemma~\ref{lem:tet}};

\node[result] (dmax) at (4.1,-2.4) {Double-hive max entropy\\ Lemma~\ref{lem:double-maxent}};
\node[result] (dpart) at (8.2,-2.4) {Double-hive partition\\ Lemma~\ref{lem:double-partition}};
\node[result] (done) at (12.3,-2.4) {One-hive marginal\\ Lemma~\ref{lem:double-one-hive}};

\draw[dep] (p3) -- (p4);
\draw[dep] (p4) -- (dkl);
\draw[dep] (dmax.east) -- (6.1,-2.4) -- (6.1,3.0) -- (dkl.south);
\draw[dep] (dpart.north) -| (dkl.south);
\draw[dep] (dkl) -- (main);
\draw[dep] (tet) -- (main);
\draw[dep] (done.east) -- (14.6,-2.4) -- (14.6,4.0) -- (main.west);
\end{tikzpicture}%
}
\caption{Proof-dependency diagram for the main probabilistic construction.
An arrow \(A\to B\) means that result \(A\) is used in the proof of result
\(B\), including conceptual inputs to the construction.  Bibliographic inputs,
auxiliary comparisons, and purely local algebraic computations are not shown.}
\label{fig:proof-dependencies}
\end{figure}

\subsection{Proof of Theorem~\ref{thm:main}.}

\begin{proof}[Proof of Theorem~\ref{thm:main}]
We first identify the limiting value of the glued side selected by the
finite-\(n\) tetrahedral maximization.  The two large right-hand sides of the double
hive differ by order \(\sqrt n\), and Lemma~\ref{lem:tet} converts this
asymptotic difference into the limiting parameter \(c_{**}\).  We then use
Lemma~\ref{lem:doubleKLhive} and the data-processing inequality,
Lemma~\ref{lem:data-processing}, for the
octahedron recurrence to compare the resulting hive law with the GUE hive
law at the intermediate parameter \(c_n\).  Finally, replacing
\(c_n\) by \(c_{**}\sqrt n\) changes the relative entropy only by an
\(O(n)\) error, which is negligible compared with the desired
\(O(n\log n)\) bound.

Let the two left-hand side lengths be
\[
a_n=a\sqrt n,\qquad b_n=b\sqrt n.
\]
In the double-hive construction, the right-hand exterior sides are
\(P_n=L_n-\operatorname{diag}(Y_n)\) and
\(N_n=L_n+\operatorname{diag}(X_n)\).  With the convention of
Figure~\ref{fig:example}, these correspond respectively to \(d_n\) and
\(f_n\).  Thus, by the moment estimates in the construction,
\[
Q_n=n^{21}+O(n^{20}),\qquad
d_n^2=Q_n-2u'n^{11}+O(n^{10}),
\qquad
f_n^2=Q_n+2u n^{11}+O(n^{10}).
\]
In particular,
\[
\sqrt{Q_n}=n^{21/2}(1+O(n^{-1})),
\qquad
d_n+f_n=2\sqrt{Q_n}+O(\sqrt n),
\]
and therefore
\[
f_n-d_n
=\frac{f_n^2-d_n^2}{f_n+d_n}
=(u+u')\sqrt n+O(n^{-1/2}).
\]
Set
\[
\delta:=u+u'.
\]

Let \(c_n\) denote the maximizing glued side in the finite-\(n\) tetrahedral
problem.  Thus \(c_n\) maximizes
\[
x\longmapsto
\frac{\Delta_{a_n b_n x}^2}{a_n b_n x}
\frac{\Delta_{x d_n f_n}^2}{x d_n f_n}.
\]
By Lemma~\ref{lem:tet},
\beq \lab{eq:comp}
c_n=c_{**}\sqrt n+O(n^{-1/2}).
\eeq

Let \(q_n^{\mathrm{dbl}}\) be the density of the double hive constructed above, and
let \(p_n^{\mathrm{dbl},*}\) denote the corresponding double-hive maximum
entropy density.  By Lemma~\ref{lem:doubleKLhive},
\[
D_{\mathrm{KL}}(q_n^{\mathrm{dbl}}\|p_n^{\mathrm{dbl},*})=O(n\log n).
\]
With the orientations in Lemma~\ref{lem:doubleKLhive}, associativity
gives the volume-preserving octahedron-recurrence bijection
\[
H_n(P_n,\mu_n;L_n)\times H_n(L_n,\lambda_n;N_n)
\longrightarrow
\bigcup_{\nu_n}
H_n(\mu_n,\lambda_n;\nu_n)\times H_n(P_n,\nu_n;N_n),
\]
where the domain is also understood as a union over the shared boundary
\(L_n\).  Write its output as \((\widehat h_n,k_n)\).  In particular,
the two small boundaries occur in the order \((\mu_n,\lambda_n)\).

We fix the following boundary-order convention.  Let
\[
\mathcal C_{\alpha,\beta}^{\gamma}:
H_n(\alpha,\beta;\gamma)\longrightarrow H_n(\beta,\alpha;\gamma)
\]
be the Henriques--Kamnitzer hive commutor, constructed by the bounded
octahedron recurrence \cite[Section~4.b, Proposition~4.4]{HenriquesKamnitzer}. The same
max-plus formulas define this map for real hives: extension from rational
data follows by homogeneity and continuity.  It is a piecewise-linear,
volume-preserving bijection.  Indeed, each interior local move replaces
one coordinate \(z\) by \(\max\{a+c,b+d\}-z\), and each boundary move
replaces it by \(a+b-z\); on every linearity region these moves have
Jacobian determinant \(-1\).  The accompanying coordinate permutations
and boundary normalizations are unimodular as well. Alternatively, both
the map and its inverse are piecewise integral linear in the free lattice
coordinates; their linear pieces therefore have determinant \(\pm1\).
Define
\[
h_n=\mathcal C_{\mu_n,\lambda_n}^{\nu_n}(\widehat h_n),
\qquad q_n=\operatorname{Density}(h_n).
\]
This is the boundary-ordered hive of Definition~\ref{def:oct} and of the
theorem, with \(h_n\in H_n(\lambda_n,\mu_n;\nu_n)\).

For completeness, the second output hive can be put in the earlier
\(\mathbb D^4\) convention explicitly.  The rotation
\[
(\rho k_n)(i,j)
=k_n(n-j,n-j+i)-k_n(n,n)
\]
belongs to \(H_n(N_n^\vee,P_n;\nu_n^\vee)\).  Applying
\(\mathcal C_{N_n^\vee,P_n}^{\nu_n^\vee}\) gives a hive in
\(H_n(P_n,N_n^\vee;\nu_n^\vee)\), as in the definition of
\(\mathbb D^4\).  All these maps retain the exterior GT decorations,
with their labels, and preserve Lebesgue measure.

The density of \(p_n^{\mathrm{dbl},*}\) depends only on its four exterior
quadratic boundary terms.  Consequently its pushforward under the
octahedron recurrence and these boundary-order maps is exactly the same
quadruply augmented Gibbs law in the displayed \(\mathbb D^4\) convention.
Equivalently, before the first commutor the small-hive marginal is
\(\mathcal H_n(b_n,a_n,c_n)\); the commutor preserves each conditional
uniform hive measure and interchanges the first two boundary parameters.
Lemma~\ref{lem:double-one-hive} therefore gives the marginal after the
commutor as
\[
p_n^*=\operatorname{Density}\bigl(\mathcal H_n(a_n,b_n,c_n)\bigr).
\]
Applying data processing to this full map and then discarding the second
hive and the exterior decorations yields
\[
D_{\mathrm{KL}}(q_n\|p_n^*)=O(n\log n).
\]

It remains to replace the intermediate parameter \(c_n\) by
\(c_{**}\sqrt n\).  Let
\[
p_n^{**}=\operatorname{Density}(\mathcal H_n(a_n, b_n, c_{**}\sqrt n)).
\]
By the chain rule,
\[
D_{\mathrm{KL}}(q_n\|p_n^{**})
=
D_{\mathrm{KL}}(q_n\|p_n^*)
+
\mathbb E_{q_n}\log\frac{p_n^*}{p_n^{**}}.
\]
Thus it suffices to prove that
\[
\mathbb E_{q_n}\log\frac{p_n^*}{p_n^{**}}=O(n).
\]

If \(u+u'=0\), then the matching condition
\[
\frac{u}{w^2}=\frac{u'}{(w')^2}
\]
forces \(u=u'=0\).  In this case the two inputs are ordinary GUE minor
processes with scales \(a=w\) and \(b=w'\), and
\[
c_{**}^2=a^2+b^2.
\]

The auxiliary-side cancellation established below identifies the
uncommuted hive with the ordinary two-GUE construction, whose law is
\(\mathcal H_n(b\sqrt n,a\sqrt n,\sqrt{a^2+b^2}\sqrt n)\).
The commutor exchanges the first two boundaries and preserves the
conditional uniform hive law.  Thus \(h_n\) has exactly the target law
\(\mathcal H_n(a\sqrt n,b\sqrt n,\sqrt{a^2+b^2}\sqrt n)\), and its
relative entropy with respect to that law is zero.  We may therefore assume
\[
u+u'\neq0.
\]
Then \(c_{**}^2>a^2+b^2\), so the target triple is strictly obtuse and the
raw-parameter inverse map below is used away from its degeneracy locus.

Let
\[
(\bar a_n^*,\bar b_n^*,\bar c_n^*)
\]
be the raw Gaussian parameters corresponding to the geometric side lengths
\((a_n,b_n,c_n)\), and let
\[
(\bar a_n^{**},\bar b_n^{**},\bar c_n^{**})
\]
be the raw Gaussian parameters corresponding to the geometric side lengths
\((a_n,b_n,c_{**}\sqrt n)\).  Thus \(p_n^*\) is written with barred
parameters \((\bar a_n^*,\bar b_n^*,\bar c_n^*)\), while \(p_n^{**}\) is
written with barred parameters
\((\bar a_n^{**},\bar b_n^{**},\bar c_n^{**})\).

We record the inverse map from geometric squared side lengths to raw squared
parameters.  Put
\[
x=a^2,\qquad y=b^2,\qquad z=c^2,
\]
and
\[
A=\bar a^2,\qquad B=\bar b^2,\qquad C=\bar c^2.
\]
The forward relation is
\[
x=\frac{A(B+C)}{A+B+C},\qquad
y=\frac{B(C+A)}{A+B+C},\qquad
z=\frac{C(A+B)}{A+B+C}.
\]
If
\[
r=x+y-z,\qquad s=x+z-y,\qquad t=y+z-x,
\]
then
\[
r=\frac{2AB}{A+B+C},\qquad
s=\frac{2AC}{A+B+C},\qquad
t=\frac{2BC}{A+B+C}.
\]
Therefore
\[
\frac{ABC}{A+B+C}=\frac{rs+rt+st}{4},
\]
and the inverse map is
\[
A=\frac{rs+rt+st}{2t},\qquad
B=\frac{rs+rt+st}{2s},\qquad
C=\frac{rs+rt+st}{2r}.
\]
In the obtuse case \(r<0\), and the same formula gives the corresponding
negative raw parameter \(C=\bar c^2<0\).

Since, by (\ref{eq:comp}),
\[
c_n^2=c_{**}^2n+O(1),
\]
while \(a_n^2=a^2n\) and \(b_n^2=b^2n\) are unchanged, the two geometric
squared triples differ by \(O(1)\).  Since we are now in the strict obtuse
case, the inverse formulas above are rational functions, homogeneous of
degree one, and smooth at the limiting triple. Hence the corresponding raw
squared triples differ by
\[
(\bar a_n^*)^2-(\bar a_n^{**})^2=O(1),\qquad
(\bar b_n^*)^2-(\bar b_n^{**})^2=O(1),
\]
and
\[
(\bar c_n^*)^2-(\bar c_n^{**})^2=O(1).
\]
Moreover all these raw squared parameters have order \(n\), in the signed
sense in the obtuse case.  Consequently
\[
\frac1{(\bar a_n^*)^2}-\frac1{(\bar a_n^{**})^2}=O(n^{-2}),
\]
\[
\frac1{(\bar b_n^*)^2}-\frac1{(\bar b_n^{**})^2}=O(n^{-2}),
\]
and
\[
\frac1{(\bar c_n^*)^2}-\frac1{(\bar c_n^{**})^2}=O(n^{-2}).
\]

Write
\[
\bar s_\lambda^*=\bar a_n^*,\qquad
\bar s_\mu^*=\bar b_n^*,\qquad
\bar s_\nu^*=\bar c_n^*,
\]
and similarly for \(\bar s_\lambda^{**},\bar s_\mu^{**},\bar s_\nu^{**}\).
Then
\[
\log\frac{p_n^*}{p_n^{**}}
=
\log
\frac{
Z_n(\bar a_n^{**},\bar b_n^{**},\bar c_n^{**})
}{
Z_n(\bar a_n^*,\bar b_n^*,\bar c_n^*)
}
-\frac12
\sum_{\xi\in\{\lambda,\mu,\nu\}}
\left(
\frac1{(\bar s_\xi^*)^2}
-
\frac1{(\bar s_\xi^{**})^2}
\right)|\xi|^2,
\]
where \(Z_n(\bar a,\bar b,\bar c)\) denotes the raw-parameter normalizing
constant from Theorem~\ref{thm:gaussian1}.

For every hive with exterior sides \(\lambda,\mu,\nu\), the Horn inequalities
are equivalent to the existence of Hermitian matrices \(A\) and \(B\) such that
the spectra of \(A\), \(B\), and \(A+B\) are \(\lambda\), \(\mu\), and
\(\nu\), respectively.
Thus the Frobenius triangle inequality gives
\[
|\nu|\leq |\lambda|+|\mu|.
\]
Applying this to the image hive and using the quadratic moment estimates for
its first two exterior sides gives
\[
\mathbb E_{q_n}|\nu|^2=O(n^3).
\]
The first two boundary sides themselves satisfy
\[
\mathbb E_{q_n}|\lambda|^2=O(n^3),
\qquad
\mathbb E_{q_n}|\mu|^2=O(n^3).
\]
Consequently, the quadratic-weight part of the Radon--Nikodym derivative
contributes only
\[
\left|
\frac12
\sum_{\xi\in\{\lambda,\mu,\nu\}}
\left(
\frac1{(\bar s_\xi^*)^2}
-
\frac1{(\bar s_\xi^{**})^2}
\right)
\mathbb E_{q_n}|\xi|^2
\right|
=
O(n).
\]

It remains to control the ratio of normalizing constants.  By
Theorem~\ref{thm:gaussian1},
\[
Z_n(\bar a,\bar b,\bar c)
=
(2\pi)^nV_n(\tau_n)^{-2}
\left(
\frac{\bar a^2\bar b^2\bar c^2}
{\bar a^2+\bar b^2+\bar c^2}
\right)^{n^2/2}.
\]
For \(0\leq t\leq 1\), set
\[
(A_n(t),B_n(t),C_n(t))
=
(1-t)\bigl((\bar a_n^*)^2,(\bar b_n^*)^2,(\bar c_n^*)^2\bigr)
+t\bigl((\bar a_n^{**})^2,(\bar b_n^{**})^2,(\bar c_n^{**})^2\bigr).
\]
The estimates above give
\[
(A_n(t),B_n(t),C_n(t))
=
n\bigl((\bar a^{**})^2,(\bar b^{**})^2,(\bar c^{**})^2\bigr)+O(1)
\]
uniformly in \(t\).  In particular, all three coordinates are comparable to
\(n\).  More explicitly, since the limiting raw squared triple is
nondegenerate, there is an \(\varepsilon>0\), independent of \(n\) and \(t\),
such that for all sufficiently large \(n\),
\[
|A_n(t)|,\ |B_n(t)|,\ |C_n(t)|,\ |A_n(t)+B_n(t)+C_n(t)|
\geq \varepsilon n .
\]
For
\[
F(A,B,C)=\log\left(\frac{ABC}{A+B+C}\right)
\]
we have
\[
\partial_A F=\frac1A-\frac1{A+B+C},
\qquad
\partial_B F=\frac1B-\frac1{A+B+C},
\qquad
\partial_C F=\frac1C-\frac1{A+B+C}.
\]
The preceding lower bounds therefore give, uniformly in \(t\),
\[
\nabla_{A,B,C}
\log\left(\frac{ABC}{A+B+C}\right)=O(n^{-1}).
\]
Since the two raw squared triples differ by \(O(1)\), it follows that
\[
\log
\frac{
Z_n(\bar a_n^{**},\bar b_n^{**},\bar c_n^{**})
}{
Z_n(\bar a_n^*,\bar b_n^*,\bar c_n^*)
}
=O(n).
\]
Combining the two estimates,
\[
\mathbb E_{q_n}\log\frac{p_n^*}{p_n^{**}}=O(n).
\]
Hence
\[
D_{\mathrm{KL}}(q_n\|p_n^{**})
=
D_{\mathrm{KL}}(q_n\|p_n^*)+O(n)
=
O(n\log n).
\]

We finish by justifying the auxiliary-side cancellation used above.
Fix the two minor processes.  Formula~\eqref{eq:oriented-second-input}
expresses the first member of the associative chain using base \(P_n\),
while the other member uses base
\(L_n=P_n+\operatorname{diag}(Y_n)\).  These are precisely the two
GT-to-hive formulas used in the Speyer representation of the hive component
in \cite[Section~3, equations~(3.9)--(3.10)]{NarSheffTao}, with the two
minor-process labels interchanged.

In that representation, substituting the GT-to-hive formulas cancels the
cumulative sums of the auxiliary base in every tiling weight together
with its boundary correction.  This is an algebraic cancellation in the
max-plus formula, and holds for every base for which the two input arrays
are hives.  A separate large-gap bound on \(P_n\) is therefore unnecessary:
the reflection construction in Lemma~\ref{lem:doubleKLhive} already proves
that both arrays are hives on every outcome.  It follows that
\(\widehat h_n\), including its third boundary \(\nu_n\), depends only
on the two minor processes.  The same is then true of
\(h_n=\mathcal C_{\mu_n,\lambda_n}^{\nu_n}(\widehat h_n)\), since the
commutor uses only the output hive and its boundaries.  This proves the
claimed independence from the auxiliary side and completes the theorem.

\end{proof}

\section{Proofs of lemmas and propositions from Section~\ref{sec:double-hives}, Subsection~\ref{ssec:4.1} and Subsection~\ref{ssec:4.2}}

\phantomsection\label{proof:lem:doublehive}
\begin{proof}[Proof of Lemma~\ref{lem:doublehive}]
Let
\[
L=\alpha+\beta+\eta-\phi.
\]
By the Gaussian conditioning formula,
\[
\operatorname{Var}(X\mid L=0)
=
\operatorname{Var}(X)
-
\frac{\operatorname{Cov}(X,L)^2}{\operatorname{Var}(L)}.
\]
Since \(\operatorname{Var}(L)=t\), we get
\[
a^2
=
a_{\mathrm{raw}}^2-\frac{a_{\mathrm{raw}}^4}{t}
=
\frac{a_{\mathrm{raw}}^2(t-a_{\mathrm{raw}}^2)}{t},
\]
\[
b^2
=
\frac{b_{\mathrm{raw}}^2(t-b_{\mathrm{raw}}^2)}{t},
\qquad
d^2
=
\frac{d_{\mathrm{raw}}^2(t-d_{\mathrm{raw}}^2)}{t},
\qquad
f^2
=
\frac{f_{\mathrm{raw}}^2(t-f_{\mathrm{raw}}^2)}{t}.
\]
Similarly,
\[
c_*^2
=
\operatorname{Var}(\alpha+\beta\mid L=0)
=
\frac{(a_{\mathrm{raw}}^2+b_{\mathrm{raw}}^2)
(d_{\mathrm{raw}}^2+f_{\mathrm{raw}}^2)}{t}.
\]

We now show that \(c_*\) is the maximizer in the tetrahedral variational
problem. First,
\[
a^2+b^2-c_*^2
=
\frac{2a_{\mathrm{raw}}^2b_{\mathrm{raw}}^2}{t}.
\]
Using
\[
\Delta_{xyz}^2
=
\frac{x^2y^2}{4}-\frac{(x^2+y^2-z^2)^2}{16},
\]
this gives
\[
\Delta_{abc_*}^2
=
\frac{
a_{\mathrm{raw}}^2b_{\mathrm{raw}}^2
(d_{\mathrm{raw}}^2+f_{\mathrm{raw}}^2)
}{4t}.
\]
Similarly,
\[
d^2+f^2-c_*^2
=
\frac{2d_{\mathrm{raw}}^2f_{\mathrm{raw}}^2}{t},
\]
and hence
\[
\Delta_{c_*df}^2
=
\frac{
d_{\mathrm{raw}}^2f_{\mathrm{raw}}^2
(a_{\mathrm{raw}}^2+b_{\mathrm{raw}}^2)
}{4t}.
\]

Let \(y=c^2\). Up to the fixed factor \((abdf)^{-1}\), the objective is
\[
\Psi(y)
=
\frac{\Delta_{abc}^2\Delta_{cdf}^2}{c^2}.
\]
At an interior critical point,
\[
\frac{d}{dy}\log \Psi(y)=0.
\]
Evaluating at \(y=c_*^2\), we obtain
\[
\frac{1}{d_{\mathrm{raw}}^2+f_{\mathrm{raw}}^2}
+
\frac{1}{a_{\mathrm{raw}}^2+b_{\mathrm{raw}}^2}
-
\frac{t}{
(a_{\mathrm{raw}}^2+b_{\mathrm{raw}}^2)
(d_{\mathrm{raw}}^2+f_{\mathrm{raw}}^2)
}
=0.
\]
Thus \(c_*\) is a critical point of the area-product objective.
To check uniqueness, write
\[
\ell_1=(a-b)^2,\quad r_1=(a+b)^2,\qquad
\ell_2=(d-f)^2,\quad r_2=(d+f)^2.
\]
On \(\max(\ell_1,\ell_2)<y<\min(r_1,r_2)\), Heron's formula gives
\[
\frac{d}{dy}\log\Psi(y)
=\frac1{y-\ell_1}-\frac1{r_1-y}
 +\frac1{y-\ell_2}-\frac1{r_2-y}-\frac1y.
\]
Its derivative is strictly negative, since \(\ell_1,\ell_2\geq0\).
The objective is positive inside this interval and tends to zero at
its endpoints.  Hence \(c_*\) is its unique maximizer.

By the tetrahedral identity \eqref{eq:tetrahedral-identity},
\[
\max_c
\frac{\Delta_{abc}^2}{abc}
\frac{\Delta_{cdf}^2}{cdf}
=
\frac{
9|\mathtt{tet}|_{abfd}^2
}{
4abfd
}.
\]
Since the maximum is attained at \(c=c_*\), the claimed identity follows.
\end{proof}

\phantomsection\label{proof:lem:double-maxent}
\begin{proof}[Proof of Lemma~\ref{lem:double-maxent}]
This is the same exponential-family maximum-entropy calculation as in
Theorem~\ref{thm:gaussian1} and Remark~\ref{rem:maxent-kl}, but applied to
the four-boundary double-hive cone \(\mathbb D^4\).  The only point to note
is that the reference measure is
Lebesgue measure on \(\mathbb D^4\), whose fibers contain the four exterior
Gelfand--Tsetlin volume factors and no Gelfand--Tsetlin volume factor on the
glued boundary.
When \(\bar f^2<0\), the normalizing integral is understood in the same sense
as the imaginary-side case in Theorem~\ref{thm:gaussian1}: the \(\phi\)-side
is represented by the \(GUE_-\) measure, and completing the square gives a
finite integral under the hypothesis \(T<0\).
Indeed, write
\[
A=\bar a^2,\qquad B=\bar b^2,\qquad D=\bar d^2,\qquad R=-\bar f^2.
\]
Then \(T<0\) is equivalent to \(R>A+B+D\).  On the matrix slice
\[
X_\alpha+X_\beta+X_\eta-X_\phi=0,
\]
we have \(X_\phi=X_\alpha+X_\beta+X_\eta\), and the exponent is
\[
-\frac12
\left(
\frac{\operatorname{Tr}X_\alpha^2}{A}
+
\frac{\operatorname{Tr}X_\beta^2}{B}
+
\frac{\operatorname{Tr}X_\eta^2}{D}
-
\frac{\operatorname{Tr}(X_\alpha+X_\beta+X_\eta)^2}{R}
\right).
\]
By Cauchy--Schwarz in the Hilbert--Schmidt norm,
\[
\operatorname{Tr}(X_\alpha+X_\beta+X_\eta)^2
\leq
(A+B+D)
\left(
\frac{\operatorname{Tr}X_\alpha^2}{A}
+
\frac{\operatorname{Tr}X_\beta^2}{B}
+
\frac{\operatorname{Tr}X_\eta^2}{D}
\right).
\]
Thus the quadratic form inside the parentheses is bounded below by
\[
\left(1-\frac{A+B+D}{R}\right)
\left(
\frac{\operatorname{Tr}X_\alpha^2}{A}
+
\frac{\operatorname{Tr}X_\beta^2}{B}
+
\frac{\operatorname{Tr}X_\eta^2}{D}
\right),
\]
which is positive definite.  Hence the completed-square integral is finite.

Let \(q\) be another density on \(\mathbb D^4\) satisfying the same four
quadratic moment constraints as \(p_{\mathrm{dbl}}\).  Since
\[
-\log p_{\mathrm{dbl}}(x)
=
\log Z_{\mathrm{dbl}}
+
\frac12
\left(
\frac{|\alpha(x)|^2}{\bar a^2}
+
\frac{|\beta(x)|^2}{\bar b^2}
+
\frac{|\eta(x)|^2}{\bar d^2}
+
\frac{|\phi(x)|^2}{\bar f^2}
\right),
\]
the moment constraints imply
\[
-\mathbb E_q\log p_{\mathrm{dbl}}
=
-\mathbb E_{p_{\mathrm{dbl}}}\log p_{\mathrm{dbl}}
=
\mathrm{ent}(p_{\mathrm{dbl}}).
\]
Therefore
\[
D_{\mathrm{KL}}(q\Vert p_{\mathrm{dbl}})
=
-\mathrm{ent}(q)-\mathbb E_q\log p_{\mathrm{dbl}}
=
\mathrm{ent}(p_{\mathrm{dbl}})-\mathrm{ent}(q).
\]
The nonnegativity of relative entropy gives
\(\mathrm{ent}(q)\leq\mathrm{ent}(p_{\mathrm{dbl}})\), with equality only if
\(q=p_{\mathrm{dbl}}\) almost everywhere.
\end{proof}

\phantomsection\label{proof:lem:tet-analogue of GangNar}
\begin{proof}[Proof of Lemma~\ref{lem:tet-analogue of GangNar}]
We first claim the following.

\begin{claim}
Whether \(\bar f>0\), or \(\bar f^2<0\) and
\[
T:=\bar a^2+\bar b^2+\bar d^2+\bar f^2<0,
\]
the numbers \(a,b,d,f\) can be the side lengths of a quadrilateral in
\(\mathbb R^3\).
\end{claim}
\begin{proof}
We first prove the quadrilateral inequalities in the positive raw case.
Write
\[
A=\bar a^2,\qquad B=\bar b^2,\qquad D=\bar d^2,\qquad F=\bar f^2,
\qquad T=A+B+D+F.
\]
Then
\[
a^2=\frac{A(T-A)}{T},\qquad
b^2=\frac{B(T-B)}{T},\qquad
d^2=\frac{D(T-D)}{T},\qquad
f^2=\frac{F(T-F)}{T}.
\]
It is enough to prove each side is at most the sum of the other three.
For instance,
\[
b
=
\sqrt{\frac{B(A+D+F)}{T}}
\geq
\sqrt{\frac{AB}{T}},
\]
and similarly
\[
d\geq \sqrt{\frac{AD}{T}},
\qquad
f\geq \sqrt{\frac{AF}{T}}.
\]
Therefore
\[
b+d+f
\geq
\sqrt{\frac{A}{T}}\bigl(\sqrt B+\sqrt D+\sqrt F\bigr)
\geq
\sqrt{\frac{A(B+D+F)}{T}}
=a.
\]
The same argument, after permuting \(A,B,D,F\), gives
\[
a\leq b+d+f,\qquad
b\leq a+d+f,\qquad
d\leq a+b+f,\qquad
f\leq a+b+d.
\]
Thus \(a,b,d,f\) are the side lengths of a quadrilateral in the positive
raw case.

It remains to prove the quadrilateral inequalities in the case
\(\bar f^2<0\) and
\[
T:=\bar a^2+\bar b^2+\bar d^2+\bar f^2<0,
\]
Write
\[
A=\bar a^2,\qquad B=\bar b^2,\qquad D=\bar d^2,\qquad R=-\bar f^2.
\]
Then \(A,B,D,R>0\), and the assumption \(T<0\) says
\[
R>A+B+D.
\]
Set
\[
S=R-A-B-D>0.
\]
Since \(T=-S\), the side-length formulas become
\[
a^2
=
\frac{\bar a^2(\bar b^2+\bar d^2+\bar f^2)}{T}
=
\frac{A(B+D-R)}{-S}
=
\frac{A(R-B-D)}{S}.
\]
Using \(R=A+B+D+S\), this is
\[
a^2
=
\frac{A(A+S)}{S}
=
A+\frac{A^2}{S}.
\]
Similarly,
\[
b^2=B+\frac{B^2}{S},
\qquad
d^2=D+\frac{D^2}{S}.
\]
Also
\[
f^2
=
\frac{\bar f^2(\bar a^2+\bar b^2+\bar d^2)}{T}
=
\frac{(-R)(A+B+D)}{-S}
=
\frac{R(A+B+D)}{S}.
\]
Since \(R=A+B+D+S\), this gives
\[
f^2
=
(A+B+D)+\frac{(A+B+D)^2}{S}.
\]

Define
\[
h(x)=\sqrt{x+\frac{x^2}{S}},
\qquad x\geq 0.
\]
Then
\[
a=h(A),\qquad b=h(B),\qquad d=h(D),\qquad f=h(A+B+D).
\]
We claim that \(h\) is subadditive. Indeed, for \(x,y\geq 0\),
\[
h(x+y)\leq h(x)+h(y)
\]
is equivalent, after squaring, to
\[
x+y+\frac{(x+y)^2}{S}
\leq
x+\frac{x^2}{S}
+
y+\frac{y^2}{S}
+
2\sqrt{\left(x+\frac{x^2}{S}\right)
\left(y+\frac{y^2}{S}\right)}.
\]
After cancellation, this becomes
\[
\frac{xy}{S}
\leq
\sqrt{\left(x+\frac{x^2}{S}\right)
\left(y+\frac{y^2}{S}\right)}.
\]
Squaring both sides, this is
\[
\frac{x^2y^2}{S^2}
\leq
xy\left(1+\frac{x}{S}\right)\left(1+\frac{y}{S}\right),
\]
which is immediate.

Therefore
\[
f
=
h(A+B+D)
\leq
h(A)+h(B)+h(D)
=
a+b+d.
\]
Moreover \(h\) is increasing, so
\[
f=h(A+B+D)\geq h(A)=a,
\qquad
f\geq b,
\qquad
f\geq d.
\]
Thus the only possibly nontrivial quadrilateral inequality is
\[
f\leq a+b+d,
\]
which we have proved. Hence
\[
a,b,d,f
\]
satisfy the quadrilateral inequalities, and therefore can be realized as the
side lengths of a quadrilateral in \(\mathbb R^3\).
\end{proof}
As in Lemma~\ref{lem:doublehive}, define
\[
c_*^2=\frac{(\bar a^2+\bar b^2)(\bar d^2+\bar f^2)}{T}.
\]
Then a direct calculation gives
\[
a^2+b^2-c_*^2
=
\frac{2\bar a^2\bar b^2}{T},
\qquad
d^2+f^2-c_*^2
=
\frac{2\bar d^2\bar f^2}{T}.
\]
Using
\[
\Delta_{xyz}^2
=
\frac{x^2y^2}{4}-\frac{(x^2+y^2-z^2)^2}{16},
\]
we obtain
\[
\Delta_{abc_*}^2
=
\frac{\bar a^2\bar b^2(\bar d^2+\bar f^2)}{4T},
\qquad
\Delta_{c_*df}^2
=
\frac{\bar d^2\bar f^2(\bar a^2+\bar b^2)}{4T}.
\]
Therefore
\[
\frac{\Delta_{abc_*}^2}{abc_*}
\frac{\Delta_{c_*df}^2}{c_*df}
=
\frac{\bar a^2\bar b^2\bar d^2\bar f^2}{16T\,abdf}.
\]
By the tetrahedral identity \eqref{eq:tetrahedral-identity},
\[
\max_x
\frac{\Delta_{abx}^2}{abx}
\frac{\Delta_{xdf}^2}{xdf}
=
\frac{9\diamondplus_{abdf}^2}{4abdf}.
\]
The same critical-point calculation as in Lemma~\ref{lem:doublehive} shows that
the maximum is attained at \(x=c_*\). Hence
\[
\frac{\bar a^2\bar b^2\bar d^2\bar f^2}{16T\,abdf}
=
\frac{9\diamondplus_{abdf}^2}{4abdf}.
\]
Multiplying by \(16abdf\) gives
\[
\frac{\bar a^2\bar b^2\bar d^2\bar f^2}
{\bar a^2+\bar b^2+\bar d^2+\bar f^2}
=
36\diamondplus_{abdf}^2.
\]

\end{proof}

\begin{claim}\lab{cl:vtau}
\[
\log V_n(\tau_n)
=
\frac{n^2}{2}\log n
-
\frac34 n^2
+
O(n\log n).
\]
\end{claim}
\begin{proof}
Since
\[
\tau_n(i)=\frac{n+1}{2}-i,
\]
we have \(\tau_n(i)-\tau_n(j)=j-i\) for \(i<j\). Hence
\[
V_n(\tau_n)
=
\prod_{1\leq i<j\leq n}(j-i)
=
\prod_{k=1}^{n-1} k^{\,n-k}.
\]
Therefore
\[
\log V_n(\tau_n)
=
\sum_{k=1}^{n-1}(n-k)\log k
=
n\sum_{k=1}^{n-1}\log k
-
\sum_{k=1}^{n-1}k\log k .
\]
By Stirling's formula \cite[Sec.~5.11]{NISTDLMF},
\[
\sum_{k=1}^{n-1}\log k
=
n\log n-n+O(\log n),
\]
and by the Euler--Maclaurin formula \cite[Sec.~2.10]{NISTDLMF},
\[
\sum_{k=1}^{n-1}k\log k
=
\frac{n^2}{2}\log n-\frac{n^2}{4}+O(n\log n).
\]
Substituting gives
\[
\log V_n(\tau_n)
=
\frac{n^2}{2}\log n-\frac34 n^2+O(n\log n),
\]
as claimed.
\end{proof}

\phantomsection\label{proof:lem:double-partition}
\begin{proof}[Proof of Lemma~\ref{lem:double-partition}]
For fixed \(\nu\), define
\[
K_{A,B}(\nu)
=
\int
V_n(\alpha)V_n(\beta)
|H_n(\alpha,\beta;\nu)|
\exp\left[
-\frac12
\left(
\frac{|\alpha|^2}{A}
+
\frac{|\beta|^2}{B}
\right)
\right]
d\alpha\,d\beta .
\]
Here and below, integrals involving a hive volume are taken over the
trace-compatible affine slice with coordinate Lebesgue measure.  In this
display \(\sum_i\alpha_i+\sum_i\beta_i=\sum_i\nu_i\), and the notation
\(d\alpha\,d\beta\) means
\[
d\alpha_1\cdots d\alpha_n\,d\beta_1\cdots d\beta_{n-1},
\qquad
\beta_n=\sum_i\nu_i-\sum_i\alpha_i-\sum_{i<n}\beta_i.
\]
The same elimination convention is used for the second hive.  Changing
from these two slices and the full vector \(\nu\) to the independent
boundary coordinates of \(\mathbb D^4\) has absolute Jacobian one.
We determine \(K_{A,B}\) by the same matrix Gaussian convolution that
underlies Theorem~\ref{thm:gaussian1}.  Let \(X\) and \(Y\) be independent
GUE matrices with variance parameters \(A\) and \(B\).  Put
\[
Z_S^{\mathrm{eig}}
=
\int_{\Spec_n+\mathbb R\one}
V_n(\xi)^2
\exp\left(-\frac{|\xi|^2}{2S}\right)d\xi .
\]
The ordered Gaussian Mehta integral (see \cite[Ch.~17]{Mehta}) gives
\[
Z_1^{\mathrm{eig}}
=
(2\pi)^{n/2}V_n(\tau_n),
\]
and hence, for every \(S>0\),
\[
Z_S^{\mathrm{eig}}
=
S^{n^2/2}(2\pi)^{n/2}V_n(\tau_n).
\]
The ordered eigenvalue densities of \(X\) and \(Y\) are
\[
\frac{1}{Z_A^{\mathrm{eig}}}
V_n(\alpha)^2
\exp\left(-\frac{|\alpha|^2}{2A}\right)d\alpha
\quad\text{and}\quad
\frac{1}{Z_B^{\mathrm{eig}}}
V_n(\beta)^2
\exp\left(-\frac{|\beta|^2}{2B}\right)d\beta .
\]
For fixed spectra \(\alpha,\beta\), Weyl integration
(Lemma~\ref{lem:Weyl}) together with the Knutson--Tao formula,
\eqref{eq:2.4new}, gives the conditional density of
\(\nu=\spec(X+Y)\) as
\[
\frac{V_n(\nu)V_n(\tau_n)}
     {V_n(\alpha)V_n(\beta)}
|H_n(\alpha,\beta;\nu)|\,d\nu .
\]
Thus the marginal density of \(\nu\) is
\[
\frac{V_n(\nu)V_n(\tau_n)}
     {Z_A^{\mathrm{eig}}Z_B^{\mathrm{eig}}}
K_{A,B}(\nu)\,d\nu .
\]
On the other hand, \(X+Y\) is itself a GUE matrix with variance parameter
\(A+B\), so the same marginal density is
\[
\frac{1}{Z_{A+B}^{\mathrm{eig}}}
V_n(\nu)^2
\exp\left[-\frac{|\nu|^2}{2(A+B)}\right]d\nu .
\]
Therefore
\[
K_{A,B}(\nu)
=
\frac{Z_A^{\mathrm{eig}}Z_B^{\mathrm{eig}}}
     {V_n(\tau_n)Z_{A+B}^{\mathrm{eig}}}
V_n(\nu)
\exp\left[-\frac{|\nu|^2}{2(A+B)}\right]
\]
Using the evaluation of \(Z_S^{\mathrm{eig}}\),
\[
\frac{Z_A^{\mathrm{eig}}Z_B^{\mathrm{eig}}}
     {V_n(\tau_n)Z_{A+B}^{\mathrm{eig}}}
=
(2\pi)^{n/2}
\left(\frac{AB}{A+B}\right)^{n^2/2}.
\]
Thus
\[
K_{A,B}(\nu)
=
(2\pi)^{n/2}
\left(\frac{AB}{A+B}\right)^{n^2/2}
V_n(\nu)
\exp\left[-\frac{|\nu|^2}{2(A+B)}\right].
\]
The same formula applies to the second hive, with \(A,B\) replaced by
\(D,F\) and \(\nu\) replaced by \(\nu^\vee\), in the positive raw case.  In
the case \(F<0\) and \(A+B+D+F<0\), the \(\phi\)-side is represented by
\(GUE_-\).  Completing the square in the Hermitian-matrix convolution gives
the same expression for \(K_{D,F}\), now with \(D+F<0\):
\[
K_{D,F}(\nu^\vee)
=
(2\pi)^{n/2}
\left(\frac{DF}{D+F}\right)^{n^2/2}
V_n(\nu^\vee)
\exp\left[-\frac{|\nu^\vee|^2}{2(D+F)}\right].
\]
Here \(DF/(D+F)>0\).  Although the last exponential grows with $|\nu^\vee|$ when
\(D+F<0\), the final \(\nu\)-integral below is still finite.  Indeed,
\[
D+F<-(A+B)
\]
and hence
\[
\frac1{A+B}+\frac1{D+F}>0.
\]
Thus the product \(K_{A,B}(\nu)K_{D,F}(\nu^\vee)\) has a decaying Gaussian
factor in \(\nu\).  We also use
\[
V_n(\nu^\vee)=V_n(\nu),
\qquad
|\nu^\vee|=|\nu|.
\]

By the definition of \(\mathbb D^4\), the four exterior
Gelfand--Tsetlin decorations contribute
\[
\frac{V_n(\alpha)V_n(\beta)V_n(\eta)V_n(\phi)}
     {V_n(\tau_n)^4}.
\]
Therefore
\[
\begin{aligned}
Z_{\mathrm{dbl}}
&=
V_n(\tau_n)^{-4}
\int K_{A,B}(\nu)K_{D,F}(\nu^\vee)\,d\nu  \\
&=
\frac{(2\pi)^n}{V_n(\tau_n)^4}
\left(\frac{ABDF}{(A+B)(D+F)}\right)^{n^2/2}
\int_{\Spec_n+\mathbb R\one}
V_n(\nu)^2
\exp\left[
-\frac{|\nu|^2}{2}
\left(\frac{1}{A+B}+\frac{1}{D+F}\right)
\right]d\nu  \\
&=
\frac{(2\pi)^{3n/2}}{V_n(\tau_n)^3}
\left(\frac{ABDF}{A+B+D+F}\right)^{n^2/2}.
\end{aligned}
\]

Since
\[
p_{\mathrm{dbl}}(x)
=
Z_{\mathrm{dbl}}^{-1}
\exp\left[-\frac12 Q(x)\right],
\]
we have
\[
\ent(p_{\mathrm{dbl}})
=
\log Z_{\mathrm{dbl}}+\frac12\E_{p_{\mathrm{dbl}}}Q.
\]
Under common scaling \(A,B,D,F\mapsto rA,rB,rD,rF\), the partition
function scales as \(r^{3n^2/2}\).  Differentiating at \(r=1\) gives
\[
\frac12\E_{p_{\mathrm{dbl}}}Q=\frac{3n^2}{2}.
\]
Thus
\[
\ent(p_{\mathrm{dbl}})
=
\frac{3n^2}{2}
+\frac{3n}{2}\log(2\pi)
-3\log V_n(\tau_n)
+
\frac{n^2}{2}
\log\left(\frac{ABDF}{A+B+D+F}\right).
\]
Substituting \(A=n\bar a^2\), \(B=n\bar b^2\),
\(D=n\bar d^2\), and \(F=n\bar f^2\), we get
\[
\frac{ABDF}{A+B+D+F}
=
n^3
\frac{\bar a^2\bar b^2\bar d^2\bar f^2}{\bar T}.
\]
Using (from Claim~\ref{cl:vtau}) 
\[
\log V_n(\tau_n)
=
\frac{n^2}{2}\log n-\frac34n^2+O(n\log n)
\]
the \(n^2\log n\) terms cancel and all remaining terms of order at most
\(n\log n\) are absorbed in the error.
This gives
\[
\ent(p_{\mathrm{dbl}})
=
\frac{n^2}{2}
\log\left(
\frac{\bar a^2\bar b^2\bar d^2\bar f^2}{\bar T}
\right)
+\frac{15}{4}n^2
+O(n\log n),
\]
as claimed.
\end{proof}

\phantomsection\label{proof:lem:double-one-hive}
\begin{proof}[Proof of Lemma~\ref{lem:double-one-hive}]
Fix the first hive boundary \((\alpha,\beta;\nu)\).  Since the double-hive
cone \(\mathbb D^4\) has Gelfand--Tsetlin decorations only on the four
exterior boundaries, integrating out the exterior decorations on
\(\alpha\) and \(\beta\) contributes the factor
\[
\frac{V_n(\alpha)V_n(\beta)}{V_n(\tau_n)^2}.
\]
It remains to integrate out the second hive, its two exterior
Gelfand--Tsetlin decorations, and the exterior boundaries \(\eta,\phi\), with
the glued boundary fixed at \(\nu^\vee\).  This contribution is
proportional to
\[
\int
\exp\left[
-\frac12\left(\frac{|\eta|^2}{\bar d^2}+\frac{|\phi|^2}{\bar f^2}\right)
\right]
\frac{V_n(\eta)V_n(\phi)}{V_n(\tau_n)^2}
|H_n(\eta,\phi^\vee;\nu^\vee)|\,d\eta\,d\phi .
\]
By the Knutson--Tao formula, this is the same as the spectral density
at \(\nu^\vee\) of the sum of two independent GUE matrices with variances
\(\bar d^2\) and \(\bar f^2\), divided by one Vandermonde factor
\(V_n(\nu^\vee)\), when \(\bar f>0\).  When \(\bar f^2<0\), the same identity
is obtained by the argument used in the imaginary-side case of
Theorem~\ref{thm:gaussian1}: replace the \(\phi\)-side GUE density by the
\(GUE_-\) measure with parameter \(\bar f^2\), and complete the square in the
Hermitian-matrix convolution.  This gives the \(GUE_-\) measure with parameter
\(\bar d^2+\bar f^2\).
Equivalently, it is proportional to
\[
V_n(\nu)\,
\exp\left[-\frac{|\nu|^2}{2(\bar d^2+\bar f^2)}\right],
\]
where we use \(V_n(\nu^\vee)=V_n(\nu)\).

Therefore the marginal density of the first hive boundary, multiplied by
the uniform Lebesgue measure on \(H_n(\alpha,\beta;\nu)\), is proportional to
\[
V_n(\alpha)V_n(\beta)V_n(\nu)
|H_n(\alpha,\beta;\nu)|
\exp\left[
-\frac12
\left(
\frac{|\alpha|^2}{\bar a^2}
+
\frac{|\beta|^2}{\bar b^2}
+
\frac{|\nu|^2}{\bar d^2+\bar f^2}
\right)
\right].
\]
This is exactly the boundary density defining
\(\mathcal H_n(a,b,c_*)\), with single-hive raw parameters
\[
\bar A=\bar a,\qquad
\bar B=\bar b,\qquad
\bar C^2=\bar d^2+\bar f^2.
\]
Indeed, applying the conversion in the definition of \(\mathcal H_n\) gives
\[
a^2=
\frac{\bar a^2(\bar b^2+\bar d^2+\bar f^2)}{T},
\qquad
b^2=
\frac{\bar b^2(\bar a^2+\bar d^2+\bar f^2)}{T},
\]
and
\[
c_*^2=
\frac{(\bar d^2+\bar f^2)(\bar a^2+\bar b^2)}{T}.
\]
The proof for the second hive is the same, with the two halves interchanged.
Its single-hive raw middle parameter is \(\bar a^2+\bar b^2\).  When
\(\bar f>0\), this directly gives \(\mathcal H_n(d,f,c_*)\).  When
\(\bar f^2<0\), one first applies the same \(GUE_-\) calculation and then
uses the symmetry of the single-hive density in its three boundary components to
rewrite the resulting density in the convention of
Theorem~\ref{thm:gaussian1}, where the imaginary raw parameter is the third
one; the geometric side lengths are still \(d,f,c_*\).

Finally, when \(\bar f>0\), the displayed formulas for \(a,b,d,f,c_*\) are
exactly the conditional-variance formulas in Lemma~\ref{lem:doublehive}, with
\((a_{\mathrm{raw}},b_{\mathrm{raw}},d_{\mathrm{raw}},f_{\mathrm{raw}})
=(\bar a,\bar b,\bar d,\bar f)\).  Hence Lemma~\ref{lem:doublehive}
identifies this same \(c_*\) as the tetrahedral maximizer of
\[
x\mapsto
\frac{\Delta_{abx}^2}{abx}\frac{\Delta_{xdf}^2}{xdf}.
\]
When \(\bar f^2<0\), the same displayed formulas are obtained from the
completion-of-squares calculation with \(GUE_-\), in the convention fixed after
Theorem~\ref{thm:gaussian1}.  More explicitly, the role of
\(f_{\mathrm{raw}}^2\) in the calculation of Lemma~\ref{lem:doublehive} is
played by the signed quantity \(\bar f^2\), while the geometric side length
\(f\) is obtained from the resulting positive variance formula.  The
completion-of-squares step only uses the raw parameters through the algebraic
combinations
\[
\bar a^2+\bar b^2+\bar d^2+\bar f^2,\qquad
\bar a^2+\bar b^2,\qquad
\bar d^2+\bar f^2,
\]
and hence the same formulas for \(a,b,d,f\) and \(c_*\) continue to hold with
\(\bar f^2\) signed.  Substituting these formulas into
\[
\Delta_{xyz}^2
=
\frac{x^2y^2}{4}-\frac{(x^2+y^2-z^2)^2}{16}
\]
gives the same two area identities as in Lemma~\ref{lem:doublehive}, and
differentiating
\[
\Psi(y)
=
\frac{\Delta_{ab\sqrt y}^2\Delta_{\sqrt y df}^2}{y}
\]
again gives the same critical-point equation at \(y=c_*^2\).  Since the
area-product objective is positive in the interior of the admissible interval,
vanishes at the endpoints, and has a unique critical point, this critical point
is the maximizer.  Thus \(c_*\) is again the tetrahedral maximizer of
\[
x\mapsto
\frac{\Delta_{abx}^2}{abx}\frac{\Delta_{xdf}^2}{xdf}.
\]
\end{proof}

\phantomsection\label{proof:prop:middle}
\begin{proof}[Proof of Proposition~\ref{prop:middle}]
First compute the distribution of the middle side under Lebesgue measure on
the double hive polytope.  For fixed \(\nu\), the fiber over the middle side
\(\nu\) is
\[
        H_n(\alpha,\beta;\nu)
        \times
        H_n(\eta,\phi^\vee;\nu^\vee).
\]
Therefore, by Fubini, the pushforward of Lebesgue measure on
\(\mathcal D(\alpha,\beta,\eta,\phi^\vee)\) to the middle side has density
\[
        |H_n(\alpha,\beta;\nu)|\,
        |H_n(\eta,\phi^\vee;\nu^\vee)|
\]
with respect to coordinate measure
\(d\nu_0=d\nu_1\cdots d\nu_{n-1}\) on the trace hyperplane.  Thus, for
normalized Lebesgue measure on the double hive polytope, the middle-side
density is
\[
        \frac{
        |H_n(\alpha,\beta;\nu)|\,
        |H_n(\eta,\phi^\vee;\nu^\vee)|
        }{Z_{\mathcal D}},
        \qquad
        Z_{\mathcal D}
        =
        \int
        |H_n(\alpha,\beta;\xi)|\,
        |H_n(\eta,\phi^\vee;\xi^\vee)|\,d\xi_0 .
\]

We now compute the same density from the random matrix model.  Let
\(p_{\alpha,\beta}(\nu)\) denote the density of \(\spec(S)\), and let
\(p_{\eta,\phi^\vee}(\xi)\) denote the density of \(\spec(T)\).  By the
Knutson--Tao formula,
\[
        p_{\alpha,\beta}(\nu)
        =
        \frac{V_n(\nu)V_n(\tau_n)}{V_n(\alpha)V_n(\beta)}
        |H_n(\alpha,\beta;\nu)|,
\]
and
\[
        p_{\eta,\phi^\vee}(\xi)
        =
        \frac{V_n(\xi)V_n(\tau_n)}{V_n(\eta)V_n(\phi^\vee)}
        |H_n(\eta,\phi^\vee;\xi)| .
\]

All four fixed spectra are traceless, so \(S,T\) take values in
\[
\Herm_n^0:=\{X\in\Herm_n:\operatorname{Tr}X=0\}.
\]
Equip this space with coordinate Lebesgue measure
\[
dX_0=\prod_{i=1}^{n-1}dX_{ii}
\prod_{1\leq i<j\leq n}
d(\operatorname{Re}X_{ij})\,d(\operatorname{Im}X_{ij}),
\qquad X_{nn}=-\sum_{i<n}X_{ii},
\]
and equip the traceless spectral chamber with
\(d\nu_0=d\nu_1\cdots d\nu_{n-1}\).
Let \(q_{\alpha,\beta}(X)\) and \(q_{\eta,\phi^\vee}(X)\) be the corresponding
densities on \(\Herm_n^0\) with respect to \(dX_0\).
These densities are conjugation-invariant.  Put
\[
        N=\frac{n(n-1)}{2}.
\]
With these measures, the trace-zero Weyl formula has factor
\(\pi^N/V_n(\tau_n)\), as in Lemma~\ref{lem:Weyl}.
Indeed, writing \(t=\operatorname{Tr}X\),
elimination of \(X_{nn}\) gives \(dX=dt\,dX_0\); elimination of
\(\nu_n\) likewise gives \(d\nu=dt\,d\nu_0\).  Disintegrating the
full-space Weyl formula in \(t\) therefore introduces no additional
factor.  Consequently,
\[
        p_{\alpha,\beta}(\nu)
        =
        \frac{\pi^N}{V_n(\tau_n)}
        q_{\alpha,\beta}(\operatorname{diag}\nu)\,V_n(\nu)^2 .
\]
Consequently,
\[
        q_{\alpha,\beta}(\operatorname{diag}\nu)
        =
        \frac{V_n(\tau_n)^2}
             {\pi^N V_n(\alpha)V_n(\beta)}
        \frac{|H_n(\alpha,\beta;\nu)|}{V_n(\nu)} .
\]
Similarly, since \(V_n(\nu^\vee)=V_n(\nu)\) and
\(V_n(\phi^\vee)=V_n(\phi)\),
\[
        q_{\eta,\phi^\vee}(-\operatorname{diag}\nu)
        =
        q_{\eta,\phi^\vee}(\operatorname{diag}\nu^\vee)
        =
        \frac{V_n(\tau_n)^2}
             {\pi^N V_n(\eta)V_n(\phi)}
        \frac{|H_n(\eta,\phi^\vee;\nu^\vee)|}{V_n(\nu)} .
\]

Because \(S\) and \(T\) are independent, their joint density gives a
version of the conditional density of \(S\) given \(S+T=Z\) proportional
to
\[
q_{\alpha,\beta}(X)q_{\eta,\phi^\vee}(Z-X).
\]
We use its density-defined value at \(Z=0\): the probability measure
on \(\Herm_n^0\), with respect to \(dX_0\), whose density is
\[
        \frac{
        q_{\alpha,\beta}(X)\,q_{\eta,\phi^\vee}(-X)
        }{C_0},
        \qquad
        C_0=
        \int_{\Herm_n^0}
        q_{\alpha,\beta}(Y)\,q_{\eta,\phi^\vee}(-Y)\,dY_0 .
\]
The calculation below shows that \(0<C_0<\infty\), so this specifies
the version at the probability-zero conditioning event unambiguously.
Pushing this conditional measure forward by the eigenvalue map and applying
the trace-zero Weyl formula again, the resulting density with respect to
\(d\nu_0\) of
\(\nu=\spec(S)\) is
\[
        \frac{\pi^N}{V_n(\tau_n)C_0}
        q_{\alpha,\beta}(\operatorname{diag}\nu)\,
        q_{\eta,\phi^\vee}(-\operatorname{diag}\nu)\,
        V_n(\nu)^2 .
\]
Substituting the preceding expressions gives
\[
        \frac{V_n(\tau_n)^3}
             {\pi^N C_0
              V_n(\alpha)V_n(\beta)V_n(\eta)V_n(\phi)}
        |H_n(\alpha,\beta;\nu)|\,
        |H_n(\eta,\phi^\vee;\nu^\vee)| .
\]
Finally, we compute \(C_0\) by Weyl integration.  Since
\(q_{\alpha,\beta}(Y)q_{\eta,\phi^\vee}(-Y)\) is conjugation-invariant,
\[
\begin{aligned}
C_0
&=
\frac{\pi^N}{V_n(\tau_n)}
\int
q_{\alpha,\beta}(\operatorname{diag}\xi)\,
q_{\eta,\phi^\vee}(-\operatorname{diag}\xi)\,
V_n(\xi)^2\,d\xi_0  \\
&=
\frac{\pi^N}{V_n(\tau_n)}
\frac{V_n(\tau_n)^4}
     {\pi^{2N}V_n(\alpha)V_n(\beta)V_n(\eta)V_n(\phi)}
\int
|H_n(\alpha,\beta;\xi)|\,
|H_n(\eta,\phi^\vee;\xi^\vee)|\,d\xi_0  \\
&=
\frac{V_n(\tau_n)^3}
     {\pi^N V_n(\alpha)V_n(\beta)V_n(\eta)V_n(\phi)}
Z_{\mathcal D}.
\end{aligned}
\]
Hence the conditional spectral density is exactly
\[
        \frac{
        |H_n(\alpha,\beta;\nu)|\,
        |H_n(\eta,\phi^\vee;\nu^\vee)|
        }{Z_{\mathcal D}}.
\]
This agrees exactly with the middle-side density obtained from normalized
Lebesgue measure on the double hive polytope.

Thus the two distributions coincide.

\end{proof}

\phantomsection\label{proof:lem:fixed-spectrum-fixed-diagonal}
\begin{proof}[Proof of Lemma~\ref{lem:fixed-spectrum-fixed-diagonal}]
Condition first on the fixed source \(D\).  The density of \(X\) is
proportional to
\[
        \exp\left(
        -\frac{1}{2w^2}\operatorname{Tr}(X-uD)^2
        \right).
\]
Expanding the square gives
\[
        -\frac{1}{2w^2}\operatorname{Tr}(X-uD)^2
        =
        -\frac{1}{2w^2}\operatorname{Tr}(X^2)
        +
        \frac{u}{w^2}\operatorname{Tr}(XD)
        -
        \frac{u^2}{2w^2}\operatorname{Tr}(D^2).
\]
On the event \(\spec(X)=\lambda\), the term
\(\operatorname{Tr}(X^2)\) is fixed.  On the event \(\diag(X)=a\), the term
\[
        \operatorname{Tr}(XD)
        =
        \sum_{i=1}^n d_i X_{ii}
        =
        \sum_{i=1}^n d_i a_i
\]
is also fixed.  Hence, conditional on \(\spec(X)=\lambda\) and
\(\diag(X)=a\), the deformed Gaussian density is constant on the corresponding
slice of the unitary orbit.

Therefore the conditional law is the same as Haar measure on the unitary orbit
with spectrum \(\lambda\), further conditioned to have diagonal \(a\).  By the
Gelfand-Tsetlin description of the orbital minor process, the pushforward of
this conditional Haar measure under the principal-minor eigenvalue map is
Lebesgue measure on
\[
        \GT_{\diag(\lambda)\rel a}.
\]
After normalization, this gives the uniform distribution on that fiber.
\end{proof}

\phantomsection\label{proof:prop:3}
\begin{proof}[Proof of Lemma~\ref{prop:3}]
The minor process \(\Gamma_n\) is a measurable function of \(X_n\).
Hence, by the data-processing inequality (see Lemma~\ref{lem:data-processing}),
\[
I(D_n;\Gamma_n)\leq I(D_n;X_n).
\]
Multiplication by the nonzero scalar \(\sqrt n\) is invertible, so
\[
I(D_n;X_n)
=
I(D_n;wG_n+uD_n).
\]

Write
\[
g=\operatorname{diag}(G_n)\sim N(0,I_n).
\]
The off-diagonal entries of \(G_n\) are independent of both \(D_n\)
and \(g\), and contain no information about \(D_n\). Therefore
\[
I(D_n;wG_n+uD_n)
=
I(D_n;wg+uD_n).
\]
Thus it remains to estimate the mutual information of the Gaussian
channel
\[
Y=uD_n+wg.
\]

Let
\[
K_D=\operatorname{Cov}(D_n).
\]
Since \(g\) is independent of \(D_n\), the covariance of the channel output is
\[
\operatorname{Cov}(Y)
=
u^2K_D+w^2I_n.
\]
Since Gaussian distributions maximize differential entropy among
random vectors with prescribed covariance,
\[
\mathrm{ent}(Y)
\leq
\frac12\log\left(
(2\pi e)^n
\det(u^2K_D+w^2I_n)
\right).
\]
On the other hand,
conditioning on \(D_n\) leaves only the Gaussian noise \(wg\): for each value
of \(D_n\), the conditional law of \(Y\) is a translate of \(wg\).  Translation
does not change differential entropy, and hence
\[
\mathrm{ent}(Y\mid D_n)=\mathrm{ent}(wg)
=
\frac12\log\left((2\pi e)^nw^{2n}\right).
\]
Therefore, using \(I(D_n;Y)=\mathrm{ent}(Y)-\mathrm{ent}(Y\mid D_n)\),
\[
I(D_n;Y)
\leq
\frac12
\log
\frac{\det(u^2K_D+w^2I_n)}{w^{2n}}.
\]
Factoring \(w^2\) out of the determinant gives
\[
\det(u^2K_D+w^2I_n)
=
w^{2n}
\det\left(
I_n+\frac{u^2}{w^2}K_D
\right),
\]
and consequently
\[
I(D_n;Y)
\leq
\frac12\log\det\left(
I_n+\frac{u^2}{w^2}K_D
\right).
\]
Using the arithmetic--geometric mean inequality for the eigenvalues
of a positive semidefinite matrix,
\[
\det(I_n+tK_D)
\leq
\left(
1+\frac{t}{n}\operatorname{Tr}K_D
\right)^n,
\qquad t\geq 0.
\]
It follows that
\[
I(D_n;Y)
\leq
\frac n2
\log\left(
1+\frac{u^2}{nw^2}\operatorname{Tr}K_D
\right).
\]

Let \(\widetilde G_n\) be the independent GUE matrix whose ordered
spectrum is \(D_n\). Then
\[
\operatorname{Tr}K_D
=
\mathbb E\|D_n-\mathbb ED_n\|_2^2
\leq \sum_i \var \,d_i.
\]
Note that the distribution of eigenvalues of a GUE is log-concave, and since marginals of log-concave measures are log-concave, it follows that each $d_i$ has a log-concave distribution, which hence is sub-exponential.
Together with  Lemma~\ref{lem:rigidity},  this gives us
\[
\sum_i \var \, d_i = O(\log^{O(1)} n).
\]
Therefore
\[
I(D_n;\Gamma_n)
\leq
I(D_n;Y)
\leq
\frac n2
\log\left(1+\frac{u^2\log^{O(1)} n}{w^2 }\right).
\]
For fixed \(u\) and \(w>0\), the right-hand side is
\(O(n \log\log n)\), proving the claim.
\end{proof}

\phantomsection\label{proof:lem:12}
\begin{proof}[Proof of Lemma~\ref{lem:12}]
We first identify the distribution of the spectrum of \(Z_n\).
A linear combination of two independent GUE matrices is again GUE,
with variance equal to the sum of the variances. Hence
\[
wG_n+u\widetilde G_n
\stackrel{d}{=}
aG_n,
\qquad
a=\sqrt{w^2+u^2}.
\]
Consequently,
\[
\operatorname{spec}(Z_n)
\stackrel{d}{=}
\operatorname{spec}(aG_n).
\]

We now compute the expected logarithmic Vandermonde for a GUE
matrix. Introduce the Gaussian Mehta integral
\[
I_n(\beta,a)
=
\int_{\mathbb R^n}
\exp\left(
-\frac{1}{2a^2}\sum_{i=1}^n x_i^2
\right)
V_n(x)^\beta\,dx,
\]
where \(V_n\) is the sorted Vandermonde from \eqref{eq:sorted-vandermonde},
namely
\[
V_n(x)=\prod_{1\leq i<j\leq n}(x^\downarrow_i-x^\downarrow_j).
\]
The Mehta integral formula \cite{Mehta} gives
\[
I_n(\beta,a)
=
(2\pi)^{n/2}
a^{\,n+\beta n(n-1)/2}
\prod_{j=1}^n
\frac{\Gamma(1+j\beta/2)}
     {\Gamma(1+\beta/2)}.
\]
At \(\beta=2\), the normalized integrand is the unordered
eigenvalue density of \(aG_n\). Differentiating under the integral
sign therefore yields
\[
\mathbb E\log V_n(\operatorname{spec}(aG_n))
=
\left.
\frac{\partial}{\partial\beta}
\log I_n(\beta,a)
\right|_{\beta=2}.
\]
Writing \(\psi=\Gamma'/\Gamma\), we obtain
\[
\mathbb E\log V_n(\operatorname{spec}(aG_n))
=
\frac{n(n-1)}{2}\log a
+\frac12\sum_{j=1}^n j\psi(j+1)
-\frac n2\psi(2).
\]

Using
\[
\psi(j+1)=H_j-\gamma,
\qquad
\psi(2)=1-\gamma,
\]
together with
\[
\sum_{j=1}^n jH_j
=
\frac{n(n+1)}{2}H_n-\frac{n(n-1)}{4},
\]
gives the exact formula
\[
\begin{aligned}
\mathbb E\log V_n(\operatorname{spec}(aG_n))
={}&
\frac{n(n-1)}{2}\log a
+\frac{n(n+1)}{4}H_n  \\
&-\frac{n(n-1)}{8}
-\frac{\gamma n(n-1)}{4}
-\frac n2.
\end{aligned}
\]
Finally,
\[
H_n=\log n+\gamma+O(n^{-1}).
\]

Combining these estimates proves
\[
\mathbb E\log V_n(\operatorname{spec}Z_n)
=
\frac{n^2}{4}\log n
+\frac{n^2}{2}\log a
-\frac{n^2}{8}
+O(n\log n).
\]
\end{proof}

\phantomsection\label{proof:prop:4}
\begin{proof}[Proof of Proposition~\ref{prop:4}]
First condition on
\[
D=\operatorname{diag}(d_1,\ldots,d_n)
\]
and temporarily remove the outer factor \(\sqrt n\), writing
\[
Z=wG+uD.
\]
For a Gelfand-Tsetlin pattern
\[
\Gamma
=
\bigl(\lambda^{(1)},\ldots,\lambda^{(n)}\bigr),
\]
define
\[
q_k
=
\sum_{i=1}^{k}\lambda_i^{(k)}
-
\sum_{i=1}^{k-1}\lambda_i^{(k-1)},
\qquad 1\leq k\leq n,
\]
The empty sum is zero. The vector \(q=(q_1,\ldots,q_n)\) is the diagonal of \(Z\).

By the fixed-spectrum and fixed-diagonal description of the deformed
minor process, its density conditional on \(D\), with respect to
Lebesgue measure on the Gelfand-Tsetlin cone, is
\[
p_D(\Gamma)
=
(2\pi)^{-n/2}w^{-n^2}
V_n\bigl(\lambda^{(n)}\bigr)
\exp\left\{
-\frac{\left|\lambda^{(n)}\right|^2}{2w^2}
+\frac{u}{w^2}\sum_{k=1}^n d_kq_k
-\frac{u^2}{2w^2}|d|^2
\right\}.
\]
Consequently,
\begin{align*}
\mathrm{ent}(\Gamma\mid D)
={}&
\frac n2\log(2\pi)+n^2\log w
+\frac{1}{2w^2}\,
  \mathbb{E}\operatorname{Tr}(Z^2) \\
&-
\frac{u}{w^2}\,
  \mathbb{E}\operatorname{Tr}(ZD)
+\frac{u^2}{2w^2}\operatorname{Tr}(D^2)
-\mathbb{E}\log V_n(\operatorname{spec}Z).
\end{align*}
Since
\[
\mathbb{E}\operatorname{Tr}(Z^2)
=
w^2n^2+u^2\operatorname{Tr}(D^2)
\]
and
\[
\mathbb{E}\operatorname{Tr}(ZD)
=
u\operatorname{Tr}(D^2),
\]
all explicit terms involving the external source cancel. Thus
\[
\mathrm{ent}(\Gamma\mid D)
=
\frac n2\log(2\pi)
+n^2\log w
+\frac{n^2}{2}
-\mathbb{E}\log V_n(\operatorname{spec}Z).
\]


For a fixed realization \(D=\operatorname{diag}(d_1,\ldots,d_n)\), let
\(\mathbb E_G\) denote expectation over the GUE matrix \(G\) only. The
preceding computation gives
\[
\mathrm{ent}(\Gamma\mid D)
=
\frac n2\log(2\pi)
+n^2\log w
+\frac{n^2}{2}
-
\mathbb E_G
\log V_n(\operatorname{spec}(wG+uD)).
\]
Taking expectation over \(D\), we obtain
\[
\mathrm{ent}(\Gamma\mid D_n)
=
\frac n2\log(2\pi)
+n^2\log w
+\frac{n^2}{2}
-
\mathbb E_{G,D_n}
\log V_n(\operatorname{spec}(wG+uD_n)).
\]

Applying Lemma~\ref{lem:12} to the last expectation, the unscaled
minor process \(\Gamma_n^Z\) of \(Z=wG+uD_n\) satisfies
\[
\mathrm{ent}(\Gamma_n^Z\mid D_n)
=
-\frac{n^2}{4}\log n
+n^2\left(\log w-\frac12\log a+\frac58\right)
+O(n\log n).
\]
Since \(X_n=\sqrt n Z\), its minor process satisfies
\(\Gamma_n=\sqrt n\,\Gamma_n^Z\) in \(m_n=n(n+1)/2\) coordinates. Hence
\[
\mathrm{ent}(\Gamma_n\mid D_n)
=
\mathrm{ent}(\Gamma_n^Z\mid D_n)+\frac{m_n}{2}\log n
=
\frac{n^2}{2}
\left(
\frac54+\log\frac{w^2}{a}
\right)
+O(n\log n).
\]
Finally, by the chain rule for mutual information and
Lemma~\ref{prop:3},
\[
\mathrm{ent}(\Gamma_n)
=
\mathrm{ent}(\Gamma_n\mid D_n)+I(D_n;\Gamma_n)
=
\mathrm{ent}(\Gamma_n\mid D_n)+O(n\log\log n).
\]
Thus
\[
\mathrm{ent}(\Gamma_n)
=
\frac{n^2}{2}
\left(
\frac54+\log\frac{w^2}{\sqrt{w^2+u^2}}
\right)
+O(n\log n).
\]
\end{proof}

\phantomsection\label{proof:lem:one-sided-vandermonde}
\begin{proof}[Proof of Lemma~\ref{lem:one-sided-vandermonde}]
For \(i<j\),
\[
\widetilde\lambda_i-\widetilde\lambda_j
\geq
\left(1-\frac1n\right)(\lambda_i-\lambda_j).
\]
Therefore
\[
V_n(\widetilde\lambda_n)
\geq
\left(1-\frac1n\right)^{\binom n2}
V_n(\lambda_n).
\]
Taking logarithms gives
\[
\log V_n(\widetilde\lambda_n)
\geq
\log V_n(\lambda_n)
+
\binom n2\log\left(1-\frac1n\right).
\]
Since
\[
\binom n2\log\left(1-\frac1n\right)=O(n),
\]
the result follows.
\end{proof}

\phantomsection\label{proof:lem:random-side-entropy}
\begin{proof}[Proof of Lemma~\ref{lem:random-side-entropy}]
The density of \(R_n=\operatorname{spec}(\sqrt n\,\check G_n)\) on the
ordered Weyl chamber is
\[
\frac{1}{(2\pi)^{n/2}n^{n^2/2}V_n(\tau_n)}
V_n(r)^2\exp\left(-\frac{|r|^2}{2n}\right).
\]
Therefore
\begin{align*}
\ent(R_n)
={}&\frac n2\log(2\pi)+\frac{n^2}{2}\log n
 +\log V_n(\tau_n)\\
&-2\E\log V_n(R_n)+\frac1{2n}\E|R_n|^2.
\end{align*}
The exact GUE logarithmic Vandermonde formula in the proof of
Lemma~\ref{lem:12}, with scale \(\sqrt n\), and the expansion
\(H_n=\log n+\gamma+(2n)^{-1}+O(n^{-2})\) give
\[
\E\log V_n(R_n)
=\frac{n^2}{2}\log n-\frac{n^2}{8}+O(n\log n).
\]
Also, Stirling's formula, summed over
\(V_n(\tau_n)=\prod_{k=1}^{n-1}k!\), gives
\[
\log V_n(\tau_n)
=\frac{n^2}{2}\log n-\frac{3n^2}{4}+O(n\log n).
\]
Together with \(\E|R_n|^2=n^3\), these two expansions show explicitly that
the terms of order \(n^2\log n\) and \(n^2\) in the entropy cancel.  Hence
\[
\bigl|\ent(R_n)\bigr|=O(n\log n).
\]

The top row of a complete minor process is its matrix spectrum.  The
classical term \(\lambda_n^{\rm cl}/n\) is deterministic, and, conditionally
on \(\Gamma_n^X\), the correction \(A_n\) is a function of the fixed top
spectrum \(\lambda_n\) and the independent top spectrum \(\lambda'_n\).
In particular, \(A_n\) and \(R_n\) are conditionally independent.
Translation invariance and the scaling rule for differential entropy give
\begin{align*}
\ent(L_n\mid\Gamma_n^X)
={}&10n\log n\\
&+\ent\left(
\left(1-\frac1n\right)R_n
+n^{-10}A_n\tau_n
\,\middle|\,\Gamma_n^X
\right).
\end{align*}
For independent random vectors \(U,V\) with well-defined differential
entropies,
\[
\ent(U+V)\geq \ent(V),
\]
because conditioning on \(U\) gives
\(\ent(U+V\mid U)=\ent(V)\).  Taking
\(V=(1-1/n)R_n\), which is independent of the other summands conditionally
on \(\Gamma_n^X\), supplies an \(-O(n\log n)\) lower bound for the last
display.  For the upper bound, the conditional Gaussian maximum-entropy
inequality, the arithmetic--geometric mean inequality, and Jensen's
inequality give the corresponding estimate in terms of the expected
conditional covariance.  The unconditional form, for a random vector \(Z\)
with finite covariance, is
\[
\ent(Z)
\leq
\frac n2\log\left(
2\pi e\,\frac{\operatorname{Tr}\operatorname{Cov}(Z)}{n}
\right).
\]
The elementary estimates
\(\E|R_n|^2=n^3\), \(|\tau_n|^2=O(n^3)\), and
\eqref{eq:exceptional-correction-tail} show that the averaged conditional
upper bound is \(O(n\log n)\).  This proves
\[
\bigl|\ent(L_n\mid\Gamma_n^X)\bigr|=O(n\log n).
\]

Conditioning instead on both complete minor processes fixes
\(\lambda_n\), \(\lambda'_n\), and hence \(A_n\).  Since \(R_n\) is
independent of them, the correction is then only a translation, and
translation invariance and scaling give the exact identity
\[
\ent(L_n\mid\Gamma_n^X,\Gamma_n^Y)
=\ent(R_n)
+n\log\left(n^{10}\left(1-\frac1n\right)\right).
\]

It remains to record the measure-preserving assertion.  Use the boundary
increments of \(L_n\), followed by the standard free Gelfand--Tsetlin
coordinates, on the source.  On each fixed-\(L_n\) fiber,
Proposition~\ref{gt-rem}(iv) has Jacobian of absolute determinant one.
The displayed formula in that proposition shows that the extension in the
\(L_n\) coordinates is triangular and has identity diagonal block.  Its
full Jacobian therefore also has absolute determinant one.  For a pasted
double hive the same argument applies to the product of the two fixed-side
maps, with the common \(L_n\) block retained only once. Negating and
reversing the GT rows and reflecting the second hive, as in
Lemma~\ref{lem:doubleKLhive}, preserve the same coordinate measure.
The one-hive boundary-ordering commutor is also volume-preserving and
retains all boundary values, including \(L_n\), so the full determinant
remains of absolute value one after that composition. The entropy
identities now follow from the change-of-variables formula and the chain
rule.  Finally, \(\Gamma_n^X\) and \(\Gamma_n^Y\) are independent because
their defining random matrices are independent.
\end{proof}

\phantomsection\label{proof:prop:large-side-moments}
\begin{proof}[Proof of Proposition~\ref{prop:large-side-moments}]
Let \(r_n=\spec(G_n)\) and \(m_n=\E r_n\).  The GUE normalization gives
\[
\E|r_n|^2=\E\Tr(G_n^2)=n^2.
\]
As in the proof of Proposition~\ref{prop:4}, eigenvalue rigidity and the
subexponential tails of the ordered eigenvalues give
\[
\sum_{i=1}^n\Var(r_n(i))=O(\log^{O(1)}n).
\]
Consequently,
\begin{equation}\lab{eq:mean-spectrum-norm}
|m_n|^2
=\E|r_n|^2-\sum_{i=1}^n\Var(r_n(i))
=n^2+O(\log^{O(1)}n).
\end{equation}

Put
\[
\rho_n=\frac{\lambda_n^{\rm cl}}{n},
\qquad c_n=1-\frac1n.
\]
Thus \(L_n^{(0)}=n^{10}(\rho_n+c_nR_n)\).  The semicircle
quantile definition of the classical spectrum gives
\[
|\rho_n|=O(\sqrt n),
\qquad
|\E R_n|=|\sqrt n\,m_n|=O(n^{3/2}),
\]
while \(\E|R_n|^2=n^3\).  Hence
\[
\E|\rho_n+c_nR_n|^2
=c_n^2n^3+2c_n\langle\rho_n,\sqrt n\,m_n\rangle+|\rho_n|^2
=n^3+O(n^2).
\]
It follows that
\[
\frac{\E|L_n^{(0)}|^2}{n^2}=n^{21}+O(n^{20}).
\]
By \eqref{eq:exceptional-correction-tail}, Cauchy--Schwarz, and the
Gaussian moment bounds for \(L_n^{(0)}\), replacing \(L_n^{(0)}\) by
\(L_n=L_n^{(0)}+A_n\tau_n\) changes this normalized quadratic moment, and
each normalized mixed moment below, by less than any inverse power of
\(n\).  This proves \eqref{eq:large-side-Q}.

Let \(\widetilde r_n=\spec(\widetilde G_n)\) and let
\(\check r_n=\spec(\check G_n)\), so that \(R_n=\sqrt n\,\check r_n\).
Writing \(g_n^{\rm diag}\) for the vector of diagonal entries of \(G_n\),
we have
\[
\diag(X_n)=\sqrt n\,(w g_n^{\rm diag}+u\widetilde r_n).
\]
The three vectors \(\check r_n,g_n^{\rm diag},\widetilde r_n\) are
independent, \(\E g_n^{\rm diag}=0\), and both ordered spectra have mean
\(m_n\).  Therefore
\begin{equation}\lab{eq:auxiliary-mixed-exact}
\E\langle R_n,\diag(X_n)\rangle
=nu|m_n|^2
=un^3+O\!\left(n\log^{O(1)}n\right).
\end{equation}
On the other hand,
\[
\left|\E\langle\rho_n,\diag(X_n)\rangle\right|
=\sqrt n\,u\,|\langle\rho_n,m_n\rangle|
=O(n^2).
\]
Multiplying these two estimates by the coefficients in \(L_n^{(0)}\),
dividing by \(n^2\), and using \(c_n=1-n^{-1}\), proves the first formula
in \eqref{eq:large-side-mixed-moments}.  The second follows identically
from the independent primed matrices.

Finally, direct calculation and independence give the exact identities
\[
\frac{\E|\diag(X_n)|^2}{n^2}
=\frac{n\{w^2\E|g_n^{\rm diag}|^2
             +u^2\E|\widetilde r_n|^2\}}{n^2}
=w^2+u^2n
\]
and its primed analogue.  Likewise,
\[
\frac{\E|\lambda_n|^2}{n^2}
=\frac{\E\Tr(X_n^2)}{n^2}
=(w^2+u^2)n=a^2n,
\]
with the analogous identity for \(\lambda'_n\).  Expanding
\(|L_n+\diag(X_n)|^2\) and
\(|L_n-\diag(Y_n)|^2\) now proves
\eqref{eq:large-side-exterior-moments}.
\end{proof}

\phantomsection\label{proof:lem:KLhive}
\begin{proof}[Proof of Lemma~\ref{lem:KLhive}]
By the definition of triply augmented hives,
\[
\mathbb A^3(\lambda,M;N)
=
GT(\lambda)\times GT(M)\times GT(N)\times H_n(\lambda,M;N).
\]
Since the attached GT patterns are conditionally independent and uniform,
\[
\mathrm{ent}(q_n)
=
\mathrm{ent}(h_n)
+
\mathbb E\log\frac{V_n(\lambda_n)}{V_n(\tau_n)}
+
\mathbb E\log\frac{V_n(M_n)}{V_n(\tau_n)}
+
\mathbb E\log\frac{V_n(N_n)}{V_n(\tau_n)}.
\]
Thus
\[
\mathrm{ent}(q_n)
=
\mathrm{ent}(h_n)
+
\mathbb E\log V_n(\lambda_n)
+
\mathbb E\log V_n(M_n)
+
\mathbb E\log V_n(N_n)
-
3\log V_n(\tau_n).
\]

By Lemma~\ref{lem:random-side-entropy} and Proposition~\ref{prop:4},
\[
\mathrm{ent}(h_n)
=
\frac{n^2}{2}
\left(
\frac54+
\log\frac{w^2}{\sqrt{w^2+u^2}}
\right)
+
O(n\log n).
\]

We now estimate the two large-gap Vandermonde terms. Since
\[
M_n=L_n,
\qquad
N_n=L_n+\operatorname{diag}(X_n)
=L_n+\operatorname{diag}\bigl(\sqrt n\,(wG_n+uD_n)\bigr),
\]
and \(L_n^{(0)}=n^{10}\widetilde\lambda_n\), the correction vanishes with
overwhelming probability by \eqref{eq:exceptional-correction-tail}.  On that
event the perturbation taking \(L_n=L_n^{(0)}\) to \(N_n\) has size
\(O(n^{1+\varepsilon})\) with overwhelming probability.  On the other hand,
the definition of the uncorrected side gives, for
\(i<j\),
\[
\begin{aligned}
\widetilde\lambda_i-\widetilde\lambda_j
={}&
\frac{1}{n}(\lambda_i^{\mathrm{cl}}-\lambda_j^{\mathrm{cl}})
+
\left(1-\frac1n\right)(\check\lambda_i-\check\lambda_j).
\end{aligned}
\]
Both spectra are ordered decreasingly, and hence
\[
\widetilde\lambda_i-\widetilde\lambda_j
\geq
\frac{1}{n}(\lambda_i^{\mathrm{cl}}-\lambda_j^{\mathrm{cl}}).
\]
The deterministic classical gaps for the spectrum of \(\sqrt n\,G_n\)
satisfy
\[
\lambda_i^{\mathrm{cl}}-\lambda_{i+1}^{\mathrm{cl}}
\geq
c
\]
uniformly in \(i\).  Indeed, if
\(\lambda_i^{\mathrm{cl}}=n\gamma_i\), where \(\gamma_i\) is the
semicircle quantile, then the quantile gaps are smallest in the bulk, where
they are of order \(1/n\).  For non-adjacent pairs the gaps are larger
after changing the constant \(c\).  Hence
\[
L_i^{(0)}-L_j^{(0)}
=
n^{10}(\widetilde\lambda_i-\widetilde\lambda_j)
\geq
c n^{10}\cdot n^{-1}
=
c n^9.
\]
Therefore the perturbation from \(L_n\) to \(N_n\) changes each logarithmic
gap by \(O(n^{-8+\varepsilon})\) on an overwhelming-probability event.
Summing over \(O(n^2)\) pairs gives \(O(n^{-6+\varepsilon})\) on this event,
which is in particular \(O(n\log n)\).  The complement of the
overwhelming-probability event, including the event on which \(A_n>0\), is
absorbed by Gaussian tails, rigidity from Lemma~\ref{lem:rigidity}, and
\eqref{eq:exceptional-correction-tail}.  Thus
\[
\mathbb E\log V_n(M_n)
=
\mathbb E\log V_n(L_n),
\]
and
\[
\mathbb E\log V_n(N_n)
=
\mathbb E\log V_n(L_n)
+
O(n\log n).
\]

Next, adding \(A_n\tau_n\), whose adjacent gaps are nonnegative, can only
increase every ordered pairwise gap.  Therefore
\[
V_n(L_n)
\geq
V_n(L_n^{(0)})
=
n^{10\binom n2}V_n(\widetilde\lambda_n).
\]
The same definition of \(\widetilde\lambda_n\) also gives the deterministic
one-sided comparison
\[
\widetilde\lambda_i-\widetilde\lambda_j
\geq
\left(1-\frac1n\right)(\check\lambda_i-\check\lambda_j).
\]
Thus
\[
V_n(\widetilde\lambda_n)
\geq
\left(1-\frac1n\right)^{\binom n2}
V_n(\check\lambda_n),
\]
and hence
\[
\mathbb E\log V_n(\widetilde\lambda_n)
\geq
\mathbb E\log V_n(\check\lambda_n)+O(n).
\]
Consequently,
\[
\mathbb E\log V_n(L_n)
\geq
10\binom n2\log n
+
\mathbb E\log V_n(\check\lambda_n)
+
O(n).
\]

Combining these estimates, we obtain the lower bound
\[
\mathrm{ent}(q_n)
\geq
\mathrm{ent}(h_n)
+
\mathbb E\log V_n(\lambda_n) + 2 \mathbb E\log V_n(\check\lambda_n)
+
20\binom n2\log n
-
3\log V_n(\tau_n)
+
O(n\log n).
\]
We use the GUE log-Vandermonde asymptotic, obtained from the ordered Gaussian
Mehta integral \cite[Ch.~17]{Mehta} together with Claim~\ref{cl:vtau}, for
\[
\lambda_n\stackrel{d}{=}
\operatorname{spec}\bigl(\sqrt{nw^2+ nu^2}\,G_n\bigr),
\]
namely
\[
\mathbb E\log \frac{V_n(\lambda_n)}{V_n(\tau_n)}
=
\frac{n^2}{2}\log\sqrt{w^2+u^2}
+
\frac{5 n^2}{8}
+
O(n\log n).
\]
Since \(\check\lambda_n=\operatorname{spec}(\sqrt n\,\check G_n)\), the same
asymptotic with scale parameter \(1\) gives
\[
\mathbb E\log \frac{V_n(\check\lambda_n)}{V_n(\tau_n)}
=
\frac{5n^2}{8}
+
O(n\log n).
\]
\begin{claim} $q_n$ has the  same entropy as the
maximum-entropy density \(p_n\), up to an \(O(n\log n)\) error:
\[
\mathrm{ent}(q_n)\geq \mathrm{ent}(p_n)-O(n\log n).
\]
\end{claim}
\begin{proof}
Let \(a=\sqrt{w^2+u^2}\).  From the preceding estimates,
\[
\ent(q_n)\geq \ent(h_n)
+\E\log\frac{V_n(\lambda_n)}{V_n(\tau_n)}
+2\E\log\frac{V_n(L_n)}{V_n(\tau_n)}
+O(n\log n).
\]
By Proposition~\ref{prop:4},
\[
\ent(h_n)=\frac{n^2}{2}\left(\frac54+\log\frac{w^2}{a}\right)+O(n\log n).
\]
Also
\[
\E\log\frac{V_n(\lambda_n)}{V_n(\tau_n)}
=
\frac{n^2}{2}\log a+\frac58n^2+O(n\log n),
\]
and the preceding comparison gives
\[
\E\log\frac{V_n(L_n)}{V_n(\tau_n)}
\geq
10\binom n2\log n
+\frac58n^2
+O(n\log n).
\]

The comparison of \(N_n\) with \(L_n\) above gives the same bound for the
third large side:
\[
\E\log\frac{V_n(N_n)}{V_n(\tau_n)}
\geq
10\binom n2\log n
+\frac58n^2
+O(n\log n).
\]

Therefore
\[
\ent(q_n)\geq
10n^2\log n
+n^2\log w
+\frac52n^2
+O(n\log n).
\]

On the other hand, Theorem~\ref{thm:gaussian1} gives
\beq\lab{eq:thelastdisplay}
\ent(p_n)
=
n^2+n\log(2\pi)
+\frac{n^2}{2}
\log\left(
\frac{\bar a^2\bar b^2\bar c^2}
{\bar a^2+\bar b^2+\bar c^2}
\right)
-2\log V_n(\tau_n).
\eeq
Let
\[
X=\frac{\E_{q_n}|\lambda_n|^2}{n^2},\qquad
E=\frac{\E_{q_n}|M_n|^2}{n^2},\qquad
W=\frac{\E_{q_n}|N_n|^2}{n^2},
\]
and write \(A=\bar a^2\), \(B=\bar b^2\), \(C=\bar c^2\).  The moment
equations give
\[
X=\frac{A(B+C)}{A+B+C},\quad
E=\frac{B(C+A)}{A+B+C},\quad
W=\frac{C(A+B)}{A+B+C}.
\]
Proposition~\ref{prop:large-side-moments} gives the corresponding sizes in
the present scaling.  With \(a^2=w^2+u^2\) and
\(E=Q_n=\E|L_n|^2/n^2\), they are
\[
X=a^2n,\qquad
E=n^{21}+O(n^{20}),\qquad
W=E+2u\,n^{11}+O(n^{10}).
\]
If
\[
R=X+E-W,\qquad S=X+W-E,\qquad T=E+W-X,
\]
then
\[
\frac{ABC}{A+B+C}=\frac{RS+RT+ST}{4}.
\]
Let \(\delta=W-E\).  Then \(R=X-\delta\), \(S=X+\delta\), and
\(T=2E+\delta-X\).  Hence
\[
\frac{RS+RT+ST}{4}
=
\frac{X^2-\delta^2+2X(2E+\delta-X)}{4}.
\]
Using
\[
X=a^2n+O(1),\qquad
E=n^{21}+O(n^{20}),\qquad
\delta=2u\,n^{11}+O(n^{10}),
\]
we get
\[
\frac{RS+RT+ST}{4}
=
XE-\frac{\delta^2}{4}+O(n^{12})
=
(a^2-u^2)n^{22}+O(n^{21})
=
w^2n^{22}+O(n^{21}).
\]
Therefore
\[
\frac{\bar a^2\bar b^2\bar c^2}
{\bar a^2+\bar b^2+\bar c^2}
=
w^2n^{22}+O(n^{21}).
\]
Thus (\ref{eq:thelastdisplay}) equals
\[
10n^2\log n
+n^2\log w
+\frac52n^2
+O(n\log n).
\]
Hence
\[
\ent(q_n)\geq \ent(p_n)-O(n\log n),
\]
as claimed.
\end{proof}

Finally, Theorem~\ref{thm:gaussian1} says that \(p_n\) is the unique
maximum-entropy density on \(\mathbb A^3\) among all densities whose boundary
quadratic moments agree with those of \(p_n\).  By construction these
moments agree with those of \(q_n\):
\[
\E_{p_n}|\lambda|^2=\E_{q_n}|\lambda|^2,\qquad
\E_{p_n}|\mu|^2=\E_{q_n}|\mu|^2,\qquad
\E_{p_n}|\nu|^2=\E_{q_n}|\nu|^2.
\]
Thus \(q_n\) is one of the admissible densities in the variational problem
solved by \(p_n\), and hence
\[
\ent(q_n)\leq \ent(p_n).
\]

It remains to identify this entropy deficit with the KL divergence.
Since \(p_n\) has density
\[
p_n(x)
=
Z_n^{-1}
\exp\left[
-\frac12
\left(
\frac{|\lambda(x)|^2}{\bar a^2}
+
\frac{|\mu(x)|^2}{\bar b^2}
+
\frac{|\nu(x)|^2}{\bar c^2}
\right)
\right],
\]
we have
\[
-\log p_n(x)
=
\log Z_n
+
\frac12
\left(
\frac{|\lambda(x)|^2}{\bar a^2}
+
\frac{|\mu(x)|^2}{\bar b^2}
+
\frac{|\nu(x)|^2}{\bar c^2}
\right).
\]
Therefore
\[
-\E_{q_n}\log p_n
=
\log Z_n
+
\frac12
\left(
\frac{\E_{q_n}|\lambda|^2}{\bar a^2}
+
\frac{\E_{q_n}|\mu|^2}{\bar b^2}
+
\frac{\E_{q_n}|\nu|^2}{\bar c^2}
\right).
\]
Using the matching of quadratic moments, this equals
\[
\log Z_n
+
\frac12
\left(
\frac{\E_{p_n}|\lambda|^2}{\bar a^2}
+
\frac{\E_{p_n}|\mu|^2}{\bar b^2}
+
\frac{\E_{p_n}|\nu|^2}{\bar c^2}
\right)
=
-\E_{p_n}\log p_n
=
\ent(p_n).
\]
Hence
\[
D_{\mathrm{KL}}(q_n\Vert p_n)
=
\E_{q_n}\log\frac{q_n}{p_n}
=
-\ent(q_n)-\E_{q_n}\log p_n
=
\ent(p_n)-\ent(q_n).
\]
Combining this identity with the claim gives
\[
0\leq D_{\mathrm{KL}}(q_n\Vert p_n)
=
\ent(p_n)-\ent(q_n)
=
O(n\log n),
\]
as required.

\end{proof}

\phantomsection\label{proof:lem:double-moment-inverse}
\begin{proof}[Proof of Lemma~\ref{lem:double-moment-inverse}]
For \(S,t>0\), set
\[
 \psi_S(t)=\frac{\sqrt{S^2+4St}-S}{2}.
\]
This is the unique positive solution \(q\) of \(t=q+q^2/S\).
Thus, once \(S\) is fixed, the first three equations force
\[
 A=\psi_S(X),\qquad B=\psi_S(Y),\qquad D=\psi_S(Z).
\]
Put 
\[
 \Delta=W-X-Y-Z
 \quad\hbox{and}\quad
 H(S)=\frac{2}{S}(AB+AD+BD).
\]
The fourth equation is equivalent to \(H(S)=\Delta\).  For \(x,y>0\),
\begin{equation}\lab{eq:inverse-pair-function}
 \frac{\psi_S(x)\psi_S(y)}{S}
 =\frac{4xy}
 {(\sqrt{S+4x}+\sqrt S)(\sqrt{S+4y}+\sqrt S)},
\end{equation}
which is strictly decreasing in \(S\).  It follows that \(H\) is strictly
decreasing, with
\[
 \lim_{S\downarrow0}H(S)
 =2(\sqrt{XY}+\sqrt{XZ}+\sqrt{YZ}),
 \qquad
 \lim_{S\to\infty}H(S)=0.
\]
Condition \eqref{eq:inverse-moment-region} says precisely that \(\Delta\)
lies strictly between these two limits.  Hence \(H(S)=\Delta\) has one and
only one root, proving existence and uniqueness.

For completeness, this root is also a simple root of the direct scalar
inverse equation.  Define
\[
 \mathcal F(S)=\psi_S(X)+\psi_S(Y)+\psi_S(Z)-\psi_S(W)
\]
and write \(U=A+B+D\), \(V=\psi_S(W)\).  Since
\(\Phi_S(q)=q+q^2/S\), direct expansion gives
\[
 \Phi_S(U)-\Phi_S(V)
 =\mathcal F(S)\left(1+\frac{U+V}{S}\right)
 =H(S)-\Delta.
\]
At the root \(U=V\), and therefore
\begin{equation}\lab{eq:inverse-simple-root}
 \mathcal F'(S)=\frac{H'(S)}{1+2U/S}<0.
\end{equation}
This proves the asserted nonvanishing derivative.

We now prove the scaled conclusion.  Here
\[
 \Delta_n=2(u+u')n^{11}+O(n^{10}),
\]
whereas the upper endpoint of the range of \(H\) is
\[
 2(\sqrt{X_nY_n}+\sqrt{X_nZ_n}+\sqrt{Y_nZ_n})
 =2(a+b)n^{11}+o(n^{11}).
\]
Since \(u<a\) and \(u'<b\), condition
\eqref{eq:inverse-moment-region} holds for all sufficiently large \(n\).
Set \(S=ns\).  Uniformly for \(s\) in compact subsets of \((0,\infty)\),
equation \eqref{eq:inverse-pair-function} gives
\begin{equation}\lab{eq:inverse-scaled-equation}
 n^{-11}H_n(ns)=h(s)+O(n^{-1}),
 \qquad
 h(s)=\sqrt{s+4a^2}+\sqrt{s+4b^2}-2\sqrt s,
\end{equation}
and \(n^{-11}\Delta_n=2(u+u')+O(n^{-1})\).
For \(s_0=\theta^{-2}\),
\[
 \sqrt{s_0+4a^2}-\sqrt{s_0}=2u,
 \qquad
 \sqrt{s_0+4b^2}-\sqrt{s_0}=2u',
\]
so \(h(s_0)=2(u+u')\).  Moreover,
\[
 h'(s)=\frac{1}{2\sqrt{s+4a^2}}
       +\frac{1}{2\sqrt{s+4b^2}}
       -\frac{1}{\sqrt s}<0.
\]
The monotonicity already proved localizes the root near \(s_0\), and the
one-dimensional implicit-function estimate applied to
\eqref{eq:inverse-scaled-equation} gives
\(S_n/n=s_0+O(n^{-1})\).  Thus \(S_n=\theta^{-2}n+O(1)\).
Finally, the exact identity
\[
 \frac{A_n}{n}
 =\frac{\sqrt{(S_n/n)^2+4(S_n/n)a^2}-S_n/n}{2}
\]
and its analogue for \(B_n\) show, using \(s_0=\theta^{-2}\), that
\(A_n=w^2n+O(1)\) and \(B_n=(w')^2n+O(1)\).
\end{proof}

\phantomsection\label{proof:lem:doubleKLhive}
\begin{proof}[Proof of Lemma~\ref{lem:doubleKLhive}]
The oriented inputs are \(h_n^{(1)}\in H_n(L_n,\lambda_n;N_n)\)
and \(h_n^{(2)}\in H_n(P_n,\mu_n;L_n)\).  Thus the first hive has
exterior sides \(\lambda_n,N_n\), the second has exterior sides
\(\mu_n,P_n\), and their common side is \(L_n\).  The second input is
obtained from the minor process of \(-Y_n\) followed by the reflection
specified in the statement.  Negating and reversing every GT row and
reflecting the hive preserve coordinate Lebesgue measure, so the entropy
transfer of Lemma~\ref{lem:random-side-entropy} still applies with the
original independent minor processes of \(X_n,Y_n\).  Whenever a boundary
is dualized in identifying the two triangulations, the duality
\(v^\vee=(-v_n,\ldots,-v_1)\) preserves both \(|v|\) and \(V_n(v)\).

Since the four exterior Gelfand--Tsetlin patterns are conditionally
independent and uniform given the pasted double hive,
\begin{align*}
\ent(q_n^{\mathrm{dbl}})
={}&
\ent\!\left(h_n^{(1)},h_n^{(2)}\hbox{ pasted along }L_n\right)\\
&+
\E\log\frac{V_n(\lambda_n)}{V_n(\tau_n)}
+
\E\log\frac{V_n(\mu_n)}{V_n(\tau_n)}
+
\E\log\frac{V_n(P_n)}{V_n(\tau_n)}
+
\E\log\frac{V_n(N_n)}{V_n(\tau_n)} .
\end{align*}
There is no \(V_n(L_n)\) term, because the glued side is not decorated in
\(\mathbb D^4\).

By Lemma~\ref{lem:random-side-entropy}, Proposition~\ref{prop:4}, and the
independence of the two deformed minor processes,
\[
\ent\!\left(h_n^{(1)},h_n^{(2)}\hbox{ pasted along }L_n\right)
=
\frac{n^2}{2}\left(\frac54+\log\frac{w^2}{a}\right)
+
\frac{n^2}{2}\left(\frac54+\log\frac{(w')^2}{b}\right)
+
O(n\log n).
\]
The deformed GUE Vandermonde asymptotic gives
\[
\E\log\frac{V_n(\lambda_n)}{V_n(\tau_n)}
=
\frac{n^2}{2}\log a+\frac58n^2+O(n\log n),
\]
and
\[
\E\log\frac{V_n(\mu_n)}{V_n(\tau_n)}
=
\frac{n^2}{2}\log b+\frac58n^2+O(n\log n).
\]
The same large-gap Vandermonde comparison as in Lemma~\ref{lem:KLhive}
applied to \(N_n\) and \(P_n\) gives
\[
\E\log\frac{V_n(N_n)}{V_n(\tau_n)}
\geq
10\binom n2\log n+\frac58n^2+O(n\log n),
\]
and
\[
\E\log\frac{V_n(P_n)}{V_n(\tau_n)}
\geq
10\binom n2\log n+\frac58n^2+O(n\log n).
\]
Combining these estimates,
\beq \lab{eq:qn}
\ent(q_n^{\mathrm{dbl}})
\geq
10n^2\log n+n^2\log(ww')+\frac{15}{4}n^2+O(n\log n).
\eeq

Let \(A,B,D,F\) be the raw parameters of \(p_n^{\mathrm{dbl}}\) attached to
the exterior sides in clockwise cyclic order
\[
\lambda_n,\qquad \mu_n,\qquad P_n,\qquad N_n,
\]
respectively, and set \(T=A+B+D+F\).  Put
\[
X=\frac{\E|\lambda_n|^2}{n^2},\qquad
Y=\frac{\E|\mu_n|^2}{n^2},\qquad
Z=\frac{\E|P_n|^2}{n^2},\qquad
W=\frac{\E|N_n|^2}{n^2}.
\]
For the four exterior sides in the present construction,
Proposition~\ref{prop:large-side-moments} gives
\begin{equation}\lab{eq:doubleKL-small-moments}
X=a^2n,\qquad
Y=b^2n.
\end{equation}
Writing \(Q_n=\E|L_n|^2/n^2=n^{21}+O(n^{20})\), the same proposition gives
\begin{equation}\lab{eq:doubleKL-large-moments}
Z=Q_n-2u'n^{11}+O(n^{10}),\qquad
W=Q_n+2un^{11}+O(n^{10}).
\end{equation}
Set
\[
\theta=\frac{u}{w^2}=\frac{u'}{(w')^2}.
\]
When \(\theta=0\), the argument can be handled similarly; we omit the
details.  We therefore assume \(\theta>0\).  Lemma~\ref{lem:double-moment-inverse},
applied to \eqref{eq:doubleKL-small-moments} and
\eqref{eq:doubleKL-large-moments}, gives unique \(S,A,B,D>0\) and hence
\(T=-S<0\), \(F=-(S+A+B+D)<0\).  Their moment equations are
\begin{equation}\lab{eq:doubleKL-negative-moments}
X=A+\frac{A^2}{S},\qquad
Y=B+\frac{B^2}{S},\qquad
Z=D+\frac{D^2}{S},
\end{equation}
and
\begin{equation}\lab{eq:doubleKL-negative-W}
W=A+B+D+\frac{(A+B+D)^2}{S}.
\end{equation}
The scaled simple-root conclusion \eqref{eq:inverse-raw-asymptotics} gives
\[
A=w^2n+O(1),\qquad
B=(w')^2n+O(1),\qquad
S=\theta^{-2}n+O(1),
\]

Lemma~\ref{lem:double-partition}, in the \(F<0\), \(T<0\) interpretation
of Lemma~\ref{lem:double-maxent}, gives
\[
\ent(p_n^{\mathrm{dbl}})
=
\frac{3n^2}{2}
+\frac{3n}{2}\log(2\pi)
-3\log V_n(\tau_n)
+
\frac{n^2}{2}
\log\left(\frac{ABDF}{T}\right).
\]
It remains to evaluate the last logarithm.  Since
\[
\frac{ABDF}{T}
=
AB\,D\,\frac{S+A+B+D}{S}
=
AB\left(D+\frac{D(A+B)}{S}+\frac{D^2}{S}\right),
\]
and
\[
\frac{D^2}{S}=Q_n+O(n^{11})=n^{21}+O(n^{20}),\qquad
D=O(n^{11}),\qquad
\frac{D(A+B)}{S}=O(n^{11}),
\]
we get
\[
\frac{ABDF}{T}
=
w^2(w')^2n^{23}+O(n^{22}).
\]
Therefore, using Claim~\ref{cl:vtau},
\[
\log V_n(\tau_n)
=
\frac{n^2}{2}\log n-\frac34n^2+O(n\log n),
\]
we obtain
\beq \lab{eq:pn}
\ent(p_n^{\mathrm{dbl}})
=
10n^2\log n+n^2\log(ww')+\frac{15}{4}n^2+O(n\log n).
\eeq
Hence by (\ref{eq:qn}) and (\ref{eq:pn}), 
\beqs
\ent(q_n^{\mathrm{dbl}})
\geq
\ent(p_n^{\mathrm{dbl}})-O(n\log n).
\eeqs
On the other hand, \(p_n^{\mathrm{dbl}}\) is the maximum-entropy density on
\(\mathbb D^4\) with the same four exterior quadratic moments, so
\[
\ent(q_n^{\mathrm{dbl}})\leq \ent(p_n^{\mathrm{dbl}}).
\]

Finally, because \(p_n^{\mathrm{dbl}}\) has exponential-quadratic density
and the four exterior quadratic moments match, in a manner analogous to the proof of Lemma~\ref{lem:KLhive},
\[
D_{\mathrm{KL}}\!\left(q_n^{\mathrm{dbl}}\Vert p_n^{\mathrm{dbl}}\right)
=
\ent(p_n^{\mathrm{dbl}})-\ent(q_n^{\mathrm{dbl}}).
\]
If \(u=0\), then the condition
\[
\frac{u}{w^2}=\frac{u'}{(w')^2}
\]
also gives \(u'=0\).  In this degenerate case the preceding negative-raw
argument is replaced by the positive raw regime: the two large raw
parameters corresponding to the large-gap sides, namely \(D\) and \(F\), are
both positive and of order \(n^{21}\).  The same exterior moment computation
then gives
\[
\frac{ABDF}{A+B+D+F}
=
w^2(w')^2n^{23}+O(n^{22}),
\]
and hence the same entropy estimate for \(p_n^{\mathrm{dbl}}\).  Therefore
the KL bound below remains unchanged.

Combining the preceding two entropy bounds gives
\[
D_{\mathrm{KL}}\!\left(q_n^{\mathrm{dbl}}\Vert p_n^{\mathrm{dbl}}\right)
=O(n\log n),
\]
as claimed.
\end{proof}

\phantomsection\label{proof:lem:tet}
\begin{proof}[Proof of Lemma~\ref{lem:tet}]
Since $w,w'>0$, we have $a>u$ and $b>u'$, so $\delta<a+b$.
Together with $d+f>|a-b|$ and $d,f>0$, this ensures that the admissible
interval in the statement is nonempty. Put $y=x^2$. Up to the fixed factor
$(abdf)^{-1}$, the objective is
\[
\Psi_{d,f}(y):=
\frac{\Delta_{ab\sqrt y}^2\Delta_{\sqrt y\,d\,f}^2}{y}.
\]
On the interval
\[
I_{d,f}=\left(\max\{(a-b)^2,(d-f)^2\},
\min\{(a+b)^2,(d+f)^2\}\right),
\]
Heron's formula gives
\[
\Psi_{d,f}(y)=
\frac{(y-(a-b)^2)((a+b)^2-y)
(y-(d-f)^2)((d+f)^2-y)}{256y}.
\]
Therefore
\[
\frac{d}{dy}\log\Psi_{d,f}(y)
=\frac1{y-(a-b)^2}-\frac1{(a+b)^2-y}
+\frac1{y-(d-f)^2}-\frac1{(d+f)^2-y}-\frac1y.
\]
Its derivative is strictly negative on $I_{d,f}$: the derivative of either
positive reciprocal term is at most $-1/y^2$, which absorbs the derivative
$1/y^2$ of the last term, and the remaining terms have strictly negative
derivatives. The objective is positive inside the interval and extends
continuously to zero at its endpoints, including a possible endpoint $y=0$.
It consequently has a unique maximizer. Writing this maximizer as $c_*^2$, the tetrahedral identity
\eqref{eq:tetrahedral-identity} shows that $c_*$ is the unique admissible
solution of \eqref{eq:tet-cstar-equation}.

Now fix $f-d=\delta$ and let $f\to\infty$. For all sufficiently large $f$,
$I_{d,f}$ is the fixed interval
\[
I_\infty=\left(\max\{(a-b)^2,\delta^2\},(a+b)^2\right).
\]
Set
\[
\Psi_\infty(y)=\frac{\Delta_{ab\sqrt y}^2(y-\delta^2)}{y}.
\]
Since $(d+f)^2=4df+\delta^2$, the normalized objectives satisfy the exact
identity
\[
\frac{4\Psi_{d,f}(y)}{df}
=\Psi_\infty(y)\left(1+\frac{\delta^2-y}{4df}\right).
\]
After continuous extension at a possible endpoint $y=0$, this proves
uniform convergence on the compact closure of $I_\infty$. Moreover,
\[
\frac{d}{dy}\log\Psi_\infty(y)
=\frac1{y-(a-b)^2}-\frac1{(a+b)^2-y}
+\frac1{y-\delta^2}-\frac1y
\]
is strictly decreasing. The limiting objective is positive in the interior
and vanishes at both endpoints, so it has a unique maximizer $c_{**}^2$.
Uniform convergence on the closure then gives
\[
c_{**}=\lim_{\substack{f\to\infty\\ f-d=\delta}}c_*.
\]

To compute the limiting critical equation, write
\[
y=c_{**}^2,\qquad A=a^2,\qquad B=b^2,\qquad K=A+B.
\]
The correctly normalized triangle area is
\[
\Delta_{ab\sqrt y}^2=\frac{AB}{4}-\frac{(K-y)^2}{16},
\qquad
\partial_y\log\Delta_{ab\sqrt y}^2
=\frac{K-y}{8\Delta_{ab\sqrt y}^2}.
\]
Thus the critical equation is
\[
\frac{K-y}{8\Delta_{ab\sqrt y}^2}
+\frac1{y-\delta^2}-\frac1y=0.
\]
The last two terms have sum $\delta^2/[y(y-\delta^2)]\geq0$, which implies
$y\geq K$. Clearing denominators gives
\[
(K-y)y(y-\delta^2)+8\Delta_{ab\sqrt y}^2\delta^2=0,
\]
and hence
\[
\delta^2
=\frac{(y-K)y^2}{8\Delta_{ab\sqrt y}^2+y(y-K)}.
\]
Using
\[
8\Delta_{ab\sqrt y}^2+y(y-K)
=2AB-\frac{(K-y)^2}{2}+y(y-K)
=\frac{y^2-(A-B)^2}{2},
\]
we obtain
\[
\delta^2=\frac{2y^2(y-K)}{(y-A+B)(y+A-B)},
\]
which is the claimed equation. Conversely, if $K\leq y<(a+b)^2$ satisfies
this equation, then
\[
y-\delta^2
=\frac{16y\Delta_{ab\sqrt y}^2}{y^2-(A-B)^2}>0.
\]
It therefore lies in $I_\infty$ and satisfies the limiting critical
equation. Uniqueness of the maximizer proves uniqueness on the stated
right-angled or obtuse branch.

Under the matching condition, write
\[
p=w^2,\qquad q=(w')^2,\qquad
\theta=\frac{u}{w^2}=\frac{u'}{(w')^2},\qquad S=p+q.
\]
Then $A=p+\theta^2p^2$, $B=q+\theta^2q^2$, and $\delta=\theta S$.
For $y=S+\theta^2S^2$ we have
\[
y-A+B=2q(1+\theta^2S),\qquad
y+A-B=2p(1+\theta^2S),\qquad
y-K=2\theta^2pq.
\]
Substitution in the rational expression above gives $\theta^2S^2=\delta^2$.
Also $y=K+2uu'\geq K$ and
\[
(a+b)^2-y=2(ab-uu')>0.
\]
This is consequently the unique solution on the required branch, proving
\eqref{eq:cstar-explicit-matching}.

For the quantitative estimate, put $c_n^2=ny_n$. The assumed expansions
imply
\[
d_n-f_n=\frac{d_n^2-f_n^2}{d_n+f_n}
=-\delta\sqrt n+O(n^{-1/2}),
\qquad
d_n+f_n=2\sqrt{Q_n}+O(\sqrt n).
\]
In particular,
\[
\frac{(d_n-f_n)^2}{n}=\delta^2+O(n^{-1}),
\qquad
\frac{(d_n+f_n)^2}{n}\longrightarrow\infty.
\]
The scaled admissible interval is therefore nonempty for sufficiently large
$n$, and contains every compact subinterval of $I_\infty$ eventually.
The logarithmic derivative of the finite objective, after $x^2=ny$, is
\[
F_n(y)=\frac{a^2+b^2-y}{8\Delta_{ab\sqrt y}^2}
+\frac{n}{ny-(d_n-f_n)^2}
-\frac{n}{(d_n+f_n)^2-ny}-\frac1y.
\]
The same strict-decrease argument as above gives a unique maximizer and
$F_n(y_n)=0$. Uniformly on compact subintervals of $I_\infty$,
\[
F_n(y)=F_\infty(y)+O(n^{-1}),\qquad
F_\infty(y)=\frac{a^2+b^2-y}{8\Delta_{ab\sqrt y}^2}
+\frac1{y-\delta^2}-\frac1y.
\]
To see that $y_n\to c_{**}^2$, choose points in $I_\infty$ arbitrarily
close to and on either side of $c_{**}^2$. The strictly decreasing function
$F_\infty$ is positive at the left point and negative at the right point.
Uniform convergence preserves these signs for large $n$, trapping the
unique zero $y_n$ between them.

Since $F_\infty'(c_{**}^2)<0$, there are a neighborhood $U$ of $c_{**}^2$
and $\eta>0$ such that $|F_\infty'(y)|\geq\eta$ for $y\in U$.
For all sufficiently large $n$, $y_n\in U$ and
\[
0=F_n(y_n)=F_\infty(y_n)+O(n^{-1}).
\]
The mean-value theorem therefore yields
\[
y_n=c_{**}^2+O(n^{-1}).
\]
As $c_{**}>0$, taking square roots gives
\[
c_n=\sqrt{ny_n}=c_{**}\sqrt n+O(n^{-1/2}),
\]
as claimed.
\end{proof}

\newpage

\section{Acknowledgements}

This work was supported by the Department of Atomic Energy, Government of
India [project number RTI4014]; by the Infosys-Chandrasekharan virtual center
for Random Geometry at the Tata Institute of Fundamental Research (TIFR).
I used large language models as aids for writing code, checking calculations, proving subsidiary lemmas 
and improving exposition.  In particular, OpenAI's Codex was used in developing
the heuristic computation leading to the Barvinok--Hartigan comparison formula
in Subsubsection~\ref{sec:BH}.  All mathematical statements, proofs, numerical
claims, and references were checked by the author, who takes full responsibility
for the contents of the paper.

\end{document}